\documentclass[a4paper, english,11pt, reqno]{amsart}

\newcommand{\cA}{\mathcal{A}}

\newcommand{\cD}{\mathcal{D}}
\newcommand{\cE}{\mathcal{E}}

\newcommand{\cH}{\mathcal{H}}
\newcommand{\cI}{\mathcal{I}}
\newcommand{\cJ}{\mathcal{J}}
\newcommand{\cK}{\mathcal{K}}
\newcommand{\cL}{\mathcal{L}}
\newcommand{\cM}{\mathcal{M}}

\newcommand{\cO}{\mathcal{O}}

\newcommand{\cR}{\mathcal{R}}

\newcommand{\cX}{\mathcal{X}}
\newcommand{\cY}{\mathcal{Y}}

\usepackage{amsmath, amsfonts, dsfont,amssymb, amsthm, tikz-cd, enumerate, hyperref,mathtools, mathrsfs, geometry, verbatim, lilyglyphs,marginnote, romannum, todonotes}
\usepackage[dvipsnames]{xcolor}

\usepackage{graphicx}

\DeclareRobustCommand{\SkipTocEntry}[5]{}

\newcommand{\ddc}{\mathrm{dd^c}}

\newcommand{\K}{\mathcal K}
\newcommand{\uz}{\underline z}
\newcommand{\uzp}{\underline {z^\prime}}
\newcommand{\uq}{\underline q}
\newcommand{\ia}{\mathfrak{a}}
\newcommand{\iac}{{\overline{\mathfrak{a}}}}
\newcommand{\iuc}{{\overline{\iu}}}
\newcommand{\ivc}{{\overline{\iv}}}
\newcommand{\ib}{\mathfrak{b}}
\newcommand{\iu}{I}
\newcommand{\iv}{J}
\newcommand{\iw}{K}
\newcommand{\ic}{\mathfrak{c}}
\newcommand{\flag}{\mathcal F}
\newcommand{\C}{\mathds C}
\newcommand{\J}{\mathcal J}
\newcommand{\ideal}{\mathscr I}
\newcommand{\Xbc}{(X\times\mathds P^1)_{\C^*}^\beth}
\newcommand{\Xb}{X^\beth}
\newcommand{\dualc}{\varprojlim_\mathcal X\, \Delta_\mathcal X }

\newcommand{\Q}{\mathds Q}

\newcommand{\R}{\mathds R}

\newcommand{\D}{\mathds D}
\newcommand{\N}{\mathds N}
\newcommand{\OX}{\mathcal O_X}
\newcommand{\OXp}{\mathcal O_{X,p}}
\newcommand{\Pro}{\mathds P^1}

\DeclareMathOperator{\id}{id}

\DeclareMathOperator{\PL}{PL}
\DeclareMathOperator{\Pos}{Pos}
\DeclareMathOperator{\Val}{Val}

\DeclareMathOperator{\Ric}{Ric}
\DeclareMathOperator{\MA}{MA}
\DeclareMathOperator{\Scal}{Scal}
\DeclareMathOperator{\ind}{ind}

\DeclareMathOperator{\codim}{codim}
\DeclareMathOperator{\spec}{Spec}
\DeclareMathOperator{\tspec}{TropSpec}
\DeclareMathOperator{\trivial}{triv}
\DeclareMathOperator{\bl}{Bl}
\DeclareMathOperator{\BC}{BC}
\DeclareMathOperator{\supp}{supp}
\DeclareMathOperator{\Aff}{Aff}

\newcommand{\dualpl}{\varinjlim_\mathcal X \,\Aff_\mathbb Q\,\Delta_\mathcal X}

\newcommand{\coxp}{\mathbin{\cdot}\mathcal O_{X\times \Pro}}

\newcommand{\VCarb}{\varinjlim_\mathcal X \VCar(\mathcal{X})}
\newcommand{\oR}{{[0,+\infty]}}
\newcommand{\oRn}{{[-\infty, 0]}}

\DeclareMathOperator{\VCar}{VCar}

\DeclareMathOperator{\val}{{val}}
\DeclareMathOperator{\chern}{c}
\DeclareMathOperator{\NA}{{NA}}
\DeclareMathOperator{\PSH}{PSH}

\DeclareMathOperator{\Hdom}{\mathcal H}

\newcommand{\Eabs}{\cE^1_{\text{abs}}}

\DeclareMathOperator{\Nef}{Nef}
\DeclareMathOperator{\triv}{\mathcal X_{triv}}
\newcommand{\trive}{X\times \Pro}
\DeclareMathOperator{\divisorial}{{div}}
\DeclareMathOperator{\an}{an}

\newcommand{\Xdivc}{(X\times\Pro)^{\divisorial}_{\mathds C^*}}
\newcommand{\Ydivc}{(Y\times\Pro)^{\divisorial}_{\mathds C^*}}
\DeclareMathOperator{\ord}{ord}
\DeclareMathOperator{\hol}{hol}

\DeclareMathOperator{\Cz}{C^0}

\DeclareMathOperator{\AVCar}{AVCar}

\newcommand{\Xval}{X^{\val}}
\newcommand{\dprime}{{\prime\prime}}
\newcommand{\Xdiv}{X^{\divisorial}}
\newcommand{\XPdiv}{(X\times\Pro)^{\divisorial}}
\makeatletter
\newcommand{\newreptheorem}[2]{\newtheorem*{rep@#1}{\rep@title}\newenvironment{rep#1}[1]{\def\rep@title{#2 \ref*{##1}}\begin{rep@#1}}{\end{rep@#1}}}
\makeatother

\newreptheorem{theorem}{Theorem}
\newreptheorem{lemma}{Lemma}

\newtheorem{defi}{Definition}[subsection]
\newtheorem{remark}[defi]{Remark}
\newtheorem{example}[defi]{Example}
\newtheorem*{example*}{Example}
\newtheorem{lemma}[defi]{Lemma}
\newtheorem{prop}[defi]{Proposition}
\newtheorem{corollary}[defi]{Corollary}
\newtheorem{theorem}[defi]{Theorem}
\newtheorem{theorema}{Theorem}

\newtheorem*{statement}{Statement}
\newtheorem*{quest}{Question}
\newtheorem*{conjecture}{Conjecture}

\numberwithin{equation}{subsection}

\newcommand\blfootnote[1]{%
  \begingroup
  \renewcommand\thefootnote{}\footnote{#1}%
  \addtocounter{footnote}{-1}%
  \endgroup
}

\AtBeginDocument{\pagenumbering{arabic}}

\newcommand{\E}{\mathrm E}

\newcommand{\dtz}{\frac{\mathrm d}{\mathrm dt}\Bigr\rvert_{t = 0}}
\hypersetup{
    colorlinks=true,
    linkcolor=blue,
    filecolor=magenta,      
    urlcolor=cyan,
    pdftitle={A non-Archimedean theory of complex spaces and the cscK problem},
}
\begin{document}
\title{A non-Archimedean theory of complex spaces and the cscK problem}
\author[P. Mesquita-Piccione]{Pietro Mesquita-Piccione}
\address{Sorbonne Universit\'e and Universit\'e Paris Cit\'e\\
CNRS\\
IMJ-PRG\\
F-75005 Paris\\
France}
\email{\href{mailto:piccione@imj-prg.fr}{piccione@imj-prg.fr}}
\begin{abstract}
In this paper we develop an analogue of the Berkovich analytification for non-necessarily algebraic complex spaces. 
We apply this theory to generalize to arbitrary compact Kähler manifolds a result of Chi Li, \cite{Li22geodesic}, proving that a stronger version of K-stability implies the existence of a unique constant scalar curvature Kähler metric.
\end{abstract}
\maketitle
\setcounter{tocdepth}{1}
\tableofcontents

\section*{Introduction}\blfootnote{This project has received funding from the European Union’s Horizon 2020 research and innovation programme under the Marie Skłodowska-Curie grant agreement No 945332. \includegraphics*[scale = 0.03]{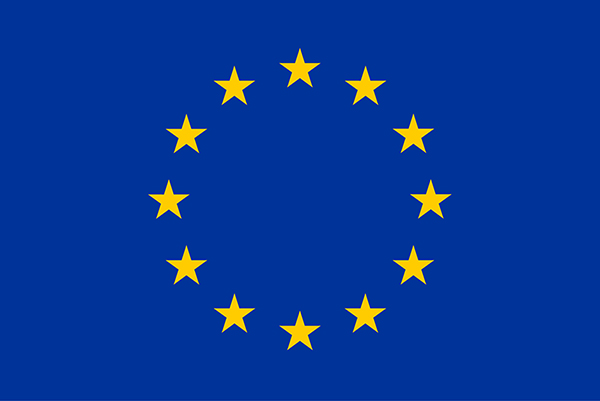}
}

\addtocontents{toc}{\SkipTocEntry}
\subsection*{The conjecture}

The study of algebraic and analytic properties of complex manifolds has been an active field of research in the area of Kähler Geometry. 

To approach many of the questions in this field, one requires generally results and techniques from both Differential and Algebraic Geometry, which shows the connection between those fields. 

One of the most important questions in the area is the Yau--Tian--Donaldson (YTD) conjecture. 
It deals with the following question: 

\begin{quest}
  Let $(X, \omega)$ be a Kähler manifold, and $\alpha\doteq [\omega]$ be the cohomology class of $\omega$.
Is there a Kähler metric $\omega^\prime\in \alpha$, in the same class as $\omega$, such that its scalar curvature 
\begin{equation}\label{csck}
\Scal(\omega^\prime) \equiv \underline{s}
\end{equation} is constant?
\end{quest}

When $X$ is a Fano variety, i.e. $X$ is anticanonically polarized, Chen--Donaldson--Sun \cite{CDS15kahler1,CDS15kahler2,CDS15kahler3} and Tian \cite{Tia15kstability} showed, using the continuity method, that $X$ admits a constant scalar curvature Kähler metric (cscK metric) if, and only if, $X$ is K-stable. 
Later other proofs became available, like \cite{BBJ21YTD}, \cite{Sze16partialc0,DS16KE}, \cite{CSW18KRflow},  and \cite{DZ24twisted} among others.

K-stability is an algebro-geometric notion that deals with the study of $1$-parameter degenerations of the manifold $X$, together with a numerical invariant. 
This notion is an infinite dimensional analogue of stability in the sense of Geomtric Invariant Theory (GIT). 
In that context, by the Kempf--Ness Theorem, studying the stability of an orbit of a hamiltonian group action is equivalent to finding zeros of an associated moment map. 

The general --still open-- version of the Yau--Tian--Donaldson conjecture reads:
\begin{conjecture}[YTD Conjecture]
  There exists a unique constant scalar curvature Kähler (cscK) metric in $\alpha$ if, and only if, $(X, \alpha)$ is uniformly K-stable in the sense of \cite{SD18,DR17kstability}.
\end{conjecture}

\addtocontents{toc}{\SkipTocEntry}
\subsection*{A variational approach}
The PDE of equation~\eqref{csck} has a variational interpretation that will play the key role in our approach to the YTD Conjecture.
This approach was developed by many authors, see \cite{BB17convexity}, \cite{BBEGZ19ke}, \cite{Che00space}, \cite{Dar15mabuchi}, and \cite{DR17tian} for some important work.
We describe this approach now.

By the $\partial\overline\partial$-lemma, if $\omega^\prime\in \alpha$ is a Kähler metric in the cohomology class of $\omega$, then $\omega^\prime - \omega = \mathrm{dd^c} u$, for $u \colon X\to \R$ a smooth function satisfying:
\[\omega + \mathrm{dd^c} u>0.\]
We denote $\omega^\prime$ by $\omega_u$, and we call $u$ a \emph{potential of $\omega^\prime$}.
In these terms, Equation~\eqref{csck} translates to a fourth order elliptic PDE on the potential.

The space of the potentials, in a given class $[\omega]$, is denoted by 
\[\mathcal H(\omega) \doteq \left\{u\colon X\to \R\bigm\vert u \text{ smooth, and } \omega+\mathrm{dd^c}u >0\right\}.\]

In order to have a good geometric setting to carry out the variational calculus for studying this PDE, we must consider a completion\footnote{for the Darvas metric $d_1$, see \cite{Dar15mabuchi}.} of $\mathcal H(\omega)$, the space $\cE^1(\omega)$.

By the groundbreaking work of Chen--Cheng \cite{CC21csck1,CC21csck2}, we know that there exists a unique solution for~\eqref{csck} if and only if a given functional, called the \emph{Mabuchi functional}, $\mathrm M_\omega\colon \cE^1(\omega)\to\R$, satisfies:
\begin{equation}
\mathrm M_\omega\ge \delta \mathrm J_\omega -C,
\end{equation}
for some $C,\delta>0$, and where $\mathrm J_\omega$ is a ``norm like" functional.
If it is the case, $\mathrm M_\omega$ is said to be \emph{coercive}.

In this setting the YTD conjecture can be stated as follows:
\begin{conjecture}[YTD Conjectrue]
  The Mabuchi functional is coercive if, and only if, $(X, \alpha)$ is \emph{uniformly K-stable} in the sense of \cite{SD18,DR17kstability}. 
\end{conjecture}

\begin{remark}
  When $X$ is a Kähler variety of klt singularities either $\Q$-Gorenstein smoothable, or admitting a resolution of singularities of nef relative anticanonical bundle,  by \cite{PTT24singularcsckmetrics} in the first case, and combining works of \cite{BJT24extremal, PT24singular} in the second case, we have that the coercivity of the Mabuchi functional implies the existence of a positive closed current in $\alpha$, that is a cscK metric on the regular locus, and has bounded potential.
  In this generality, the above conjecture is completely open.
\end{remark}

\subsubsection*{Varional approach in the Fano setting}
In the Fano case, Berman--Boucksom--Jonsson, in \cite{BBJ21YTD}, used the variational approach to prove the conjecture.
Their proof relied on two observations:
\begin{itemize}
  \item It was already known, in the Fano case, that the existence of a cscK metric is equivalent to the \emph{coercivity} of a different --and simpler-- functional defined on the space of potentials $\mathrm D\colon\cE^1(\omega)\to \R$, the \emph{Ding functional}.
  \item By \cite{BHJ17uniform}, there is a non-Archimedean description of K-stability: there is a non-Archimedean counterpart of $\cH(\omega)$ on the \emph{Berkovich analytification} of $X$, $X^{\an}$, denoted $\cH^{\NA}$, in  which we can define a non-Archimedean analogue of the Ding functional, $\mathrm D^{\NA}\colon \cH^{\NA}\to \R$, and  K-stability becomes equivalent to the coercivity of this functional.
\end{itemize}

Thus, using the simpler analysis of the Ding functional, Berman--Boucksom--Jonsson proved that the slope formulas for psh rays of \emph{algebraic singularities} guarantee the coercivity of the functional $\mathrm D$, as soon as one supposes the coercivity of the non-Archimedean counterpart $\mathrm D^{\NA}$.

In the general case the conjecture is still open, \cite{SD18} and \cite{DR17kstability} independetly proved --in the Kähler case-- one direction of the conjecture: the coercivity of the Mabuchi functional implies uniform K-stability.

The best result in the reciprocal direction is due Chi Li: in \cite{Li22geodesic}, he adapted the strategy of Berman--Boucksom--Jonsson to the Mabuchi functional, getting a weaker form of the open direction of the YTD conjecture for projective manifolds, i.e. he proved that stronger version of K-stability implies the coercivity of the Mabuchi functional.

The key ingredient in Chi Li's paper consists of proving that a \emph{distabilizing ray}, a ray of functions that contradicts the coercivity of the Mabuchi functional $\mathrm M_\omega$, must be a \emph{maximal geodesic ray}, that is, it must come from a \emph{non-Archimedean potential of finite energy}:
Like for $\cH(\omega)$, the set of potentials of finite energy $\cE^1(\omega)$ also has a non-Archimedean counterpart, $\cE^{1, \NA}$, and there is a correspondence between such potentials of finite energy and maximal geodesic rays, a disguished class of rays in $\cE^1(\omega)$. 
Having this in hand, he uses the finer slope formulas for maximal rays, and then the strategy of \cite{BBJ21YTD} follows in this general polarized case.

The goal of this paper is to generalize this result of Chi Li to the transcendental setting. 
We prove:
\begin{theorema}\label{thma;ytd}
  Let $(X,\alpha)$ be a compact Kähler manifold that is uniformly $K$-stable over $\cE^1$. 
  Then, $\alpha$ contains a unique cscK metric. 
\end{theorema}
  
\addtocontents{toc}{\SkipTocEntry}
\subsection*{General strategy and main results}

By the work of Sjöström Dyrefelt and Dervan--Ross, there exists a theory of K-stability for Kähler manifolds, but in order to adapt the strategy of Chi Li --or ultimately of BBJ--, one needs a non-Archimedean theory for a transcendental complex manifold, where the language of K-stability can be translated to.

We hence start by developing this non-Archimedean theory.
More precisely, we associate to a complex manifold, $X$, a ``non-Archimedean" compact Hausdorff topological space of \emph{semivaluations} on $X$, defined as the \emph{tropical spectrum} of the set of coherent ideals of $X$, which we denote by $X^\beth$.
This notion coincides with the Berkovich analytification of $X$ whenever the latter is a proper algebraic variety\footnote{We recall that if $X$ is a proper scheme over a trivially valued field, its Berkovich analytification $X^{\an}$ is defined as the set of all valuations $v\colon \C(Y)^*\to \R$, for $Y\subseteq X$ a subvariety, and $\C(Y)$ its function field.
When $X$ is a compact complex manifold, we may have that its function field is trivial, and moreover it may not have any non-trivial subvarieties. We cannot use this approach to define $X^{\beth}$.}, and moreover it also has a description as the limit of a Dual Complex, just like in algebraic setting.

The most important class of semivaluations is the one of the \emph{divisorial valuations}, which are given by
\[\OX\supseteq\iu\mapsto \ord_F(\iu\cdot\mathcal O_Y)\]
for $F\subseteq Y\to X$ a smooth irreducible reduced divisor on a normal model of $X$. 
The next main theorem we prove establishes the density of the set of divisorial valuations, $\Xdiv$, on $X^\beth$.
\begin{theorema}
\label{thma;dense}
$\Xdiv$ is dense in $X^\beth$.
\end{theorema}
The proof of Theorem~\ref{thma;dense} crucially requires a description of the divisorial valuations of $X$ as the $\C^*$-invariant divisorial valuations on $X\times\Pro$. 
We get this description by studying the divisorial valuations on ``local algebraic models" of $X$, that is, for each $p\in X$, we study the divisorial valuations of the scheme $X_p\doteq \spec \OXp$.

After establishing some first basic results, we develop a non-Archimedean pluripotential theory for $X^\beth$, key for the non-Archimedean approach of \cite{BBJ21YTD}.
We define non-Archimedean psh functions, and a mixed energy coupling, that allow us to:
\begin{enumerate}
  \item Use the synthetic pluripotential theory of \cite{BJ23synthetic}.
  \item Define non-Archimedean versions of classical functionals like: the \emph{Monge--Ampère energy}; the \emph{twisted energy}; the \emph{entropy}; and finally the J functional.  
\end{enumerate}
We then denote by $\cE^1(\alpha)$ the set of \emph{non-Archimedean potentials of finite Monge--Ampère energy}.

In the algebraic setting this theory coincides with the one of \cite{BJ22trivval}.
In the transcendental setting Darvas--Xia--Zhang  develop a non-Archimedean pluripotential theory for a big class, on \cite{DXZ23transcendental}, that coincides with ours on a Kähler class.
Their non-Archimedean pluripotential theory is not over a ``non-Archimedean space", and hence ours is a more direct analogue of the algebraic theory.
In particular, we can make sense of Monge--Ampère equations in our case, while it is not clear how to interpret them in their formalism.

Like in \cite{BBJ21YTD}, we have that the pluripotential theory developed here behaves well with the complex one:
to each psh ray of potentials $U_t$ we have associated a non-Archimedean psh function $U^\beth$.
Reciprocally, to each non-Archimedean psh function of finite energy $\varphi\in \cE^1(\alpha)$, there exists a psh ray $V_t$ such that $V^\beth = \varphi$.
Moreover, there exists a 1to1 correspondence between \emph{maximal geodesic rays} in $\cE^1(\omega)$ and non-Archimedean potentials in $\cE^1(\alpha)$.
This is the basis of the correspondence with the theory of \cite{DXZ23transcendental}.

For maximal geodesic rays we get formulas of the type:
\begin{equation}\label{eq;introconvergence1}
\lim_{s\to\infty }\frac{\mathrm F(U_s)}{s}= \mathrm F^\beth (U^\beth),
\end{equation}
for $\mathrm F$ either $\mathrm E$, $\mathrm E^\eta$, or $\mathrm J$, and an inequality for the entropy:
\begin{equation}\label{eq;introinequality}
\lim  \frac{\mathrm H_\omega (U_s)}{s} \ge \mathrm H_{\alpha}(U^\beth).
\end{equation}
Adding all together we get:
\[\lim_{s\to\infty} \frac{\mathrm M_{\omega}(U_s)}{s} \ge \mathrm M_{\alpha} (U^\beth),\]
where $\mathrm M_\alpha$ is the \emph{non-Archimedean Mabuchi functional}, the non-Archimedean counterpart of $\mathrm M_\omega$ providing the good inequality to conclude the proof of Theorem~\ref{thma;ytd}.

Indeed, just like in the projective setting,  if $\mathrm M_\omega$ is not coercive we can find a geodesic ray $U_t\in \cE^1$ such that $t\mapsto \mathrm M_\omega(U_t)$ is decreasing for $t$ large.
We call such a ray a \emph{destabilizing geodesic ray}.

Analogously to \cite{Li22geodesic}, every destabilizing ray is maximal, that is, it is a geodesic ray coming from a non-Archimedean potential of finite energy, $\varphi\in\cE^1(\alpha)$.
In particular:
\begin{equation}
\lim_{t\to\infty} \frac{\mathrm M_\omega(U_t)}{t}\ge \mathrm M_{\alpha}(\varphi).
\end{equation}

Like this, we prove that if $(X,\omega)$ is such that $\mathrm M_{\alpha}\colon \cE^1(\alpha)\to \R$ is positive, then there exists a unique cscK metric on $\alpha$,
proving Theorem~\ref{thma;ytd}.

\medskip
\addtocontents{toc}{\SkipTocEntry}
\subsection*{The YTD conjecture and Chi Li's result}
\label{sec;questions}
In the non-Archimedean terms of the present paper K-stability reads as the coercivity of the non-Archimedean Mabuchi functional:
\begin{defi}
The pair $(X, \alpha)$ is \emph{uniformly K-stable} if there exists a $\delta>0$ such that 
\[\mathrm M_{\alpha}(\varphi) \ge \delta \mathrm J_{\alpha}(\varphi)\]
for every $\varphi\in \Hdom(\alpha)$, that is if $\mathrm M_\alpha$ is coercive over $\cH(\alpha)$.
\end{defi}

In this paper, we have proved that a stronger version of K-stability, uniform K-stability over $\cE^1(\alpha)$, implies the coercivity of the K-energy, and hence the existence of a unique cscK metric.
In order to prove the conjecture one would need to prove that 
\[\mathrm M_{\alpha}|_{\Hdom(\alpha)} \ge \delta \mathrm J_{\alpha}|_{\Hdom(\alpha)}\implies \mathrm M_{\alpha}|_{\cE^1(\alpha)} \ge \delta^\prime \mathrm J_{\alpha}|_{\cE^1(\alpha)},\]
for some $\delta, \delta^\prime>0$.

In \cite{Li22geodesic}, Chi Li proves that, in the projective case, one has
\[\mathrm M_{\alpha}|_{\Cz\cap \PSH(\alpha)} \ge \delta \mathrm J_{\alpha}|_{\Cz\cap \PSH(\alpha)}\implies \mathrm M_{\alpha}|_{\cE^1(\alpha)} \ge \delta \mathrm J_{\alpha}|_{\cE^1(\alpha)},\]
obtaining that an intermediate version of K-stability implies the coercivity of the K-energy.

In this paper we not able to get a sharper result, like this one of Chi Li, in the transcendental setting. 
His result relies on the solution of the non-Archimedean Monge--Ampère equation of \cite{BFJ15solution} that we don't have at our disposal. 
\begin{remark}
  We would like to note that in an upcoming work of David Witt Nyström together with the author \cite{MPWN25nama}, they are able to establish the analogue of \cite{BFJ15solution} in the setting of the present paper. 
  They solve the non-Archimedean Monge–Ampère equation for non-algebraic Kähler manifolds and, in particular, prove that the intermediate notion of uniform K-stability of Chi Li implies uniform K-stability over $\cE^1$. This recovers Chi Li’s result in the projective setting.
  Moreover, this forthcoming work provides a new valuative criterion for K-stability over $\cE^1$.
\end{remark}

Furthermore, in the same paper Chi Li also gives a version of his theorem when $X$ admits automorphisms, which we don't treat in present paper.

\addtocontents{toc}{\SkipTocEntry}
\subsection*{Organization of the paper}

\begin{itemize}
\item In Section~\ref{sec;berkovich} we define our ``non-Archimedean analytification" of a locally ringed space, we do some basic constructions and we compare it to Berkovich analytification of a scheme.

\item Section~\ref{sec;Xbethanalytic} deals with the study of $X^\beth$, when $X$ is a normal compact complex space. We prove Theorem~\ref{thma;dense}, getting a theory which resembles closely the one of the Berkovich analytification of a projective variety over $\C$. 

\item Section~\ref{sec;duallog} contains a dual complex description of $X^\beth$, analogue to the one found in \cite{BJ22trivval}.

\item Section~\ref{sec;nonarchimedeanpluri} is devoted to developing the non-Archimedean pluripotential theory of $X^\beth$.

\item In Section~\ref{sec;complexandnon} we compare the pluripotential theory of $X^\beth$ with the transcendental non-Archimedean theory of \cite{DXZ23transcendental, xia24operations} , and get slope formulas, generalizing the results of \cite{SD18,DR17kstability, Li22geodesic}.

\item In Section~\ref{sec;csck} we define \emph{strong K-stability}, and prove Theorem~\ref{thma;ytd}, the main result of the paper.

\end{itemize}

\addtocontents{toc}{\SkipTocEntry}
\subsection*{Acknowledgements}
I would specially like to thank my PhD advisors Sébastien Boucksom and Tat Dat Tô for all their generous help, support and patience while discussing the present article for many months.

I would also like to thank Mingchen Xia for great conversations and advice, Ruadhaí Dervan, Chung-Ming Pan and Antonio Trusiani for spotting some typos in an earlier draft, and all my friends for their support.

\addtocontents{toc}{\SkipTocEntry}
\section*{Notations and Conventions}
Through out this paper, a \emph{ring} will always be unital and commutative.
By an \emph{analytic variety} we mean a reduced and irreducible complex analytic space.
Given $X$ an analytic variety we denote by $\widetilde X$ the normalization of $X$, cf. \cite[Chapter 6]{GR12coherent}. 
Moreover, we simply call \emph{ideal} a coherent ideal sheaf of $\OX$.
We call \emph{flag ideal} a $\C^*$-invariant fractional coherent ideal of $\mathcal O_{X\times\Pro}$, supported on $X\times\{0\}$, and we denote the set of such ideals by $\flag$.

A flag ideal $\ia\in \flag$ can be written as a sum:
\begin{equation}\label{eq;flagideal}
\ia = \sum_{\lambda\in \mathds Z} \ia_{\lambda} t^\lambda,
\end{equation}
for $(\ia_\lambda)_\lambda$ an increasing sequence of ideals of $X$ such that $\ia_\lambda = 0$ for $\lambda\ll 0$ and $\ia_\lambda  = \OX$ for $\lambda\gg 0$.

An ideal on $X$ will typically be denoted by $\iu,\iv$ or $\iw$.
An ideal on $X\times\Pro$ will typically be denoted either $\ia$ or $\ib$.

An ideal $\iu$ on $X$  is a \emph{prime ideal} if given $\iv, \iw$ ideals on $X$ satisfying: \[\iu\supseteq \iv\mathbin{\cdot}\iw,\] then either $\iu\supseteq\iv$, or $\iu\supseteq\iw$.
An ideal is prime if, and only if, it is a radical ideal and the underlying analytic subspace is irreducible.
That is, prime ideals of $X$ are exactly the ideals attached to analytic subvarieties.

Let $X$ be an analytic variety, $\iu$ an ideal, and $q\in \Q_{>0}$, a map $g\colon X\to \left[-\infty, +\infty\right[$ has \emph{singularities of type $\iu^{q}$} if locally:
\[ g(z) = q\log\sum_{i=1}^k\lvert f_i\rvert(z) + O(1),\]
for $f_1, \dotsc, f_k$ local generators of $\iu$.

If $X$ is a Kähler manifold, we denote by $\K(X)$ the set of Kähler forms on $X$, by $\Pos(X)\subseteq H^{1,1}(X)$ the set of Kähler classes, and $\Nef(X)$ the set of nef classes, i.e. the closure $\overline{\Pos(X)}\subseteq H^{1,1}(X)$.

For $\omega\in \cK(X)$, we denote by $\Ric(\omega)$ its \emph{Ricci form}. 
The trace of $\Ric(\omega)$ is denoted by $\Scal(\omega)$, the \emph{scalar curvature}, and the cohomological quantity:
\[n\cdot\frac{[\Ric(\omega)]\cdot[\omega]^{n-1}}{[\omega]^{n}},\] 
the average of the scalar curvature, by $\underline s$.
Moreover, if $\alpha$ denotes the cohomology class of a Kähler form $\omega$ we denote by $V_\alpha$ or $V_\omega$ the volume $\int_X \omega^n>0.$

For $x, y\in \R$, we write $x\lesssim y$ if there exists a uniform constant $C_n>0$, depending only on an integer $n$ given in the setup, such that $x\le C_n y.$





\section{Berkovich spectra as tropical spectra}
\label{sec;berkovich}
\subsection{Berkovich spectrum}
Let $X$ be an affine scheme, that is $X=\spec R$ for some ring $R$.
 
If we consider $R$ as a normed ring, with the trivial norm (the norm that to $a\in R$ associates $\lVert a\rVert_{\trivial}$ that is $1$ if $a\neq 0$, and $0$ otherwise), then $R$ can be seen as a Banach ring.
In particular, we can associate to it a compact Hausdorff topological space, the \emph{Berkovich spectrum of $R$}, first defined in \cite{Berkovich90}.
\begin{defi}
The \emph{Berkovich spectrum} associated to $\left(R,\lVert \mathbin{\cdot}\rVert_{\trivial} \right)$, denoted by $\cM(R)$, is the set of bounded multiplicative semi-norms on $R$, i.e. \[N\in \cM(R)\iff N\le \lVert\cdot\rVert_{\trivial},\]
equipped with the Hausdorff topology of pointwise convergence.
\end{defi}
The topology of pointwise convergence is the induced subspace topology given by natural inclusion 
\begin{equation}\label{eq;product}
\cM(R)\subseteq \prod_{a\in R} [0,1],
\end{equation}
as an easy consequence of Tychonoff's theorem $\cM(R)$ is compact.

To a semi-norm $N\in \cM(R)$ we can attach a \emph{semivaluation} on $R$, by taking $-\log N$.
Thus, equivalently, we can consider the Berkovich spectrum of $R$ as the set of semivaluations of $R$ with values on $\oR$, which, following the notation of \cite{Thu07}, we will denote by $X^\beth$, that is:
\begin{equation}
X^\beth\doteq\left\{
\begin{array}{ccc}
         & \vline &  v(a\mathbin{\cdot} b) = v(a)+v(b) \\
        v\colon R\to {\oR} & \vline & v(a +b) \ge \min\{v(a), v(b)\} \\
        & \vline & v(1) =0\enspace \&\enspace v(0) = +\infty
\end{array}\right\}.
\end{equation}


The construction is functorial, that is, given $Y$ another affine scheme and a morphism $f\colon Y\to X$, we can associate a continuous map between $Y^\beth $ and $X^\beth$:
\begin{align*}
f^\beth\colon Y^\beth &\to X^\beth\\
v&\mapsto  f^\beth (v)\colon a\mapsto v\left(f^\#(a)\right),
\end{align*}
compatible with compositions.

The Berkovich spectrum comes with a natural class of continuous functions.
Given $a\in R$, we associate $\lvert a\rvert \colon \cM(R)\to \left [0, +\infty\right[$ by the formula $N\mapsto N(a)$, or equivalently:
\begin{align*}
\log\lvert a\rvert\colon X^\beth&\to \left[-\infty, 0\right]\\
v&\mapsto -v(a).
\end{align*}
The notation is so that $\exp(\log\lvert a\rvert) = \lvert a\rvert$.

There is an equivalent formulation of Berkovich spectrum of a trivially normed ring, namely the \emph{tropical spectrum} of the semi-ring of its ideals of finite type.

Now, we study such an object.
\subsection{Tropical spectrum}
\label{sec;tspec}
For more details on this section, see Appendix~\ref{apx;semi-rings}.
\begin{defi}
Let $(S, +, \mathbin{\cdot})$ be a semi-ring, and consider the semi-ring of the extended real line $(\left]-\infty, +\infty\right], \min, +)$, we define the
\emph{tropical spectrum of $S$} as the topological space given by the set of \emph{tropical characters}, i.e. the semi-ring morphisms from $S$ to $\left]-\infty, +\infty\right]$:
\begin{equation}
	\tspec S  \doteq\left\{
\begin{array}{ccc}
         & \vline &\chi(a\mathbin{\cdot} b) = \chi(a)+\chi(b) \\
         \chi\colon S\to {\left]-\infty, +\infty\right]} & \vline   &\chi(0)= +\infty \\
         & \vline  &\chi(a +b) = \min\{\chi(a), \chi(b)\} 
\end{array}\right\},
\end{equation}
endowed with the pointwise convergence topology.
Moreover, if $S$ has a multiplicative identity, we ask $\chi(1) =0$. 
Equivalently, this is the same as asking that $\chi$ is not identically $+\infty$.
\end{defi}
As before, with this topology $\tspec S$ is compact and Hausdorff.

\smallskip
 
The tropical spectrum comes with a natural order relation, the usual order of functions.
Moreover, the $\R_{>0}$-action given by the usual scalar multiplication is compatible with the order structre. 
Furthermore, the $\R_{>0}$ action on $\tspec S$, induces an action on $\Cz(\tspec S,\R)$, for $t\in\R_{>0}$ we define:
\begin{equation}\label{eq;actiononfunctions}
	\left(t\mathbin{\cdot}\varphi\right)(v) = t\varphi(t^{-1} v).
\end{equation}
 
 \begin{remark}
 For any semi-ring $S$, $\tspec S$ is non-empty. 
 Indeed, we define $\chi_{\trivial}\colon S\to \oR$ by the formula:
 \[\chi_{\trivial}(a) = 0, \quad \text{ for every } a\in S\setminus\{0\}.\]
It is easy to see that $\chi_{\trivial}$ is a semivaluation, and, moreover, it is a fixed point of the $\R_{>0}$-action.
 \end{remark}

\subsubsection{Comparison with the Berkovich spectrum}
Now, let $R$ be a ring, and $\ideal(R)$ be the set of ideals of finite type of $R$.
Together with the usual operations of sum and product $\ideal(R)$ can be seen as a semi-ring. 
Moreover, as an easy consequence of the algebraic structure of $\ideal(R)$, its tropical characters are all positive, that is:
\[	\tspec \ideal(R) = \left\{
\begin{array}{ccc}
         & \vline &  \chi(\iu\mathbin{\cdot} \iv) = \chi(\iu)+\chi(\iv) \\
        \chi\colon \ideal(R)\to {\oR} & \vline & \chi(\iu +\iv) = \min\{\chi(\iu), \chi(\iv)\} \\
        & \vline & \chi(R) =0 \enspace \& \enspace \chi(0) = +\infty 
\end{array}\right\}_,\]
see Appendix~\ref{apx;semi-rings} for more details.
\medskip

There is a natural continuous map:
 \[\tspec \ideal (R)\to \left(\spec R\right)^\beth,\] 
 that assigns to each tropical character $\chi\in\tspec\ideal(R)$ the semivaluation: \[R\ni a\mapsto \chi(a\mathbin{\cdot} R).\]

\begin{prop}
The natural map, $\tspec \ideal(R) \to \left(\spec R\right)^\beth$, is a homeomorphism.
\end{prop}
\begin{proof}
Since $\tspec \ideal(R)$ is compact it is enough to check that the map is bijective.

The desired inverse function, $\left(\spec R\right)^\beth\to \tspec \ideal(R)$, is the one that assigns to $v\in \left(\spec R\right)^\beth$  the character:
\[\ideal(R)\ni\iu\mapsto \min_{f\in \iu} v(f),\]
where the minimum is achieved on any (finite) set of generators.
\end{proof}

With this ``tropical" characterization of the Berkovich spectrum of a ring, we will extend this construction to locally ringed spaces.

\subsection{Semivaluations on locally ringed spaces}
Let $(X,\OX)$ be a locally ringed topological space, and denote by $\ideal_X$ the set of $\OX$-ideals locally of finite type.
The set $\ideal_X$ has a semi-ring structure given by the usual addition and multiplication of ideal sheaves, and we define:
\begin{defi}
Let $X$ be a locally ringed space, we define the space of \emph{semivaluations on $X$} as the tropical spectrum of $\ideal_X$\footnote{Again by Appendix~\ref{apx;semi-rings} we have that all the tropical characters are positive.}:
\begin{equation}
X^\beth\doteq\tspec \ideal_X = \left\{
\begin{array}{ccc}
         & \vline &  v(\iu\mathbin{\cdot} \iv) = v(\iu)+v(\iv) \\
        v\colon \ideal_X\to {\oR} & \vline & v(\iu +\iv) = \min\{v(\iu), v(\iv)\} \\
        & \vline & v(\OX) =0 \enspace \& \enspace v(0_X) = +\infty 
\end{array}\right\},
\end{equation}
Endowed with the usual spectrum topology: the pointwise convergence topology.
\end{defi}
Some geometrically relevant examples are as follows:
\begin{example}

\begin{enumerate}
\item If $X$ is a scheme locally of finite type over a field $k$, equiped with the trivial absolute value, $X^\beth$ is a subset of its Berkovich analytification, $X^{\an}$, the set of semivaluations centered on $X$. 
This construction goes back to \cite{Berkovich90} and \cite{Thu07}.
Whenever the scheme is proper, by the valuative criterion of properness $X^\beth = X^{\an}$.
\item If $X$ is a complex analytic space, $X^\beth$ is an analogue of the Berkovich analytification of algebraic varieties over $\C$.
The study of $X^\beth$ will be done on Section~\ref{sec;Xbethanalytic}, and will be the central object of study of the present paper. 
\item If $X$ is a Berkovich space we can also associate to it an ``analytified" $X^\beth$.
\end{enumerate}
\end{example}

Here again the construction is functorial:
a morphism  ${f\colon (X,\OX)\to (Y,\mathcal O_Y)}$, induces a mapping:
\[f^*\colon \ideal_Y\to \ideal_X\]
which in turn induces a continuous map:
\[f^\beth\colon X^\beth \to Y^\beth.\]

\begin{remark}

If $X$ is a proper algebraic variety over $\C$, the GAGA theorems --for the Berkovich and complex analytifications-- allow us to compare $(X^{\an})^\beth$, $(X_{\hol})^\beth$ and $X^\beth$, where $X_{\hol}$ denotes the usual complex analytification.

Indeed, the theorems provide us with morphisms of ringed spaces between $X_{\hol}$, $X$ and $X^{\an}$, which induce a 1to1 correspondence on the set of coherent ideal sheaves.
Therefore we have canonical homeomorphisms: \[(X^{\an})^\beth\simeq X^\beth\simeq (X_{\hol})^\beth.\]
Moreover, as explained before, for a proper $\C$-scheme $X^{an}$ coincides with $X^\beth$.
\end{remark}

\subsubsection{Properties of the functor $\beth$}

Let $X$ be a locally ringed space, and $\iu\subseteq \OX$ be an ideal locally of finite type.
Consider $Y\doteq \supp\left(\OX/\iu\right)\subseteq X$, together with the sheaf $\mathcal O_Y \doteq \left(\OX/\iu\right)|_Y$
We thus have that the inclusion
\[i\colon \left(Y, \mathcal O_Y\right)\hookrightarrow \left(X,\OX\right)\]
is a ringed space morphism, and moreover we have the following lemma.

\begin{prop}
Let $Y\subseteq X$ as above,  then 
\[i^\beth\colon Y^\beth\to X^\beth\]
is an embedding.
\end{prop}
\begin{proof}
  Since $Y^\beth$ is compact it is enough to prove that $i^\beth$ is injective.

By definition, the morphism
\[i^*\colon \OX\twoheadrightarrow\mathcal O_Y\]
is surjective.

Therefore, if $v,u\in Y^\beth$, with $v\ne u$, then 
 there exists an ideal locally of finite type $\iv\in\ideal_Y$ such that $v(\iv)\neq u(\iv)$.
Moreover, $\iw\doteq (i^*)^{-1}(\iv)\subseteq \OX$ is an ideal locally of finite type, and  \[i^\beth v(\iw) = v(\iv)\neq u(\iv) = i^\beth u(\iw).\]
\end{proof}

\begin{remark}
We have just seen that the $\beth$ functor preserves embeddings, however it does not preserve open mappings.

In point of fact, if $U\subseteq X$ then $U^\beth$ will be a compact subset of $X^\beth$, even if $U$ is open on $X$.
\end{remark}

\subsubsection{Examples of semivaluations in $X^\beth$}

\begin{example}\label{ex;oxp}
Let $p\in X$, we denote by $X_p$ the affine scheme $\spec \OXp$.

We have a natural continuous map from the set of local semivaluations at $p$ to the set of global semivaluations on $X$, that to each $v\in (X_p)^\beth$  assigns the valuation:
 \[\ideal_X\ni\iu\mapsto v_p(\iu)\doteq v(\iu_p),\]
where $\iu_p$ denotes the stalk of $\iu$ at $p$.

\end{example}

\begin{remark}
If we suppose that the local ring $\OXp$ is noetherian (or at least its maximal ideal of finite type), the map \[v\mapsto v_p\] is injective on the set of semivaluations \emph{centered at $p$}, that is the set of semivaluations such that $v(m_p)>0$.

Indeed, if $v, v^\prime$, centered at $p$, are such that there exists an ideal $\iv\subseteq \OXp$ with \[v(\iv)\ne v^\prime(\iv),\]
then, for every $k\in \N$ consider the ideal 
\[\iv_k\doteq \iv+m_p^k\subseteq \OXp.\]
Since $\iv_k$ is primary at $p$ it extends to a global ideal, by triviality at any other point.
Hence 
\begin{align*}
v_p(\iv_k) &= v(\iv_k)=\min \{v(\iv), k\cdot v(m_p)\}\\
v_p^\prime(\iv_k) &= v^\prime(\iv_k)=\min \{v^\prime(\iv), k\cdot v^\prime(m_p)\},
\end{align*}
for every $k$.
 Letting $k\to +\infty$, it follows that
\[\lim_{k\to\infty}  v_p(\iv_k) = v(\iv), \quad \lim_{k\to\infty}  v_p^\prime(\iv_k) = v^\prime(\iv).\]
We thus have $v_p(\iv_k)\neq v_p^\prime(\iv_k)$ for $k\gg 1$, and hence  $v_p\neq v_p^\prime$.
\end{remark}

When $X$ is an analytic surface and $p\in X$ a regular point, $(X_p)^\beth$ was deeply studied in \cite{FJ04valuativetree}.

\subsection{PL functions}\label{section;pl}
Just like in the affine case, the space $X^\beth$ comes with a natural class of functions, which we will call the \emph{piecewise linear functions}.
The terminology will become clear in Section~\ref{sec;dual}.
\begin{defi}
Let $\iu\in\ideal_X$ be an ideal locally of finite type, we have a function $\log\lvert\iu\rvert\colon X^\beth \to \oRn$ that maps:
\[X^\beth\ni v\mapsto \log\lvert\iu\rvert\,(v) \doteq -v(\iu).\]


\end{defi}

Clearly, for every ideal $\iu$ the function $\log\lvert\iu\rvert$ is monotonic decreasing (with respect to the natural partial order on $X^\beth$).
Moreover:
\begin{equation}
\log\lvert\iu\mathbin{\cdot} \iv\rvert = \log\lvert\iu\rvert +\log\lvert\iv\rvert, \quad \log\lvert\iu+\iv\rvert = \max\{\log\lvert\iu\rvert, \log\lvert\iv\rvert\}.
\end{equation}

\begin{defi}
The set of functions $\psi\colon X^\beth \to \R$, of the form
\begin{equation}\label{eq;pl+}
v\mapsto \frac{1}{m}\max\left\{\log\lvert\iu_i\rvert (v) + k_i\right\}
\end{equation}
for  $\iu_1,\dotsc, \iu_N$ ideals locally of finite type, and integers $k_1,\dotsc, k_N\in\mathds Z$, $m\in \mathds N$,
is denoted $\PL^+\left(X^\beth\right)$, and the $\Q$-vector space it generates $\PL(X^\beth)\subseteq \Cz\left( X^\beth, \mathds R \right)$.

An element of $\PL(X^\beth)$ will be called a \emph{piecewise linear function}, we denote  $\PL_\R(X^\beth)\doteq \PL(X^\beth)\otimes_\Q \R$.
\end{defi}


\begin{lemma}
$\PL^+(X^\beth)$ separates points.
\end{lemma}
\begin{proof}
If $v,w\in X^\beth$ are distinct, then there exists $\iu\in \ideal_X$ such that $v(\iu)\neq w(\iu)$, with no loss of generality we can assume $w(\iu)<v(\iu)$, and thus taking \[\ell\in \left]\,w(\iu),\, v(\iu)\,\right[\,\bigcap\, \Q,\] and $\varphi \doteq \max\{\log\lvert \iu\rvert, -\ell\}\in\PL(X^\beth)$, we have  $\varphi(v) =-\ell<-w(\iu)=\varphi(w)$.
\end{proof}

\begin{prop}\label{prop:pldense}
$\PL(X^\beth)$ is dense in $\Cz(X^\beth,\R)$
\end{prop}
\begin{proof}
Since PL is a $\Q$-linear subspace of $\Cz$ stable by $\max$, containing the ($\Q$-) constants, and separating points, the result follows from the lattice version of the Stone-Weierstrass Theorem.
\end{proof}

\section{Semivaluations on a complex space $X$}
\label{sec;Xbethanalytic}
In this section we study $X^\beth$ attached to a compact analytic variety $X$.
From now on $X$ will always denote such a space.

For a complete reference on the more standard algebraic setting see \cite{BJ22trivval}, most of the results proved here, are direct analogues of results found there.

We recall the defintion of $X^\beth$:
\[
X^\beth= \left\{
\begin{array}{ccc}
         & \vline &  v(\iu\mathbin{\cdot} \iv) = v(\iu)+v(\iv) \\
        v\colon \ideal_X\to {\oR} & \vline & v(\iu +\iv) = \min\{v(\iu), v(\iv)\} \\
        & \vline & v(\OX) =0 \enspace \& \enspace v(0_X) = +\infty 
\end{array}\right\},
\]
where $\ideal_X$ denotes the set of ideals locally of finite type, that for a complex space coincides with the set coherent ideals by Oka's theorem.

We call an element, $v\in X^\beth$, a \emph{semivaluation}.
For $D\subseteq X$ an effective divisor, and $v\in X^\beth$ we set: 
\[v(D) \doteq v\left(\OX(-D)\right).\]

\subsection{Support and center of a semivaluation}

\begin{lemma}\label{lem;center}
Let $v\in X^\beth$ be a semivaluation, then there exist unique coherent ideals $\iu_s(v),\iu_c(v)$ that satisfy
\begin{equation}
v(\iu_s) = \infty, \quad \text{ and }\quad v(\iu_c)>0, \label{eq:s1}
\end{equation}
and are maximal with this property.
Moreover,  $\iu_s(v)$ and $\iu_c(v)$ are prime ideals.
\end{lemma}

\begin{proof}
Let $S\subseteq\ideal_X$ be the set of ideals on which $v$ is infinite, $ C$ the set on which $v$ is positive, and take:
\[\iu_s\doteq\sum_{\iv\in S} \iv,\quad \text{ and }\quad \iu_c\doteq \sum_{\iv\in C}\iv.\]
By the strong noertherian property, $\iu_s$ and $\iu_c$ are coherent, and by construction satisfy \eqref{eq:s1}.

For primality, if $\iv $ an $\iw$ are ideals on $X$, such that $\iv\mathbin{\cdot} \iw \subseteq \iu_c(v)$, then $0<v(\iu_c)\leq v(\iv\mathbin{\cdot} \iw)= v(\iv) + v(\iw)\implies $ either $0<v(\iv)$ or $0<v(\iw)$, and hence either $\iv\subseteq \iu_c$ or $\iw\subseteq \iu_c$, which implies that $\iu_c$ is prime. 
We proceed similarly for $\iu_s$.
\end{proof}

\begin{defi}
Let $v\in X^\beth$, we denote by $S_X(v)= S(v)$, the \emph{support of $v$}, the subvariety of $X$ given by $\iu_s$.
In the same fashion, we denote by $Z_X(v) = Z(v)$, the \emph{central variety of $v$}, the subvariety attached to $\iu_c$.
\end{defi}

\begin{remark}\label{rem;centerp}
Since $\iu_s\subseteq \iu_c$ it follows that $Z(v)\subseteq S(v)$.

Moreover, for any $p\in X$ the global ideal $m_p$ is maximal, and thus 
\[Z_X(v) = \{p\} \iff v(m_p)>0, \quad \text{ and } \quad S_X(v) =\{p\} \iff v(m_p)= \infty.
\]
\end{remark}
It is easy to see that the support and the center are well-behaved under inclusions.

If $v$ is finite valued on nonzero ideals, i.e. $S_X(v) = X$, then we will say that $v$ is \emph{valuation} on $X$.
We will denote the set of valuations by $\Xval$.
The support thus give us the following decomposition
\begin{equation}
X^\beth = \bigsqcup_{Y\subseteq X} Y^{\val} 
\end{equation}
where $Y$ ranges over all the subvarieties of $X$.

We can think the set of valuations centered at $\{p\}$ in terms of the Example~\ref{ex;oxp}:
\begin{example}\label{ex;centeredval}
The assignment \[(X_p)^\beth\ni v\mapsto v_p\in X^\beth\] induces a bijection from the set of local semivaluations centered at $p$ to the set of global semivaluations of $X$ centered at $p$.

Indeed, by Remark~\ref{rem;centerp} it is clear that the mapping sends semivaluations centered at $p$ to semivaluations centered at $p$. 

On the other hand, if $Z_X(v)=\{p\}$, $f\in \OXp$, and $\iu_f$ the local ideal generated by $f$.
We denote \[\iu_k \doteq \iu_f+m_p^k\] and  observe:
\begin{enumerate}
\item $\iu_k$ can be extended to a global ideal, just by triviality outside $p$.
\item $\iu_{k+1} = \iu_k +m_p^{k+1}\subseteq \iu_k$.
\end{enumerate}
We define $\nu(f)$ as the decreasing limit: \[\nu(f) \doteq \lim_{k\to\infty} v(\iu_k) =\lim_{k\to\infty} \min\{v(\iu_{k-1}), k\mathbin{\cdot} v(m_p)\}\in [0,+\infty].\]
It is easy to see that $\nu$ is a valuation and that  $\nu_p = v$.

\end{example}

In the next section we will use the above example to reconstruct $X^\beth$ from $(X_p)^\beth$ for every $p$, whenever $X$ a smooth analytic curve. 

\begin{remark}
A locally ringed space $(X, \OX)$ satisfies the \emph{strong noetherian property} if locally every increasing chain of $\OX$-ideals, locally of finite type, is locally stationary.

This is the property needed to defined the center and support of a semivaluation, since Lemma~\ref{lem;center} holds in this case.
\end{remark}

\begin{prop}\label{prop:bijval}
If $\pi\colon Y\to X$ is a bimeromorphic morphism,  the map
\[\pi^\beth\colon Y^\beth \to X^\beth\] induces a bijection $Y^{\val} \simeq X^{\val}$.
\end{prop}
\begin{proof}
We first observe that $\pi^\beth$ maps valuations to valuations:
 if $v\in Y^\beth$ is finite valued on the set of non-zero ideals of $Y$, then we need to check that, for a non-zero ideal $\iu$ of $X$, $\pi^{-1}(\iu)$ is not zero.
Let $U\subseteq X$ be an open set such that $\iu$ is not zero, then $\pi^{-1}(U)$ is an open set bimeromorphic to $U$, and 
\[\pi^{-1}\iu\, (U) = \{0\} \iff \pi(U)\subseteq Z_\iu,\]
where $Z_\iu$ is the zero locus of $\iu$, that has strictly positive codimension on $U$, therefore $\pi^{-1}\iu\ne 0$.
Thus $\pi^\beth (v)(\iu)<+\infty$, for $\iu$ a non-zero ideal on $X$.

Let's prove the that $\pi^\beth$ induces the desired bijection.

First suppose that there exists an anti-effective divisor $E\subseteq Y$ that is $\pi$-ample.
Then for every ideal of $Y$, $\iv\in \ideal_Y$, choosing $m_\iv\in\mathds N$ sufficiently large, $\iv\mathbin{\cdot} \mathcal O_Y(E)^{m_\iv}$ is $\pi$-globally generated, that implies
\begin{equation}
	\pi^{-1}\iu_\iv\cdot \cO_Y = \iv\mathbin{\cdot} \mathcal O_Y(E)^{m_\iv},
\end{equation}
for some ideal of $X$, $\iu_\iv\in\ideal_X$.
Since $\mathcal O_Y(E)$ is also $\pi$-globally generated we can also find $\iw\in \ideal_X$ such that
\[\pi^{-1} \iw\cdot \cO_Y = \mathcal O_Y(E).\]

Hence, given $v_X\in\Xval$ the function: 
\begin{equation}
\ideal_Y\ni \iv\mapsto v_X\left (\iu_\iv\right ) - m_\iv v_X\left(\iw\right)\footnote{Observe that we can take the difference because $v_X$ is finite valued.}
\end{equation}
is a valuation on $Y$, which we will denote by $v_Y$, such that $\pi^\beth(v_Y) = v_X$.

If, for $v, w\in Y^{\val}$, $\pi^\beth(v) = \pi^\beth(w)$, then for every $\iv$ ideal on $Y$:
\begin{align*}
	v(\iv) &=v\left(\iv\mathbin{\cdot} \mathcal O_Y(E)^{m_\iv}\right) - m_\iv v\left(\mathcal O_Y(E)\right)\\
	&= v(\pi^{-1}\iu_\iv\cdot \cO_Y) - m_\iv v(\pi^{-1} \iw\cdot \cO_Y) \\
	&=  \pi^\beth(v)(\iu_\iv) - m_\iv\pi^\beth(v)(\iw)\\
	&= \pi^\beth(w)(\iu_\iv) - m_\iv\pi^\beth(w)(\iw)\\
	&= w(\pi^{-1}\iu_\iv\cdot \cO_Y) - m_\iv w(\pi^{-1} \iw\cdot \cO_Y) \\
	&=w\left(\iv\mathbin{\cdot} \mathcal O_Y(E)^{m_\iv}\right) - m_\iv w\left(\mathcal O_Y(E)\right)= w(\iv),
\end{align*}
which implies that $v=w$.

If $Y$ does not admit a divisor $E$ as above,  by Hironaka, we can find $Y^\prime$ a smooth complex analytic space:
\begin{center}
\begin{tikzcd}
Y^\prime \arrow[dd, "\mu"] \arrow[rd, "\nu"] &                     \\
                                             & Y \arrow[ld, "\pi"] \\
X                                            &                    
\end{tikzcd}
\end{center}
such that $\mu$ is a sequence of blow-ups of smooth center, and $\nu$ is a bimeromorphic morphism.

Thus taking $E_\mu$ the exceptional divisor of $\mu$, we have --by the Negativity Lemma of \cite[Lemma~3.39]{KM98birational}-- that $E_\mu$ will be anti-effective, and relatively ample, giving  the bijection
\begin{equation}
\begin{tikzcd}
(Y^\prime)^{\val} \arrow[rd, "\nu^\beth"] \arrow[dd, "\mu^\beth", two heads, hook] &                                  \\
                                                                                                    & Y^{\val} \arrow[ld, "\pi^\beth"] \\
\Xval                                                                                               &                                 
\end{tikzcd}
\end{equation}
which implies that $\pi^\beth|_{Y^{\val}}$ is surjective on $X^{\val}$, and that $\nu^\beth|_{(Y^\prime)^{\val}}$ is injective.
In turn, the same argument gives the surjection \[\nu^\beth\colon (Y^\prime)^{\val}\twoheadrightarrow Y^{\val}.\] 
Since $\mu^\beth$ and $\nu^\beth$ induce bijections on the set of valuations, so does $\pi^\beth$.
\end{proof}

\subsection{Integral closure of an ideal and PL functions}
In this section we can be in a slightly more general setting of either complex analytic spaces or excellent schemes of equi-characteristic $0$.
The two examples, we have in mind are complex analytic spaces, and the scheme $\spec \mathcal O_{\C^n, 0}$.

For simplicity of the exposure from here on $X$ will be of normal singularities, see \cite{GR12coherent} for a complete reference on normal complex spaces.

\begin{defi}
Let $\iu\in\ideal_X$, we consider $\iuc$ the \emph{integral closure of $\iu$} to be the ideal given locally by all the elements $f\in \OX$ that satisfy a polynomial equation 
\begin{equation}
f^d = \sum_{i=0}^{d-1} a_i f^{i}
\end{equation}
for $a_i\in \iu$.
\end{defi}
It turns out that for complex analytic spaces, or excellent schemes as above, the ideal $\iuc$ is a coherent ideal, and hence $\iuc\in \ideal_X$.

This follows indeed from the following geometric description of the integral closure of an ideal, that will be useful later, given by the follwoing results:
\begin{prop}\label{blowup}
Let $\iu\in \ideal_X$, $\nu\colon Y \to X$ the normalized blow-up of $X$ along $\iu$, and $E\subseteq Y$ the exceptional divisor. 
We then have that $\nu_*(\mathcal O_{Y}(-E)) = \iuc$.
\end{prop}
\begin{proof}
See \cite[Proposition~9.6.6]{lazarsfeld2017positivity2} 
\end{proof}

\begin{corollary}\label{uniqueness}
Let $\mu\colon Y\to X$ be a normal projective modification of $X$, $D\subset Y$ an effective divisor such that 
\begin{equation}\label{cor1}
\mathcal O_Y (-D) = \iu\mathbin{\cdot} \mathcal O_Y
\end{equation}
for some $\iu$  ideal of $\mathcal O_{X}$. Then  $\mu_*(\mathcal O_Y (-D)) = \iuc$.
\end{corollary}
\begin{proof}
By Equation~\eqref{cor1}  $\mu$ factors through the normalized blow-up of $X$ along $\iu$, \(\widetilde{\operatorname{Bl}_\iu X}\to X\), and the result follows from Proposition~\ref{blowup}.
\end{proof}

\begin{lemma}\label{closure}
If $\iu$ is a coherent ideal then the associated function satisfies
\begin{equation}
\log\lvert \iu\rvert = \log\lvert \iuc\rvert
\end{equation}
on $X^{\val}$.
\end{lemma}
\begin{proof} 
Let $k\in\mathds N$ such that $\iu \mathbin{\cdot}\overline{\iu^k} = \overline \iu\mathbin{\cdot}\overline{\iu^k}$ \footnote{See \cite[Remark 8.7]{Bou18l2}},  
\begin{align*}
\log\lvert \iu\rvert + \log\lvert{\overline{\iu^k}}\rvert =\log\lvert\iu \mathbin{\cdot}\overline{\iu^k}\rvert&= \log\lvert\overline \iu\mathbin{\cdot}\overline{\iu^k}\rvert=\log\lvert \iuc\rvert +\log\lvert\overline{\iu^k}\rvert\\
\implies \log\lvert \iu\rvert &=\log\lvert\iuc\rvert.
\end{align*}
\end{proof}
The converse also holds, the valuative criterion of integral closedness gives us that if $\log\lvert \iu\rvert = \log\lvert\iv\rvert$, then $\iuc = \ivc$.
Later, in Section~\ref{section;pldiv}, we will see a more general statement that will imply it.

\subsection{Divisorial and monomial valuations}\label{sec;divisorialvaluations}
The trivial character described before, will be denoted $v_{\trivial}$ on $X^\beth$, and called the \emph{trivial valuation}: 
\[v_{\trivial} (\iu) =0, \quad \text{for every non-zero ideal } \iu.\]

More interestingly, to each irreducible divisor $F$, of a normal analytic variety $ Y\overset{\mu}{\to} X$ bimeromorphic to $X$, we can associate a valuation on $Y$ --and hence on $X$ by Proposition~\ref{prop:bijval}-- given by: 
\begin{align*}
\ideal_Y\ni \iu\mapsto \ord_{F}(\iu),
\end{align*}
where $\ord_F$ is given by the following procedure: choose any point $q\in F$, consider $\ord_{F_q}\in (Y_q)^\beth$ the order of vanishing along the germ of $F$ at $q$, and denote by $\ord_F\in Y^{\val}$ the induced global valuation of Example~\ref{ex;oxp}:
\[
\ord_F\doteq (\ord_{F_q})_q.
\]
It is classical that this construction does not depend on $q\in F$.
More details will be given in the discussion below on monomial valuations, see Proposition~\ref{prop;independencep}.

Equivalently, $\ord_F(\iu)= k$ if and only if we can find a decomposition \[\iu = \mathcal O_Y(-kF) \mathbin{\cdot} \iv\] where $F\nsubseteq Z_\iv$, the zero set of $\iv$.

\begin{defi}
We denote by $\Xdiv$ the set of valuations of the form: \[r\mathbin{\cdot} \ord_F\colon \ideal_X\setminus\{0_X\}\to \Q,\]  for a rational number $r\in\mathds Q_{\ge 0}$ and $F$ a divisor as above. 

We say that an element of $\Xdiv$ is a \emph{divisorial valuation}.
In particular, the trivial valuation $v_{\trivial}$ is a divisorial valuation. 
\end{defi}

\begin{remark}\label{rem:divbirational}
  Divisorial valuations are a birational invariant: if $f\colon Y\to X$ is a bimeromorphic morphism, then $f^\beth$ maps divisorial valuations to divisorial valuations, moreover the restriction \[f^\beth|_{Y^{\divisorial}}\colon Y^{\divisorial}\to \Xdiv\] is bijective.
\end{remark}

The study of divisorial valuations will be of central importance for the following,
as they encode a lot of the geometry of $X$.

For an analytic curve, $X^\beth$ can be completely described in terms of divisorial valuations, as we can see in the next example:

\begin{example}
Let $X$ be a smooth analytic curve, and $v\in X^\beth$ a semivaluation on $X$.
The central variety of $v$, $Z_X(v)$, is an irreducible subvariety of $X$, therefore either $ Z(v) =X$ or $Z(v) =\{p\}$ a point on $X$. 
Let's study the two options:
\begin{enumerate}
\item If $Z(v) = X$ then $v = v_{\trivial}$.

\item If $Z(v) =\{p\}$, then \[v = t\mathbin{\cdot} \ord_p\] for $t\doteq v(m_p)$.

Indeed, $v$ being centered at $p$ implies, by Example~\ref{ex;centeredval}, that $v =  w_p$ for some local semivaluation, $w$, centered at $p$.
In addition, in dimension $1$, the ring $\OXp$ is a discrete valuation ring and thus  $w= t\mathbin{\cdot} \ord_p$.
Hence we get $v=  (t\mathbin{\cdot}\ord_p)_p = t\mathbin{\cdot} \ord_p$.
We observe that if $t = +\infty$, we denote by $t\cdot \ord_p$ the semivaluation given by 
\[\ideal_X\ni \iu \mapsto \lim_{s\to +\infty }s\cdot \ord_p(\iu),\]
that is, is the only semivaluation that maps ideals co-supported on $p$ to $+\infty$ and ideals not co-supported on $p$ to $0$.
\end{enumerate}
\end{example}

Before further studying the divisorial valuations, we will study a slightly more general class of examples of points on $X^\beth$, the so-called \emph{monomial valuations}.
\subsubsection{Monomial valuations and cone complexes}\label{sec:monomialcone}
Let $(Y,B)$ be a \emph{snc pair over $X$}, i.e. a smooth bimeromorphic model of $X$ together with $B=\sum_{i\in I} B_i$ a reduced simple normal crossing (snc) divisor.
Define, for each subset $J\subseteq I$, the intersection $B_J\doteq \bigcap_{j\in J}B_j$.  
The connected components of each $B_J$ are called the \emph{strata} of $B$, and we usually denote a stratum of $B$, by the letter $Z$.

We can associate to the pair $(Y,B)$ a \emph{cone complex}, $\hat\Delta(Y,B)$, given by the following rule: for each $J\subseteq I$ , and each connected component, $Z$, of $B_J$, we associate $ \hat\sigma_Z\subset \mathbb R^ I$, a cone identified with $(\mathbb R_{\ge 0})^ J$. Moreover, if $Z_1\subseteq Z_2$ are connected components of $B_{J_1}$ and $B_{J_2}$ respectively, we identify the cone $\hat\sigma_{Z_2}$ with the subset of $\hat\sigma_{Z_1}$ that corresponds to the subset $(\R_{\ge 0})^{J_2}\subseteq (\R_{\ge 0})^{J_1}$.

We give now a procedure for assigning to each point \[w\in \hat\sigma_Z\subseteq \hat\Delta(Y,B)\] 
a valuation on $X$:

Take any point $p\in Z$, and let's suppose for convenience that $J=\{1,\dotsc, k\}$, we can construct a monomial valuation on $\OXp$ with respect to the germs $B_1, \dotsc, B_k$ at $p$ and weights $w_1, \dotsc, w_k$ following \cite{JM12val}.
This gives us a valuation $\val_p(w)\in\Val(\OXp) = (X_p)^{\val}$.
If we are given a coordinate open set $z\colon U\to \D^n$ around $p$, such that in $U$ the divisor $B_j$ is given by the local equation $z_j = 0$, for $j\in J$, then the valuation $\val_p(w)$ is given by:
\[\val_p(w)(f) = \min_{c_\alpha\ne 0} \langle \alpha, w\rangle,\]
for $f\in \OXp$, and $f= \sum_{\alpha\in \N^n} c_\alpha z^\alpha$, in a possibly smaller open set.
By some standard algebraic machinery, like in \cite{MN15weight, JM12val},  $\val_p(w)$ does not depend on the coordinates chosen, and only on the divisors.
For an analytic proof see Appendix~\ref{apx;monomial}.

Just like in Example~\ref{ex;oxp}, consider the valuation $ (\val_p(w))_p\in X^\beth$.
We will now argue that $(\val_p(w))_p$ does not depend on the point chosen $p\in Z$\footnote{Again, the analytic point of view of the monomial valuations explored in Appendix~\ref{apx;monomial}, give us the local independance of $p$, but for simplicity we will give here a more elementary approach.}, and therefore for $w\in \hat\sigma_Z\subseteq \hat\Delta(Y, B)$ we denote:
\begin{equation}\label{eq;monomial}
\val(w)\doteq (\val_p(w))_p.
\end{equation}

\begin{prop}\label{prop;independencep}
For any $w\in \hat\sigma_Z$ and $p\in Z$, the image in $X^{\val}$ of $\val_{p}(w)\in (X_p)^{\val}$ is independent of the choice of $p$.
\end{prop}
\begin{proof}
By connectedness of $Z$, it is enough to show that the statement holds locally near a given $p\in X$. 

Let $z\colon U\to \D^n$ be a local coordinate chart around $p$ such that:
\begin{itemize}
  \item $z(p)=0$;
  \item $B_J\cap U = Z\cap U$;
  \item \(B_j\cap U = (z_j =0),\) for every $j\in J$.
\end{itemize}
Let $q\in z^{-1}\left(\D(\frac{1}{3})\right)\cap Z$, and $f\in \OX(U)$, we'll prove that $\val_p(w)(f) = \val_q(w)(f)$.
Before going on, we introduce some notations that will be useful.
Denote by $z =(\uz_1,\uz_2)$ in such a way that $\uz_1\in \C^{k}$, and $\uz_2\in \C^{n-k}$. 

Using the coordinate system $z$ to identify $U$ and $\D^n$, writing $q = (\uq_1, \uq_2)$, we have that $\uq_1 =0$, and
if \[f(z) =\sum_{\alpha\in \N^n} c_\alpha z^\alpha\] is the expansion of $f$ at $0$, then: 
\begin{align*}
f(q +z^\prime) &= \sum c_\alpha (q + z^\prime)^\alpha =\sum c_{\alpha_1, \alpha_2}\, (\uq_1 +\uzp_1)^{\alpha_1}(\uq_2 +\uzp_2)^{\alpha_2}\\
&=\sum c_{\alpha_1, \alpha_2}\, \uzp_1^{\alpha_1}(\uq_2 +\uzp_2)^{\alpha_2}\\
&=\sum c_{\alpha_1, \alpha_2}\, \uzp_1^{\alpha_1}\left[\sum_{0\le j\le \alpha_2} \binom{\alpha_2}{j}\uq_2^j \uzp_2^{\alpha_2 -j}\right]\\
&=\sum\left[\sum_{j\in\mathds N^{n-k}} c_{\alpha_1, \alpha_2+j}\,\uq_2^j\binom{\alpha_2+j}{j} \right]\uzp_1^{\alpha_1}\uzp_2^{\alpha_2},
\end{align*}
for $\alpha = (\alpha_1, \alpha_2)\in \mathds N^{k}\times \mathds N^{n-k}$. 
Denoting $c^\prime_{\alpha_1, \alpha_2}\,\doteq \sum_j c_{\alpha_1, \alpha_2+j}\,\uq_2^j\binom{\alpha_2+j}{j}$, the exapansion of $f$ around $q$ becomes:
\[f(q+z^\prime) =\sum c^\prime_{\alpha_1, \alpha_2}\uzp_1^{\alpha^1}\uzp_2^{\alpha_2},\] 
hence we get:
$\val_q(w)(f) =\min_{c^\prime_{\alpha_1,\alpha_2}\ne 0}\langle w, \alpha_1\rangle,$
 and
\begin{equation}\label{eq;vwp}
\begin{aligned}
\val_p(w)(f) &=\min_{c_{\alpha_1,\alpha_2}\ne 0}\langle w, \alpha_1\rangle\\
&=\min\left\{ \langle w, \alpha_1\rangle \bigm\vert \forall\,\alpha_1\in \N^k \text{ such that } \exists\,\alpha_2\in \N^{n-k} \text{ with } c_{\alpha_1, \alpha_2}\ne 0 \right\}.
\end{aligned}
\end{equation}
Now, if $c^\prime_{\alpha_1, \beta}\ne 0$ for some $\alpha_1\in \N^k$ and $\beta\in\N^{n-k}$ , there exists $\alpha_2\in \N^{n-k}$ such that $c_{\alpha_1,\alpha_2}\ne 0$.
By Equation~\eqref{eq;vwp} we thus get that $\val_p(w)(f)\le \val_q(w)(f)$.

As long as $q$ is sufficiently close to $p$, close enough to find a coordinate chart around $q$ that contains $p$ with the property that the divisors $B_1, \dotsc, B_k$ are given by the equations $(z_1=0), \dotsc,(z_k=0)$,  we can exchange the role of $p$ and $q$, and get the equality:
\begin{equation}\label{eq;invarianceofpoint}
\val_p(w)(f)= \val_q(w)(f).
\end{equation}

Thus, if we let $\iu$ be a coherent ideal,  there exists an open set $U$ containing $p$, and local generators $f_1,\dotsc, f_\rho\in \iu(U)$ such that at any point $q\in U$, the germs $(f_1)_q, \dotsc, (f_\rho)_q$ generate $\iu_q$.
By Equation~\ref{eq;invarianceofpoint}, this implies that:
\[\val_p(w)(\iu_p) = \val_q(w)(\iu_q),\]
for $q\in U$.
Therefore, we have just proved that for every $\iu$ coherent ideal $(\val_p(w))_p(\iu)$ is locally independent of $p$, getting the desired result.

\end{proof}

The above construction of monomial valuations gives us an embedding:
\begin{align*}
\val\colon\hat\Delta(Y,B)&\hookrightarrow X^\beth.
\end{align*}
In fact, if $w\ne w^\prime$ both lie on $\hat\sigma_Z$, then there exists an irreducible component of $B$, $Z\subseteq B_i$, such that $\val(w)\left(B_i\right)\ne \val(w^\prime)\left(B_i\right)$.

An element of the image of the above mapping is called a \emph{monomial valuation}.

\begin{remark}
A monomial valuation associated to a rational point on $\hat\Delta(Y, B)$ is a divisorial valuation.
As in the algebraic setting, this can be seen using weighted blowups. 
\end{remark}

\subsubsection{Back to divisorial valuations}
As stated before, divisorial valuations will be of fundamental importance for our study of the geometry of $X$.
As an example, one is able to prove that $\Xdiv$ alone can tell homogeneous PL functions apart.
To see that, one proves that the homogeneous PL functions are essentially ($\Q$-Cartier) $b$-divisors over $X$, and the value on divisorial valuations will completely determine the associated $b$-divisor.
This idea will be further  developed in Section~\ref{section;pldiv}, when we'll prove that $\PL(X^\beth)$ is isomorphic to a set of $\C^*$-equivariant $\Q$-Cartier $b$-divisors over $X\times \Pro$.

Now we will focus our attention to obtain Theorem~\ref{thma;dense}.
\begin{theorem}[Theorem~\ref{thma;dense}]\label{teo;dense}
$\Xdiv$ is dense in $X^\beth$.
\end{theorem}

Since $\PL(X^\beth)$ is dense in $\Cz(X^\beth,\R)$,  $\PL(X^\beth)$ separates points from closed sets, and hence to prove Theorem~\ref{teo;dense}, it is enough to show that $\varphi ( \Xdiv) = \{0\}\implies\varphi = 0$, for every $\varphi \in \PL\left(X^\beth\right)$.

Therefore we can restate Theorem~\ref{teo;dense}, in the following way:
\begin{statement}
If $\varphi\in\PL(X^\beth)$ is such that \[\varphi|_{\Xdiv} =0\] then $\varphi = 0$.
\end{statement}


\begin{proof}[General idea of the proof of Theorem~\ref{teo;dense}]\label{proof;generaldense}

In order to prove Theorem~\ref{teo;dense}, it will be useful to write a PL function as the evaluation function of some ideal.

Let's assume that it is the case, and $\varphi\in\PL$ is attached to an ideal $\iu$, i.e. $\varphi=\log\lvert \iu\rvert$, then if $\log\lvert \iu\rvert (v) = 0$ for every $v\in \Xdiv$, and if $\iuc\neq \mathcal O_X$, it would be enough to take an irreducible component of the exceptional divisor $F\subseteq\widetilde{\operatorname{Bl}_\iu X}$, and consider $\ord_F\in \Xdiv$.
This would give us $0=\log\lvert \iu\rvert(\ord_F) =-\ord_{F}( \iu)\ne 0$.
\end{proof}
Just like in the algebraic trivially valued case of \cite{BJ22trivval}, this lead us to consider $\C^*$-equivariant models of $X\times \Pro$, since we will be able to see $\PL^+\left(X^\beth\right)$ as the evaluation functions attached to $\C^*$-equivariant ideals.

\subsection{$\C^*$-equivariant non-Archimedean space}
Recall that $\ia$ is a flag ideal of $X\times\Pro$ if it is $\C^*$-equivariant coherent ideal of $X\times \Pro$ whose support is contained in $X\times\{0\}$. 
In the rest of the section, we will denote by $t$ both the pullback of the usual affine coordinate on $\Pro$ around $[0:1]$, and the coherent ideal it generates on $X\times \Pro$.

Let $\flag$ be the set of fractional flag ideals of $X\times \Pro$, and consider the following subset of $(X\times\Pro)^\beth$:
\begin{equation}
(X\times\Pro)_{\C^*}^\beth\doteq\left\{\begin{array}{ccc}
         & \vline &  v(\ia\mathbin{\cdot} \ib) = v(\ia)+v(\ib) \\
        v\colon \flag\to \mathds R & \vline & v(\ia +\ib) = \min\{v(\ia), v(\ib)\} \\
        & \vline & v(t) =1 \quad \&\quad v(\mathcal O_{X\times \Pro}) =0
    \end{array}\right\}
\end{equation}

together with a continuous map, $\sigma$, the \emph{Gauss extension map} 
\begin{align*}
\sigma\colon X^\beth\to (X\times\mathds P^1)_{\C^*}^\beth\subseteq (X\times \mathds P^1)^\beth
\end{align*}
where $\sigma(v)\colon \flag\to\R$, the \emph{Gauss extension} of $v$, is given by
\begin{equation}
\sigma(v)(\ia) =\sigma(v)(\sum_{\lambda\in\mathds Z} a_\lambda t^\lambda) \doteq \min_\lambda\{v(a_\lambda)+ \lambda\}\in\R.
\end{equation}

\begin{remark}
As said before, we can see $\Xbc$ as a subset of $(X\times\Pro)^\beth$.
Given $v\in(X\times\Pro)_{\C^*}^\beth$, we can extend it to the set of all ideals $\ideal_{X\times\Pro}$, by setting \[v(\iu)\doteq \lim_{k\to\infty}v\left((t^k)+\sum_{\lambda\in\C^*} \lambda^* \iu\right)\]
for $\iu$ an ideal in $X\times\Pro$.
\end{remark}

\begin{lemma}\label{lem:gauss}
The Gauss extension, \[\sigma\colon X^\beth \to \Xbc,\] is a bijection.
\end{lemma}

\begin{proof}
In fact, taking $r\colon \Xbc\to \Xb$ to be the \emph{restriction map} \[r(v)\colon \ideal_X\ni \iu\mapsto  v\left(\iu\coxp\right)\] we get the inverse of $\sigma$.
It is clear that, as defined, $r(v)$ is a semivaluation on $X$.
So we are left to checking that for $\sigma(r(v)) = v\in \Xbc$, and $r(\sigma(v)) = v\in X^\beth$, which follows from:
\begin{align*}
	\sigma(r(v))(\ia) &= \min\left\{r(v)(\ia_\lambda) +\lambda  \right\}\\
	&= \min\left\{ v\left(\ia_\lambda\coxp\right) +\lambda \right\}\\
	&= \min\left\{ v\left(\ia_\lambda\coxp \mathbin{\cdot} (t^\lambda) \right) \right\}\\
	&= v\left(\sum_\lambda \ia_\lambda\coxp\mathbin{\cdot}(t^\lambda)\right) =v(\ia)
\end{align*}
and
\begin{align*}
r(\sigma(v))(\iu)&=\sigma(v)(\iu\coxp)\\
 &= \sigma(v)\left(\sum_{\lambda\ge 0}  \iu\mathbin{\cdot} t^\lambda\right) =\min\left\{ v(\iu)+\lambda \right\} = v(\iu).
\end{align*}
\end{proof}

For simplicity sometimes we identify $X^\beth$ and $\Xbc$.

The set of PL functions has a nicer description in $\Xbc$.
If we denote \[\varphi_\ia\doteq  \log\lvert \ia\rvert \circ \sigma\]
we get the following result:
\begin{prop}
The set $\{\varphi_\ia\bigm\vert \ia\in\flag\}$ $\Q$-generates $\PL(\Xb)$, the space of PL functions of $X^\beth$.
Moreover \[\PL^+(\Xb) = \left\lbrace \frac{1}{m}\varphi_\ia\bigm\vert\ia\in\flag, m\in \mathds N\right\rbrace.\]
\end{prop}
\begin{proof}
The proof follows form the description of flag ideals given in the equation~\eqref{eq;flagideal}.
\end{proof}

As we remarked at the end of Section~\ref{sec;divisorialvaluations}, this is a step in the right direction in order to prove Theorem~\ref{teo;dense}.
But then two problems rise up:
\begin{enumerate}
	\item The argument at the end of section~\ref{sec;divisorialvaluations} gives us a $\C^*$-equivariant divisorial valuation $v_E\in \Xbc$, but \emph{a priori} we don't know if $v_E$ comes from a divisorial valuation over $X$, i.e. the restriction of $v_E$, $r(v_E)$, lies in $\Xdiv$. 
  Hence, we don't get a contradiction.
	\item Even though the set  $\PL^+$ generates the PL functions, it is not enough to check that Theorem~\ref{teo;dense} is true for a $\PL^+$ function.
	What we need to check is that if $\varphi_1 (\ord_F) = \varphi_2 (\ord_F)$ for every $\ord_F\in \Xdiv$ then $\varphi_1 = \varphi_2$.  
\end{enumerate}

Section~\ref{sec;Cdivisorial} will deal with the first problem, and Section~\ref{section;pldiv} with the second one.

\subsection{$\C^*$-equivariant divisorial valuations}\label{sec;Cdivisorial}
\begin{defi}[Test configuration]
We define a \emph{test configuration for $X$} as the data of 
\begin{itemize}
  \item a normal compact Kähler space $\cX$;
  \item  a $\C^*$-action on $\cX$;
  \item a $\C^*$-equivariant flat morphism $\pi\colon \cX\to\Pro$;
  \item a $\C^*$-equivariant biholomorphism
    \[\cX\setminus \cX_0 \simeq X\times (\Pro\setminus\{0\}),\]
    where $\cX_0\doteq \pi^{-1}(0)$.
\end{itemize}

Moreover, if $\cX$, and $\cX^\prime$ are test configuration, $\cX$ \emph{dominates} $\cX^\prime$ if the bimeromorphic map
\[\cX\dashrightarrow X\times\Pro\dashrightarrow \cX^\prime\]
extends to a $\C^*$-equivariant holomorphic map, which we say is a \emph{morphism of test configurations}.
When $\cX$ dominates $X\times\Pro$, we say that $\cX$ is \emph{dominating}.
\end{defi}

\begin{remark}\label{rem:projcofinal}
Test configurations of a given normal analytic variety $X$ define a directed system. 
By the equivariant version of Hironaka's theorem, see for instance \cite[Section 13]{BM97canonical}, the set of test configurations that are projective over $X\times \Pro$, and of snc central fiber, is cofinal in all test configurations.

Unless otherwise stated, we will consider always such test configurations.
\end{remark}

An important class of examples of test configuration are given by flag ideals:
\begin{example}
Let $\ia$ be a flag ideal on $X\times \Pro$, then
\[\cX\doteq \widetilde{\bl_{\ia}(X\times \Pro})\] 
is a test configuration of $X$ that dominates $X\times \Pro$.

For more examples see \cite[Example 2.11]{DR17kstability}.
\end{example}

\begin{defi}[$\C^*$-invariant divisorial valuations]
Let $\cX$ be a test configuration of $X$ that dominates $X\times\Pro$, and $\cX_0 = \sum b_E E$ its decomposition into irreducible components.
We can associate to $E$ an element, $v_E$, of $\XPdiv\cap(X\times\Pro)^\beth_{\C^*}$ by the formula:
\begin{equation}
v_E(\ia) \doteq\frac{1}{b_E} \ord_E(\ia\mathbin{\cdot}\mathcal O_\cX).
\end{equation}
We denote by $\Xdivc\subseteq (X\times\Pro)^\beth_{\C^*}$ the set of such valuations given by $\E\subseteq \cX_0$, where $\cX$ is any test configuration, and  $E\subseteq \cX_0$ any irreducible component of the central fiber.
\end{defi}

Now, we focus our attention to $\Xdivc$, and prove the following theorem:
\begin{theorem}\label{thm;divisorialpointscoincide}
Let $r\colon (X\times\Pro)^\beth \to \Xb$ be the restriction map of Lemma~\ref{lem:gauss}, we then have \[r(\Xdivc) = \Xdiv.\]
\end{theorem} 
This is known in the algebraic case, see \cite[Theorem~4.6]{BHJ17uniform}. In that context, the proof relies on some valuative machinery that we don't have at our disposal here. 
More specifically, when $X$ is a proper algebraic variety over $\mathds C$, $X^\beth$ corresponds to the Berkovich analytification, whose points are (semi)valuations on the field of functions of (a subvariety of) $X$. 
Hence, we can associate to it invariants such as the rational rank, and the transcendental degree, that characterize the divisorial valuations completely.

To prove Theorem~\ref{thm;divisorialpointscoincide} in our transcendental setting we will reduce to the algebraic setting.
In order to do that we use Proposition~\ref{rem;divisorialvaluation}, which states that the order of vanishing along a smooth divisor can be computed locally. 
Hence it will be enough to consider ``germs" of manifolds and divisors, that is to consider at any $p\in X$, the scheme $X_p=\spec \cO_{X,p}$, and do local computations.

A useful tool in the following will be the following `GAGA'/base change theorem.

\begin{theorem}
Given $X$ an analytic variety, and $p\in X$ a point.
There is an equivalence between the category of projective $\OXp$--schemes, and that of analytic spaces which are projective over the germ of $X$ at $p$.
\end{theorem}
 
\begin{proof}[Sketch of correspondence]
Let's first build the functor on the objects.

Given a projective morphism $Y\xrightarrow{\pi} U$, $U\subseteq X$ an open set containing $p$,  there exists an embedding $Y\xhookrightarrow{j} U\times \mathds P^N_{\mathds C}$ such that 
\begin{center}
\begin{tikzcd}
Y \arrow[rr, "j", hook] \arrow[rd, "\pi"] &   & U\times \mathds P^N_{\mathds C} \arrow[ld, "\text{pr}_1"] \\
                                          & U &                                                          
\end{tikzcd}
\end{center}
commutes.
This means that we can find a finite number of homogeneous polynomials $f_1, \dotsc, f_k\in \mathcal O_X(V)[t_1, \dotsc, t_{N+1}]$ that cut-out $Y|_V\doteq \pi^{-1}(V)$, for some $V$ open neighborhood of $p$.
Taking the germ of the coefficient of $f_i$ at $p$ we get $f_1, \dotsc, f_k\in \OXp[t_1, \dotsc, t_{N+1}]$, which defines a subvariety $Y_p$ of $\mathds P^N_{\OXp}\cong \spec \OXp\times_{\spec \mathds C}\mathds P^{N}_{\mathds C}$ and hence we get a projective morphism $\pi_p\colon Y_p\to \spec\OXp$ given by the diagram 
\begin{center}
\begin{tikzcd}
Y_p \arrow[rd, "\pi_p"] \arrow[rr, "j"] &            & X_p\times_{\spec \mathds C}\mathds P^N_{\mathds C} \arrow[ld, "\text{pr}_1"'] \\
                                        & X_p &                                                                                     
\end{tikzcd}
\end{center}
where $j$ is the inclusion.

To analytify a projective morphism over $\spec \OXp$ the strategy is the same, see for instance \cite{jonsson_mustaţă_2014}.
It is clear how the correspondence of the objects induce a correspondence on the morphisms. 
For more details see \cite{bingener}.

\end{proof}
Apart from the usual correspondence of sheaves, one important property of this  `GAGA' theorem is the following dimension compatibility result:
\begin{prop}\label{prop;dimensionp}
Let $S\subseteq X$ be a Stein open set, and let $p\in S\subseteq X$, $Z$ a complex analytic space that is projective over $S$. 
Then, the dimension of the germ of $Z$ over $p$ seen as a $\OXp$-scheme, $Z_p$, is equal to the dimension of $Z$. 
\end{prop}
\begin{proof}
Apply \cite[Aussage~2.8]{bingener} for $K=p$.
\end{proof}

Given a projective morphism $\varphi\colon Y\to X$ we say the \emph{localization of $\varphi$ at $p$} is its isomorphism class on the category of the projective morphisms over the germ of $X$ at $p$.

\medskip

\subsubsection*{Birational Geometry intermezzo}
Before proving Theorem~\ref{thm;divisorialpointscoincide}, we recall some basic facts of Birational Geometry.
\begin{remark}\label{rem;centerdivisor}
Let $v\in \Xdiv$ be a divisorial valuation, $X^\prime$ a normal variety and $\mu\colon X^\prime\to X$ a bimeromorphic morphism such that $Z\doteq Z_{X^\prime}(v)\subseteq X^\prime$ is a (irreducible) divisor, then $\ord_Z = v$.

Indeed, if $F\subseteq Y\overset{\pi}{\to} X$ is chosen such that $v=\ord_F$, then it is enough to choose a normal bimeromorphic model $Y^\prime$ that dominates $Y$ and $X^\prime$, together with an irreducible divisor, $F^\prime \subseteq Y^\prime$, making the diagram,
\begin{equation}
\begin{tikzcd}
Z\subseteq  X^\prime \arrow[d, "\mu"] &  & F^\prime \subseteq Y^\prime \arrow[ll, "\nu"] \arrow[d,"\nu^\prime"] \\
X &  & F\subseteq Y, \arrow[ll, "\pi"]
\end{tikzcd}
\end{equation}
commute, where $F^\prime\doteq\nu^{-1}(Z)$ is the strict transform of $Z$ by $\nu$. 
Then  $\ord_{F^\prime} = \ord_Z$, and $\ord_{F^\prime} = \ord_F$, in particular $\ord_F = \ord_Z$. 
\end{remark}

The next Lemma is a version of Zariski's Lemma, found on \cite[Appendix: Prime Divisors, pg. 229]{Artin86}, that will be important for the following.
\begin{lemma}\label{lem:zariskiartin}
  Let $X$ be an integral scheme, and $v$ a divisorial valuation of $X$, then after blowing-up a finite number of times the center of $v$, $c(v)$, the latter will be the generic point of a divisor.\qed
\end{lemma}
For more details see \cite{Artin86}.

\medskip

\subsubsection*{Back at the discussion of Theorem~\ref{thm;divisorialpointscoincide}}

 Let's recollect the discussion on Section~\ref{sec;divisorialvaluations}, when $F\subseteq X$ is a prime smooth divisor on $X$, and $p$ a point in  $F$, then the valuation  $ \ord_F\in \Xdiv$ is given by the following procedure: 
\begin{enumerate}
\item Consider the valuation $\ord_{F_p}\in (X_p)^{\val}$ given by the germ of $F$ at $p$
\item Then define 
\[\ord_F(\iu)\doteq \ord_{F_q}(\iu_p)\] where $\iu_p$ denotes the germ of $\iu$ at $p$.
We saw that this definition does not depend on the point $p\in F$.
\end{enumerate}

More generally:
\begin{prop}
\label{rem;divisorialvaluation}
Let $G\subseteq Y\xrightarrow{\mu} X$ be a (prime smooth) divisor on a normal variety $Y$, and $\mu$ a projective bimeromorphic morphism,  consider $p\in Z_X(\ord_G) = \mu (G)$.
Then, localizing at $p$ we get a (prime smooth) divisor \[G_p\subseteq Y_p\xrightarrow{\mu_p} X_p\] and the associated divisorial valuation on $X_p^{\beth}$ satisfies 
\begin{equation}
\ord_G(\iu\mathbin{\cdot}\mathcal O_Y) = \ord_{G_p}(\iu_p\mathbin{\cdot} \mathcal O_{Y_p})
\end{equation}
for $\iu$ a coherent ideal of $X$, and $\iu_p$ the germ of $\iu$ at $p$.
\end{prop}
\begin{proof}
Let $\iu$ be an ideal on $X$ and  $k \doteq \ord_G(\iu)$, write:
\[\iu\mathbin{\cdot} \mathcal O_Y = \mathcal O_Y(-kG)\mathbin{\cdot} \iv,\] 
with $\iv$ an ideal such that $G\nsubseteq Z_{\iv}$.
Localizing at $p$ we get:
\[\iu_p\cdot\mathcal O_{Y_p} = \mathcal O_{Y_p}(-k G_p)\cdot \iv_p.\]
By primality of $G$ we have that $G_p$ is prime and smooth, in particular $G_p\nsubseteq Z_{\iv_p}$.
Getting \[\ord_{G_p}(\iu_p\mathbin{\cdot} \mathcal O_{Y_p}) = k.\]
\end{proof}

In this $\C^*$-equivariant setting we also have an analogue statement as of Remark~\ref{rem:divbirational}, that is of key importance for Theorem~\ref{thm;divisorialpointscoincide}.

\begin{prop}\label{prop:bijdivc}
  Let $f\colon Y\to X$ be a bimeromorphic morphism, then the morphism $F \doteq (f,\, \id)\colon Y\times\Pro \to X\times\Pro$ induces a bijection:
  \[F^\beth|_{\Ydivc}\colon \Ydivc\to \Xdivc.\]
\end{prop}

By Remark~\ref{rem:divbirational}, $F^\beth$ is a bijection between $(Y\times\Pro)^{\divisorial}$ and $(X\times\Pro)^{\divisorial}$, but since the valuations on $\Xdivc$ (or $\Ydivc$) are attached to divisors corresponding to ireducible components of $\C^*$-equivariant degenerations of $X$ (or $Y$ resp.),  we need to check that, for $v_E\in \Ydivc$, the divisorial valuation $F^\beth (v_E)$ can be obtained from an irreducible component of the central fiber of a test configuration of $X$.
\begin{proof}[Proof of Proposition~\ref{prop:bijdivc}]
  Let's start proving that $F^\beth$ maps $\Ydivc$ to $\Xdivc$.

  If $\cY$ is a test configuration for $Y$, dominating $Y\times\Pro$, and $E$ is a prime smooth vertical divisor, then we'll show that $F^\beth(v_{E})\in \Xdivc$. 
  That is, that there exists a test configuration $\cX$ for $X$, together with an irreducible vertical divisor $D$ and a $\C^*$-equivariant birational map $\mu_T\colon \cY\to \cX$:
  \[
  \begin{tikzcd}
  Y\times\Pro \arrow[d, "{(f, \,\id)}"'] & \cY \supseteq E \arrow[l] \arrow[ld, "\mu_1"'] \arrow[d, "\mu_T", dashed] \\
  X\times \Pro                           & \cX\supseteq D, \arrow[l]                                                 
  \end{tikzcd}\]
  such that $D$ and $E$ generate the same divisorial valuation on $X\times\Pro$.

  To prove this we first observe that $v\doteq F^\beth(v_E)$ is a valuation on $X\times\Pro$, and thus, denoting $X\times\Pro$ by $\cX^1$, the central variety $Z_1\doteq Z(v, \cX^1)$ is well defined and a  $\mathds C^*$-invariant irreducible set supported on the central fiber $X\times\{0\}$, given by the zeroes of $(\mu_1)_*\mathcal O_{\cY}(-E)$.
  Therefore, the blow-up $\cX^2\doteq\bl_{Z_1}\cX^1$ is a test-configuration for $X$, and the central variety, $Z_2$, of $v$ on $\cX^2$ is $\C^*$-invariant, and supported on the central fiber.

  Inductively, the blow-up $b_{k+1}\colon \cX^{k+1}\to \cX^k$ of $\cX^k$ along $Z_k$, is a test configuration and the center, $Z_{k+1}$, of $v$ in $\cX^{k+1}$ is $\C^*$-invariant and supoorted on central fiber:
\begin{center}
\begin{tikzcd}
E\subseteq \cY \arrow[d, "\mu_k", dashed] \arrow[rrd, "\mu_{k+1}", dashed] &  &                                                                      \\
Z_k\subseteq \cX^k                                            &  & \cX^{k+1}\supseteq Z_{k+1}= \overline{\mu_{k+1}(E)},\arrow[ll, "b_{k+1}"]
\end{tikzcd}
\end{center}
where $\mu_{k+1}$ is the bimeromorphic map defined by $\mu_k$ and $b_{k+1}$.

In the algebraic case, by a Lemma of Zariski  after blowing up the center of the divisorial valuation a finite number of times we get that $Z(v, \cX^k)$ is a divisor.
But in our non-algebraic context, Zariski’s result does not a priori apply.
The strategy of our proof will be to localize at a point $p\in Z_1\subseteq X\times\Pro$, use the version of Zariski's lemma given in Lemma~\ref{lem:zariskiartin}, that applies in this local case, to get that, for some $k\gg 0$ sufficiently big, $(Z_k)_p$ is a divisor. 
By Proposition~\ref{prop;dimensionp} this implies that $Z_k$ is a divisor, and thus by Remark~\ref{rem;centerdivisor} we are done. 

Taking $p\in Z_1\subseteq X\times\Pro$, and localizing at $p$ we get:
\begin{equation}
\begin{tikzcd}
E_p\subseteq \cY^p \arrow[d, "(\mu_1)_p"] &  &                                                        &                  &                                                                     \\
(Z_1)_p\subseteq (\cX^1)_p                &  & (Z_2)_p\subseteq(\cX^2)_p \arrow[ll, "(b_2)_p"] & \dotsb \arrow[l] & (Z_k)_p\subseteq (\cX^k)_p. \arrow[l] \arrow[llll, "b_p", bend left]
\end{tikzcd}
\end{equation}
The valuation $v_{E_p}$ is a divisorial valuation on the $\cO_{X,p}$-scheme $(\cX^i)_p$ whose (scheme theoretic) center is the generic point of $(Z_i)_p$. 
Since $(\cX^i)_p = \bl_{(Z_{i-1})_p}(\cX^{i-1})_p$, by Lemma~\ref{lem:zariskiartin} after a finite number of steps $(Z_k)_p$ becomes a divisor.
By irreducibility of $Z_k$ and Proposition~\ref{prop;dimensionp}, $Z_k$ is a --global-- divisor of $\cX^k$.

The map $F^\beth|_{\Ydivc}$ is injective by Proposition~\ref{prop:bijval}, and it is easy to see that it is surjective. 
\end{proof}


\medskip
On to the proof of Theorem~\ref{thm;divisorialpointscoincide}:

\begin{proof}[Proof of Theorem~\ref{thm;divisorialpointscoincide}]
  Let's start proving that $\Xdiv \subseteq r(\Xdivc)$, that is for every $F\subseteq Y\overset{\mu}{\to} X$, irreducible smooth divisor on a bimeromorphic model of $X$, 
  \[\ord_F \in r(\Xdivc).\] 
  
  Let $\cY$ be the deformation to the normal cone of $F\subseteq Y$, that is the blow-up of $F\times \{0\}$ in $Y\times \Pro$, with exceptional divisor $E\subseteq \cY\xrightarrow{\mu}Y\times \Pro$, the irreducible divisor corresponding to the blow-up of $Y$ along $F$, cf. \cite[Chapter 5]{Ful98intersection}.
  Localizing at $p\in F$:
  \[F_p\subseteq Y_p \quad \&\quad E_p\subseteq \cY_p =\bl_{F_p\times\{0\}}(Y_p\times_{\spec \mathds C} \Pro).\]
  Now, applying Proposition~\ref{rem;divisorialvaluation} we get that for any ideal $\iu$:
  \begin{equation}
  \ord_F(\iu)= \ord_{F_p}(\iu_p) = r(v_{E_p})(\iu_p)= v_{E_p}\left(\iu_p\mathbin{\cdot} \mathcal O_{\cY_p}\right)  = v_E(\iu\cdot\mathcal O_{\cY}),
  \end{equation}
  where the second equality is given by \cite[Theorem~4.8]{BHJ17uniform}\footnote{The set up there is for a scheme of finite type over a field of characteristic zero, but the same arguments apply.}, and thus \[\ord_F= r(v_E)\in r(\cY_{\C^*}).\]
  By Proposition~\ref{prop:bijdivc} we have stablished that $\Xdiv\subseteq r(\Xdivc)$.
 
 To prove that $\Xdiv\supseteq r(\Xdivc)$, the strategy will be the same.
 
Let $\cX$ be a test configuration that dominates $X\times \Pro$, and $E\subseteq \cX_0$ an irreducible component. 
We then have:
\begin{equation}\nonumber
\begin{tikzcd}
E\subseteq \cX \arrow[d] & E_1\subseteq \cY^1 \arrow[l, "f_1"] \arrow[dd, "\rho_1"] \\
X\times \Pro \arrow[d]     &                                                                 \\
Z_0\subseteq X             & Y_1\doteq\bl_{Z_0} X  \arrow[l, "b_1"]                         
\end{tikzcd}
\end{equation}
where $Z_0 \doteq Z(r(v_E), X)$ is the central variety of $r(v_E)$ on $X$, and $\cY^1$ is dominates $Y_1$ and $\cX$, and $f_1$ is a bimeromorphism such that the strict transform $E_1 = f_1^{-1} E$ is an irreducible smooth divisor.
We can define then $Z_1\doteq Z(r(v_E), Y_1)=\overline{\rho_1(E_1)}\subseteq Y_1$.

Localizing at a point $p\in Z_0$ we get:
\begin{equation}\nonumber
\begin{tikzcd}
E_p\subseteq \cX_p \arrow[d] & (E_1)_p\subseteq (\cY^1)_p \arrow[l, "{f_{1,p}}"] \arrow[dd, "{\rho_{1,p}}"] \\
(X\times \Pro)_p \arrow[d]     &                                                                                     \\
(Z_0)_p\subseteq X_p           & (Z_1)_p\subseteq(Y_1)_p \arrow[l, "{b_{1,p}}"]                                     
\end{tikzcd}
\end{equation}
and, as before, $E_p$ defines a divisorial valuation on $X_p$ --again by the same arguments as in \cite[Theorem 4.8]{BHJ17uniform}-- and its schematic center in $X_p$ is the generic point of $(Z_0)_p$, similarly the center of $v_{E_p}$ on $(Y_1)_p$ is the generic point $(Z_1)_p$. 
Repeating the construction we get:
\begin{equation}
\begin{tikzcd}
E\subseteq \cX \arrow[d] & E_1\subseteq \cY^1 \arrow[l, "f_1"] \arrow[dd, "\rho_1"] & \dotsb \arrow[l] & E_\ell\subseteq \cY^\ell \arrow[l] \arrow[dd]  & \dotsb \arrow[l] \\
X\times \Pro \arrow[d]     &                                                                 &                  &                                                       &                  \\
Z_0\subseteq X             & Z_1\subseteq Y_1 \arrow[l, "b_1"]                               & \mathbin{\cdot}s \arrow[l] & Z_\ell\subseteq \bl_{Z_{\ell-1}} Y_{\ell-1} \arrow[l] & \dotsb, \arrow[l]
\end{tikzcd}
\end{equation}
and 
\begin{equation}
\begin{tikzcd}
E_p\subseteq \cX_p \arrow[d] & (E_1)_p\subseteq (\cY^1)_p \arrow[l, "{f_{1,p}}"] \arrow[dd, "{\rho_{1,p}}"] & \dotsb \arrow[l] & (E_\ell)_p\subseteq (\cY^\ell)_p \arrow[l] \arrow[dd] & \dotsb \arrow[l] \\
(X\times \Pro)_p \arrow[d]     &                                                                                     &                  &                                                              &                  \\
(Z_0)_p\subseteq X_p           & (Z_1)_p\subseteq(Y_1)_p \arrow[l, "{b_{1,p}}"]                                      & \dotsb \arrow[l] & (Z_\ell)_p\subseteq (Y_\ell)_p. \arrow[l]                     & \dotsb. \arrow[l]
\end{tikzcd}
\end{equation}
Again by Lemma~\ref{lem:zariskiartin}, it exists $k\in \mathds N$ such that $(Z_k)_p \subseteq (Y_k)_p = \bl_{(Z_{k-1})_p}(Y_{k-1})_p$ is a prime divisor.
 Proposition~\ref{prop;dimensionp}, together with the irreducibility of $Z_k$, gives us that $Z_k\subseteq Y_k$ is a prime divisor.
 Moreover,  by construction, $Z_k$ is the central variety of $v_E$ on $Y_k$, and therefore by Remark~\ref{rem;centerdivisor}:  \[v_{E_p}=\ord_{(Z_k)_p}.\] 
 Thus
\begin{equation}
	v_E(\iu\mathbin{\cdot} \mathcal O_{\cX}) = v_{E_p}(\iu_p\mathbin{\cdot} \mathcal O_{\cX_p}) = \ord_{(Z_k)_p} (\iu_p\mathbin{\cdot} \mathcal O_{(Y_k)_p}) = \ord_{Z_k}(\iu\mathbin{\cdot} \mathcal O_{Y_k}),
\end{equation}
getting \[r(v_E) = \ord_{Z_k}\in \Xdiv.\] 
This shows that $r(\Xdivc)\subseteq \Xdiv$, completing the proof.

\end{proof}

\subsection{PL functions as divisors}\label{section;pldiv}

\begin{defi}
Let $\cX$ be a test configuration of $X$, we denote by $\VCar(\cX)$ the finite dimensional $\Q$-vector space given by the $\mathds C^*$-invariant $\mathds Q$-Cartier divisors supported on $\cX_0$.
An element of $\VCar(\cX)$ is called a \emph{vertical $\Q$-Cartier divisor on $\cX$}, or simply, \emph{vertical divisor of $\cX$}.
\end{defi}

Now, since morphisms of test configurations induce linear mappings between the vertical $\Q$-Cartier divisors, we define \emph{vertical Cartier b-divisors.}

\begin{defi}
Consider the direct system given by $\langle\VCar(\cX), \mu^*_{\cX, \cX^\prime}\rangle$, we then say that the elements of the directed limit:
 \[ \varinjlim_\cX \VCar(\mathcal{X}),\] 
are called \emph{vertical Cartier $b$-divisors}.
\end{defi}

For each test configuration $\cX$ we can define a natural map:
\[\PL^+(X^\beth)\to\VCar(\cX),\]
that assigns to each $\varphi\in \PL^+$ the vertical divisor given by:
\begin{equation}\label{eq;defdiv}
\sum_{E\overset{\text{irred}}{\subseteq}\cX_0} b_E\,\varphi(v_E) E, 
\end{equation}
where $\cX_0 = \sum b_E\, E$ is the irrreducible decomposition.
We show now that these maps glue well to define a universal one to the direct limit $\VCarb$.
\begin{lemma}\label{lem:plvcarb}
  The collection of the above mentioned maps induces the mapping:
\[\PL^+(X^\beth)\to\VCarb.\]
\end{lemma}
\begin{proof}
  Let $\varphi\in \PL^+(\Xb)$, we show that for a cofinal set of test configurations, and $\mu\colon \cX^\prime\to \cX$ a morphism of such test configurations: 
\[\mu^*\left(\sum_{E\overset{\text{irred}}{\subseteq}\cX_0} b_E\,\varphi(v_E) E\right) = \sum_{E^\prime\overset{\text{irred}}{\subseteq}\cX^\prime_0} b_{E^\prime}\,\varphi(v_{E^\prime}) E^\prime.\]

After scaling, we may assume that $\varphi = \varphi_\ia$ for some flag ideal $\ia$. 
The set of test configurations that dominate the normalized blow-up of $X\times\Pro$ along $\ia$ is cofinal\footnote{Observe that more than being cofinal, this set is also upward closed, that is, every test configuration that dominates an element of this cofinal set is also in the cofinal set.}, and thus we will suppose that $\cX$ is in this set.
Let $G$ be the effective divisor induced by $\ia$ on $\cX$:
\begin{equation}
\mathcal O_{\cX}(-G) = \ia\mathbin{\cdot}\mathcal O_{\cX}.
\end{equation}
Therefore,
 \[\sum_{E\subseteq\cX_0} b_E\,\varphi(v_E) E =-G.\] 

Writing $\cX^\prime_0 =\sum_{E^\prime\subseteq \cX_0^\prime} b_{E^\prime}\, E^\prime$ as the irreducible decomposition, it follows that:
\begin{align*}
\mu^*\left(\sum_{E\subseteq\cX_0} b_E\,\varphi(v_E) E \right)=\mu^*{(-G)}&= \sum_{E^\prime\subseteq \cX_0^\prime} \ord_{E^\prime}(\mathcal O(\mu^* G)) E^\prime\\
 &= \sum_{E^\prime\subseteq\cX_0^\prime} -\ord_{E^\prime}(\ia\mathbin{\cdot}\mathcal O_{\cX^\prime}) E^\prime \\
 &=\sum_{E^\prime\subseteq \cX_0^\prime} -b_{E^\prime}v_{E^\prime}(\ia)E^\prime\\
 &= \sum_{E^\prime\subseteq\cX^\prime_0} b_{E^\prime}\,\varphi(v_{E^\prime}) E^\prime,
\end{align*}
concluding the proof.
\end{proof}

\begin{theorem}
\label{vcarpl}
The above map induces an isomorphism of linear spaces 
\[
\PL(X^\beth) \simeq \varinjlim_\cX \VCar(\mathcal{X}).
\]
\end{theorem}
\begin{proof}
Since the map is additive, we have a unique linear extension:
\[\PL(X^\beth)\to\VCarb.\] 

We now construct its inverse.

Let $\cX$ be a test configuration of $X$. 
By Remark~\ref{rem:projcofinal} we can suppose that there exists a morphism of test configurations $\mu\colon\cX\to X\times \Pro$ that is projective, that is we can embbed $\cX$:
\[
\begin{tikzcd}
\cX \arrow[rr, hook] \arrow[rd, "\mu"'] &                    & (X\times\mathbb P^1)\times \mathbb P^\ell \arrow[ld, "p_1"] \arrow[d, "p_2"] \\
                                               & X\times\mathbb P^1 & \mathbb P^\ell                                                              
\end{tikzcd}\]
making the diagram commute, where $p_i$ is the $i$-th coordinate projection.
Then the set:
\[\AVCar(\cX)\doteq\{ D\in \VCar(\cX) | -D\text{ is }\mu\text{-very ample}\}\] 
is non-empty, since $p_2^*(mL)|_{\cX}\in \AVCar(\cX)$, for $L$ an ample line bundle on $\mathds P^\ell$ and for $m\gg 0$.
Moreover, $\AVCar(\cX)$ is a semigroup that $\Q$-spans $\VCar(\cX)$.
 
For each $D\in \AVCar(\cX)$,  $-D$ is $\mu$-globally generated, which implies that there exists $\ib$, a fractional ideal sheaf of $\mathcal O_{X\times\Pro }$, such that 
\begin{equation}\label{gg}
\mathcal O_{\cX}(-D) = \ib\mathbin{\cdot}\mathcal O_{\cX}.
\end{equation}  
Since $D$ is a vertical divisor implies that we can suppose $\ib\in \flag$, and hence we define $\varphi_D\doteq -\varphi_\ib$.
To see that $\varphi_D$ is well defined, is enough to observe that if 
\begin{equation}
\mathcal O_{\cX}(-D) = \ib^\prime\mathbin{\cdot}\mathcal O_{\cX}
\end{equation}
then by Corollary~\ref{uniqueness}  $\overline\ib = \overline{\ib^\prime}$, which implies \[\varphi_\ib = \varphi_{\overline\ib} = \varphi_{\overline{\ib^\prime}} = \varphi_{\ib^\prime}\]
by Lemma~\ref{closure}\footnote{The proof of Lemma~\ref{closure} applies, since for every flag ideal $\ia$ the function $\log \lvert \ia\rvert $ is finite valued on $\Xbc$.}.

Let us now check that $\varphi\colon \AVCar(\cX)\to \PL(X^\beth)$ is additive.
Pick $D,D^\prime\in \AVCar(\cX)$ and write
\[
\mathcal O_{\cX}(-D) = \ib\mathbin{\cdot}\mathcal O_{\cX}\quad \& \quad \mathcal O_{\cX}(-D^\prime) = \ic\mathbin{\cdot}\mathcal O_{\cX},
\]
which implies
\begin{align*}
\mathcal O_{\cX}(-D - D^\prime)=\mathcal O_{\cX}(-D) \cdot \mathcal O_{\cX}(-D^\prime) = (\ib\mathbin{\cdot}\ic)\mathbin{\cdot}\mathcal O_{\cX}
\end{align*}
and thus 
\[ \varphi_{D+D^\prime} = -\varphi_{\ib\mathbin{\cdot}\ic}= -\varphi_\ib - \varphi_\ic = \varphi_D + \varphi_{D^\prime}.
\]
Again we can extend $\varphi$ uniquely to $\VCar(\cX)$ by linearity.

Observe that this definition does not depend on $\cX$, in the sense that if $\mu\colon\cX^\prime\to \cX$  is a morphism of test configurations, and $D^\prime\doteq\mu^*D\subseteq\cX^\prime$, then we have
\begin{equation}
\mathcal O_{\cX^\prime}(-D^\prime) = (\ib\mathbin{\cdot}\mathcal O_{\cX})\mathbin{\cdot}\mathcal O_{\cX^\prime} = \ib\mathbin{\cdot}\mathcal O_{\cX^\prime}.
\end{equation}
Hence $\varphi_{D^\prime} =\varphi_D$.

This defines \footnote{The set of test configurations obtained by a sequence of blow-ups is cofinal.} a linear map \[\VCarb\to \PL(X^\beth)\]
which is the inverse of \eqref{eq;defdiv}.
\end{proof}

We thus can conclude the proof of Theorem~\ref{teo;dense}. 
\begin{proof}[Proof of Theorem~\ref{teo;dense}]
If $\varphi\in \PL(X^\beth)$ is such that \[\varphi(v)=0, \quad \text{ for every } v\in \Xdiv,\] 
then, using Theorem~\ref{thm;divisorialpointscoincide} we have that for all test configurations $\cX^\prime$ and all prime vertical divisors $E^\prime\in\VCar(\cX^\prime)$ \[\varphi(v_{E^\prime}) = 0.\]
In particular, if $D\in \VCar(\cX)$ is a vertical divisor, such that $\varphi_D = \varphi$, writing \[D = \sum_{E\overset{\text{irred}}{\subseteq}\cX_0} \varphi_D(v_E) E,\] we will have $D=0$, and thus $\varphi =0$.

\end{proof}
\section{Dual complexes and log discrepancy}
\label{sec;duallog}
From this point on $X$ will be a compact complex manifold.
\subsection{Non-Archimedean as a limit of tropical}
\label{sec;dual}
In the algebraic setting it is known that the Berkovich analytification corresponds to taking a limit of tropical complexes, known as the \emph{Dual Complex}, associated to test configurations. 
See \cite[Appendix A]{BJ22trivval} for a version in the trivially valued case.
In this section we will show the analogous result in our transcendental setting.


\subsubsection{Contruction of the dual complex}\label{sec:dualcomplex}
Let $\cX$ be a smooth snc test configuration for $X$.  
Let $\cX_0=\sum_i b_i E_i$ be the decomposition of the central fiber in its irreducible components.

Then $(\cX,\cX_{0,\text{red}})$ is a snc reduced birational model of $\cX_{\text{triv}}\doteq X\times \Pro$. 
Recall from Section~\ref{sec:monomialcone} that we can then construct $\hat\Delta_{\cX}\doteq \hat \Delta(\cX, \cX_{0,\text{red}})$.

Now, we will construct a simplicial complex, $\Delta_\cX$, as a sort of compact representative of $\hat\Delta_{\cX}$.
For each ``cone" face, $\hat\sigma_Z\cong (\mathbb R_{\ge 0})^ J$, of $\hat\Delta_{\cX}$ we will associate a ``simplex" face, $\sigma_Z$, of $\Delta_\cX$ given by the equation $\sum b_i\,w_i =1$, that is:
\[\sigma_Z\doteq \left\{ w\in \hat\sigma_Z \cong (\mathbb R_{\ge 0})^ J \,\bigm\vert\, \sum_{i\in J} b_i w_i =1\right\}.\]

Given a test configuration $\cX$, we have a natural map:
\begin{equation}
p_\cX\colon X^\beth \to \Delta_\cX
\end{equation} 
defined by $p_\cX(v) = \left( v(E_i) \right)\in \R_{\ge 0}^J$, where the latter that corresponds to the stratum $Z$, the smallest one that contains $Z(v, \cX)$ the central variety of $v$ on $\cX$.

\subsubsection{Morphisms}\label{sec;morphismsofcomplex}
Let $\cX, \cX^\prime$ be test configurations of $X$, $\mu\colon \cX\to\cX^ \prime$  a test configuration morphism between them, and $\sum_{i\in I} b_i E_i$, $\sum_{j\in J} c_j E^\prime_j$ be the decomposition in irreducible components of $\cX_0, \cX^ \prime_0$ respectively. 
Then, clearly
\[\text{Supp}(\cX_0)\subseteq \text{Supp}(\mu^*\cX^\prime_0).\] 
In particular, we can write $\mu^* E^\prime_j = \sum_i d^i_j E_i$, for $D_j = (d_j^ 1,\dotsc, d_j^ M)\in \R^{M}$, and we define the map:
\begin{align*}
  r_{\cX, \cX^\prime}\colon \Delta_\cX&\longrightarrow\Delta_{\cX^\prime}\\
  (\mathbb R_{\ge 0})^J\cong\sigma_Z\ni w&\mapsto r_{\cX, \cX^\prime}(w)\in \sigma_{Z^\prime}\cong (\mathbb R_{\ge 0})^{J^\prime_w},
\end{align*}
for \(J^\prime_w \doteq \{j\in J\bigm\vert d_j^ i\neq 0 \text{ for some  }i\in I\},\) given by:
\[r_{\cX, \cX^\prime}(w)\doteq \left( \sum d^i_j\,w_i\right)_{j\in J^\prime_w}.\]

Since the snc test configurations form a directed poset we can take the projective limit of topological spaces,
\[ \Delta\doteq \varprojlim_{\cX \text{ snc}}\, \Delta_\cX,\]
and the family of maps $\left( p_\cX\right)_\cX$ induces an injective\footnote{this is equivalent to $\PL(X^\beth)$ separating points on $X^\beth$.} continuous map:
\[p\colon X^\beth \to \Delta.\]
\begin{theorem}\label{thm;dualcomplex}
The map $p\colon X^\beth\to \Delta$ is a homeomorphism.
\end{theorem}
To get this, we'll see that, just like $X^\beth$, $\Delta$ has a PL structure, that will be isomorphic to the PL structure on $X^\beth$.

\subsubsection{PL Functions}
There is a natural class of functions defined on $\Delta$,  the $\ind$-type set of \emph{piecewise linear functions}. 
That is, the set of real valued functions that, on a complex $\Delta_\cX$, are $\mathds Q$-piecewise linear:
\[\PL(\Delta)\doteq \bigcup_{\cX} (\pi_\cX)^*\PL(\Delta_\cX),\]
for $\pi_\cX\colon \Delta\to \Delta_\cX$ the canonical projection.

After going to a higher model, we can assume that the functions are rationally affine on each face of the associated dual complex $\Delta_{\cX}$, and
hence we have: 
\[\PL(\Delta)= \varinjlim_{\cX}\, \Aff_{\mathds Q}(\Delta_{\cX}).\]

Now, observe that $\text{VCar}(\cX)_\mathbb Q\cong \text{Aff}_\mathbb Q\,\Delta_\cX$, where the isomorphism is given by:
\begin{equation}\label{eq;affvcar}
\text{Aff}_\mathbb Q\,\Delta_\cX\ni f\mapsto \sum b_i\, f(e_i)\,  E_i,
\end{equation}
where $\cX_0 = \sum_{i\in I} b_i \, E_i$ is the irreducible decomposition, and $e_i$ corresponds to the standard generator of $\R^i\subseteq \R^I$.
Taking the limit we get:
\begin{equation}\label{eq;plisomorphism}
\PL(\Delta) = \varinjlim_\cX \,\text{Aff}_\mathbb Q\,\Delta_\cX\cong \varinjlim_\cX \, \VCar(\cX)\cong \PL(X^\beth).
\end{equation} 

\begin{lemma}
The map 
\[p\colon X^\beth\to \Delta\] is an ismorphism of PL structures\footnote{See Appendix~\ref{apx;semi-rings}.}.
\end{lemma}
\begin{proof}
 We need to check that if $\eta\colon \PL(\Delta)\to \PL(X^\beth)$ is the ismorphism of Equation~\eqref{eq;plisomorphism}, then, for $f\in \PL(\Delta)$ and $v\in \Xb$:
\begin{equation}\label{eq;plaffinecommuting}
\eta(f)(v) = f\left(p(v)\right).
\end{equation}
By Theorem~\ref{teo;dense} it is enough to check~\eqref{eq;plaffinecommuting} for $v\in \Xdiv$.

Now, given $f\in \PL(\Delta)$, and $v\in \Xdiv$, let $\cX$ be a smooth test configuration such that:
\begin{itemize}
  \item the function $f|_{\Delta_\cX}$ is rationally affine, for $\Delta_\cX$ the associated dual complex;
  \item decomposing the central fiber $\cX_0 \doteq \sum_i b_iE_i$, we have $v=v_{E_1}$.
\end{itemize} 
Then \[f(p(v)) = f\left(v(E_1), v(E_2), \dotsc, v(E_k)\right) =   v(E_1)f(e_1)= f(e_1) = \eta(f)(v_{E_1})\]
where the last equality is given by~\eqref{eq;affvcar} together with~\eqref{eq;defdiv}.
\end{proof}

Now, we will prove that the isomorphism~\eqref{eq;plisomorphism} induces a homeomorphism \[X^\beth\overset{p}{\cong} \Delta.\]
To do that, as mentioned before, we will use an analogue of the Gelfand transform, to show that $X^\beth$ and $\Delta$ can be seen as the ``tropical spectra" of $\PL(X^\beth)$ and $\PL(\Delta)$ respectively.
Then, since the map \[p\colon X^\beth\to \Delta\] is an isomorphism of PL structures, $p$ will be a homeomorphism.
\subsubsection{Tropical Gelfand transform }
Let's recall some definitions from Section~\ref{sec;tspec}, and from Appendix~\ref{apx;semi-rings}.

Let $\mathcal A$ be a tropical algebra, that is a vector space together with a semi-ring operation, which we will denote by $\{\cdot,\cdot\}$, that makes $(\mathcal A, \{\cdot, \cdot\}, +)$ a semi-ring.
The \emph{tropical spectrum of $\mathcal A$} is the topological space given by\footnote{cf. Lemma~\ref{lem;tspecalgebra}}
\begin{equation}
	\tspec \mathcal A = \left\{\varphi\in \mathcal A^* |\, \varphi\left(\{f,g\}\right) = \max\left\{\varphi\left(f\right), \varphi(g)\right\}\right\},
\end{equation}
where $\mathcal A^*$ denotes the algebraic dual.
We endow $\tspec \cA$ with pointwise convergence topology.

\begin{prop}\label{prop:tropicalhomeo}
Let $K$ be a compact Hausdorff topological space, and $\mathcal A\subseteq \Cz(K,\R)$ a dense linear subspace, containing all the constants, that is stable by $\max$.
Then $\cA$ is subtroipcal algebra of $\Cz(K, \R)$, and the map: 
\[\delta\colon K\to \tspec \mathcal A , \quad \delta_x(f) = f(x) \text{ for }x\in K \text{ and }f\in \cA,\]
induces the homeomorphism:
\[[\delta]\colon K\to \left(\tspec( \mathcal A) \setminus\{0\}\right)/\R_{>0}.\]
\end{prop} 
\begin{proof}
It is clear that $\mathcal A$ is tropical algebra with the $\max$ (or $\min$ as it also preserves minima) and sum as operations.

Moreover, the map $\delta$ is clearly continous and injective, therefore it suffices to prove that $[\delta]$ is surjective, since $K$ is compact.

Let $\varphi\in \tspec \mathcal A\setminus\{0\}$, then
\begin{equation}\label{eq:positivephi}
\varphi(\lvert f\rvert ) = \varphi(\max\{f,-f\}) = \max \{\varphi(f), -\varphi(f)\} = \lvert \varphi(f)\rvert 
\end{equation}
hence $\varphi (1) >0$, thus we can suppose that $\varphi(1) =1$, and 
\begin{align*}
\lvert \varphi(f)\rvert = \varphi(\lvert f\rvert )&\leq \max\{\varphi(\lvert f\rvert), \varphi(\lVert f\rVert_\infty)\}\\
&=\varphi\left(\max\{\lvert f\rvert, \lVert f\rVert_\infty\}\right) = \varphi(\lVert f\rVert_\infty) = \lVert f\rVert_\infty\mathbin{\cdot} 1
\end{align*}
Hence $\varphi$ can be extended to a continuous linear functional on $\Cz(K)$. 
That is, $\varphi$ is a signed measure on $K$.
By~\eqref{eq:positivephi}, $\varphi$ is actually a positive measure.

Let $x\in \supp \varphi$, we will show that $\varphi = \delta_x$.
To do that we will just prove that $\ker \delta_x = \ker \varphi$, and the equality will follow since $\varphi(1) = 1=\delta_x(1) $.

If $f\in \ker \varphi$, then, by Equation~\eqref{eq:positivephi},  $\lvert f\rvert \in \ker \varphi $.
Since $\varphi$ is a positive measure, and $\lvert f\rvert \geq 0$, we get $f = 0$ $\varphi$-almost everywhere.
Therefore, since $f$ is continuous, the restriction:  \[f|_{\supp \varphi} = 0.\]
In particular, $f(x) =0$, and thus $f\in \ker \delta_x$.
Since $\codim \ker \varphi = \codim\ker \delta_x$, we conclude.
\end{proof}

\begin{proof}[Proof of Theorem~\ref{thm;dualcomplex}]
The map 
\[
p\colon X^\beth\to \dualc
\]
induces the isomorphism:
\[
\eta\colon \dualpl \to\PL(X^\beth),
\]
and therefore we get a homeomorphism:
\begin{align*}
\eta^*\colon \tspec \left(\dualpl\right)\to\tspec\left(\PL(X^\beth)\right),
\end{align*}
given by $\eta^*(\delta_x)(f)= \delta_x\left(\eta(f)\right)=\eta(f)(x) = f(p(x)) = \delta_{p(x)}(f)$
which means that the map
\begin{equation}
X^\beth\overset{\delta}{\to}\tspec \left(\PL(X^\beth)\right)\overset{\eta^*}{\to} \tspec \left(\dualpl\right)\overset{\delta^{-1}}{\to} \dualc
\end{equation}
is given by $x\mapsto p(x)$.
Hence, $p$ is a homeomorphism.
\end{proof}

\subsection{Log discrepancy on $X^\beth$}

\subsubsection*{Log discrepancy over $X$}
Let $F\subseteq Y\overset{\pi}{\longrightarrow} X$ be an irreducible divisor on a normal variety, and $v= r
\ord_F$ for some $r\in\Q_{>0}$
We define the \emph{log discrepancy of $v$} to be the quantity
\begin{equation}\label{eq;logdis1}
 A_X(v) \doteq r\mathbin{\cdot}\left(1+ \ord_F(K_{Y/X})\right).
\end{equation}
This gives us a function $A_X\colon \Xdiv\to \Q$, which will be called the \emph{log discrepancy over $X$}.

If $r=1$ we sometimes denote $A_X(F)\doteq A_X(\ord_F)$.

By some standard calculations, as in \cite[Section 3]{Kol97singularities}, the restriction of $A_X$ to (the rational points of) each face of the cone complex of a snc pair $(Y,B)$\footnote{See Section~\ref{sec:monomialcone}.}
\[\hat\Delta(Y,B)\overset{\val}{\hookrightarrow} X^\beth\] is linear. 
Hence, we can extend $A_X$ to $\hat\Delta(Y,B)$ by linearity.

Again by \cite[Section 3]{Kol97singularities}  if $(Y^\prime, B^\prime)\overset{\mu}{\to} (Y, B)$ is a snc reduced birational projective morphism over $(Y,B)$, then the piecewise linear function induced by the pullback:
\[r_\mu\colon \hat\Delta(Y, B)\to \hat\Delta(Y^\prime, B^\prime)\] satisfies the inequality
\begin{equation}
A_X\circ r_\mu \le A_X.
\end{equation}

\subsubsection*{Log discrepancy over $X\times \Pro$}
The same log discrepancy defined on the previous section makes sense for $X\times \Pro$.
We study now the relationship between $A_X$, and $A_{X\times \Pro}\circ\sigma$, where $\sigma$ is the Gauss extension.

Let $F\subseteq Y\overset{\pi}{\longrightarrow} X$ be a prime divisor, and $\ord_F$ the associated divisorial valuation.
Then,  consider the divisors: 
\begin{center}
\begin{tikzcd}
F\times\Pro\subseteq Y\times\Pro \arrow[d] & \& & Y\times\{0\}\subseteq Y\times \Pro \arrow[d] \\
X\times \Pro                               &    & X\times\Pro.                               
\end{tikzcd}
\end{center}

A direct calculation gives us that $\sigma(\ord_F)$ is monomial with respect to $F\times \Pro$ and $Y\times \{0\}$, with associated weights $ (1,1)$.
Therefore, using linearity of $A_{X\times \Pro}$, we get:
\begin{equation}
\begin{aligned}
A_{X\times \Pro}\left(\sigma(\ord_F)\right) &= A_{X\times \Pro}(F\times \Pro) + A_{X\times \Pro}(Y\times\{0\})\\
&= A_X(F) + 1 = A_X(\ord_F) +1.
\end{aligned}
\end{equation}

\smallskip

Let $\cX$ be a smooth test configuration.
Like we did in Section~\ref{sec:dualcomplex}, one can associate $\C^*$-equivariant monomial valuations on $\Xbc$ to points $w\in \sigma_Z\subseteq \Delta_\cX$.
Again,  the rational points on $\Delta_\cX$ correspond to points in $\Xdivc$. 
By the same the reasoning as before we get:
\begin{enumerate}
\item Let $w\in (\sigma_Z)_\Q\subseteq (\Delta_\cX)_\Q$, and $\val(w)$ the associated valuation satisfies: 
\[A_{X\times \Pro}(\val(w)) = \sum w_i\, A_{X\times\Pro}(E_i).\]
\item For $\cX^\prime\overset{\mu}{\to}\cX$ a morphism of test configurations, we have 
\[A_{X\times \Pro}\circ p_{\cX} \le A_{X\times\Pro}\]
on $(\Delta_{\cX^\prime})_\Q$.
\end{enumerate}
Therefore, by (1) we can extend by linearity $A_{X\times\Pro}$ to $\Delta_\cX$, and, by (2), define the limit:
\begin{equation}\label{eq;log}
A_{X\times\Pro}\colon \Xb \to \R\cup \{+\infty\}
\end{equation}
as the sup $A_{X\times \Pro}(v) \doteq \sup_{\cX}A\left(p_{\cX}(v)\right)$.

The function on~\eqref{eq;log} will be called the \emph{log discrepancy}, and from here on will be denoted by $A\colon X^\beth\to \R\cup\{+\infty\}$.

\begin{remark}
It is clear to see that from the definition of the log discrepancy, if $\cX$ is a test configuration, and $p_\cX$ is the function defined in Section~\ref{sec;dual}, then $A\circ p_\cX$ is a PL function.
\end{remark}

\section{Non-Archimedean plurisubharmonic functions}
\label{sec;nonarchimedeanpluri}
From now on, $X$ will be a compact Kähler  manifold, with a fixed Kähler class $\alpha\in \Pos(X)$.

\subsection{Plurisubharmonic PL functions}
Let us denote by $\triv\doteq \trive$ the trivial configuration, and by $p_1\colon\triv\to X$ the first projection.
Given $\beta\in H^{1,1} (X)$, we then denote \[\beta_{\triv}\doteq p_1^*\beta\in H^{1,1} (\trive)\]
More generally, given any test configuration that $\mu$-dominates $\trive$, we denote \[\beta_\cX\doteq \mu^*\beta_{\triv}.\]

\begin{remark}
  In \cite{SD18,DR17kstability} the authors introduce, independently, the notion of \emph{cohomological test configurations}, which are generalizations --to the transcendental setting-- of the usual algebraic test configurations for a polarized manifold $(X,L)$.
  
  For them, a cohomological test configuration is a test configuration $\cX$ together with a $\C^*$-invariant Bott--Chern cohomology class  $\mathcal A\in H^{1,1}_{\BC}(\cX)$ such that away from the central fiber:
  \[\cA|_{\cX^*} = h^*\alpha_\cX,\]
  for $h\colon \cX^*\to X\times (\Pro\setminus\{0\})$ the $\C^*$-equivariant biholomorphism.
  
  By \cite[Proposition~3.10]{SD18}, the data of a cohomological test configuration is the same of a test configuration together with the choice of a vertical divisor $D$, i.e.
  \begin{equation}\label{eq;cohomologicaltest}
  \mathcal A = \alpha_\cX + D
  \end{equation}
  for $D\in \VCar{(\cX)}$.
  \end{remark}

\begin{defi}
Let $\varphi\in \PL(X^\beth)$ we say that $\varphi$ is $\alpha$-\emph{plurisubharmonic} if given a dominating test configuration $\cX$ with $D\in \VCar(\cX)$ such that $\varphi = \varphi_D$, we have 
\begin{equation}
\alpha_\cX +D \text{ is nef relatively to } \Pro.
\end{equation}
Since the pullback by an holomorphic map of a $(1,1)$-class is nef if and only if the $(1,1)$-class itself is nef, this definition does not depend on $\cX$.

We will denote the set of $\alpha$-psh functions by $\PL\cap\PSH(\alpha)$.
Moreover, we denote by $\Hdom(\alpha)$ the set of functions $\varphi\in \PL(X^\beth)$ such that there exists a dominating test configuration $\cX$ and a vertical divisor $D\in \VCar(\cX)$ satisfying:
\[\varphi = \varphi_D, \quad \alpha_\cX +D \text{ is Kähler relatively to } \Pro.\] 
\end{defi}
\begin{remark}
In the standard algebraic setting, the set of non-Archimedean Fubini--Study metrics is usually denoted by $\cH^{\NA}$.
Here $\cH(\alpha)$ will play the role of this set in our more general context. 
Note, however, that for an algebraic variety, $\cH(\alpha)$ is not the set Fubini--Study functions, the latter is only a subset: $\cH^{\NA}\subsetneq\cH(\alpha)$. 
\end{remark}

\begin{prop}
Let $\varphi, \psi\in \PL(X^\beth)\cap \PSH(\alpha)$, and $f\colon X\to Y$ be a finite holomorphic map, then the following properties hold:
\begin{enumerate}
\item $f^*\varphi \in \PL\cap\PSH(f^*\alpha);$
\item $\varphi + c$, and  $t\mathbin{\cdot} \varphi$ lie in $\PL\cap\PSH(\alpha)$ for $c\in \R$ and $t\in\Q_{>0};$
\item $\max\{\varphi, \psi\}\in \PL\cap\PSH(\alpha).$
\end{enumerate}
\end{prop}
While proof of items $(1)$ and $(2)$ is essentially the same as in the algebraic trivially valued case, cf. \cite[Proposition~3.6]{BJ22trivval}, the proof of item $(3)$ is different and relies on the analysis of singularities of psh functions of Lemma~\ref{lem;sumnef}.

\begin{proof}
  
Let $\cX$ be a test configuration $\mu$-dominating the trivial one, such that 
\[\varphi = \varphi_D, \quad \text{and}  \quad \psi= \varphi_E,\] 
for some $D, E\in \VCar(\cX)$.

For $(1)$ it is enough to observe that the pull-back of a PL function is PL and that pull-back of a nef class is nef.

For the first part of $(2)$ it is enough to observe that:
\begin{itemize} 
  \item $\varphi + c = \varphi_{D + c\cX_0};$
  \item On the quotient \[H^{1,1}(\cX)/\mu^*(H^{1,1}(\Pro))\] we have $[\alpha+ D] = [\alpha + D + c\cX_0]$, and hence \[D+\alpha+c\cX_0\in \Nef(\cX/\Pro)\iff D+\alpha\in \Nef(\cX/\Pro).\]
\end{itemize}

If $t\in \mathds N$, then it is enough to observe that, like in the trivially valued case \cite[Proposition~3.6]{BJ22trivval},  \[\varphi_{D_t} = t\cdot \varphi_D\] where $D_t = \mu_t^*D$ is given by the base change
\begin{equation}
\begin{tikzcd}
\cX^t \arrow[r, "\mu_t"] \arrow[d, "\pi_t"] & \cX \arrow[d, "\pi"] \\
\Pro \arrow[r, "z^t"]                         & \Pro.                  
\end{tikzcd}
\end{equation}
The result follows from $(1)$, the general case $t\in \Q_{>0}$ follows from this one, since a cohomology class is nef iff its pullback by a finite branched covering is.

For item $(3)$, let's suppose $\varphi, \psi\in \Hdom(\alpha)$, the general case will follow from an approximation argument, cf. Theorem~\ref{thm;plpsh} below.

Now, let $c>0$ be large enough so that $D^\prime\doteq -D + c\cX_0$ and $E^\prime\doteq-E +c\cX_0$ are effective, we have:
\[\max\{\varphi_{-D^\prime}, \varphi_{-E^\prime}\} = \max\{\varphi, \psi\} -c,\]
and thus it is $\alpha$-psh iff $\max\{\varphi, \psi\}$ is $\alpha$-psh.

Moreover, let $\cX^\prime$ be a test configuration $\nu$-dominating $\cX$, such that 
\[\varphi_{G} = \max\{\varphi_{-D^\prime}, \varphi_{-E^\prime}\}\] 
for $G\in \VCar(\cX^\prime)$, then we observe that $\cO_{\cX^\prime}(-G)=\nu^*[\mathcal O_{\cX}(D^\prime) + \mathcal O_{\cX}(E^\prime)]$, and the result follows from Lemma~\ref{lem;sumnef}.
\end{proof}


\begin{theorem}\label{thm;plpsh}
Let $\varphi\in\PL(X^\beth)$, and $(\varphi_{\lambda})_{\lambda\in\Lambda}$ be a net of PL $\alpha$-psh functions such that \[\varphi_{\lambda}(v)\to\varphi(v), \; \forall\,v\in \Xdiv\]
then $\varphi$ is a PL $\alpha$-psh function.
\end{theorem}
\begin{proof}
Let $\cX_\lambda$, and $\cX$ be snc test configurations together with morphisms of test configurations $\mu_\lambda\colon\cX_\lambda\to\cX$, and $\nu\colon\cX\to\trive$, such that there exist vertical divisors $D\in \VCar(\cX)$, and $D_\lambda\in \VCar(\cX_\lambda)$ satisfying:
\[\varphi=\varphi_D, \quad  \text{ and } \quad \varphi_\lambda = \varphi_{D_\lambda}.\]

To prove that $D+\alpha_\cX$ is nef relatively to $\Pro$ it is enough to show that for every irreducible --hence smooth-- component of the central fiber, $E\subseteq \cX_0$, the restriction \((D+\alpha_\cX)_{|_E}\) is in \(\Nef(E)\). 
This follows from the simple observation that for $\tau\neq 0$ we have:
\[D_{|_{\cX_\tau}} = 0 \quad \&\quad (\alpha_\cX)_{|_{\cX_\tau}} = h_\tau^*\alpha\in \Nef(\cX_\tau),\] where $h_t\colon \cX_\tau \to X$ is the biholomorphism provided by the $\C^*$-action.

Let then $Y^d\subseteq E$ be a $d$-dimensional subvariety of $E$, and $\gamma\in \Pos(\cX)$ a Kähler class, by Demailly-Paun numeric characterization of nefness, it sufices to show that:
\begin{equation}
\int_Y (D+\alpha_\cX)\wedge\gamma^{d-1} \ge 0.
\end{equation}
To simplify notation, we rewrite the left hand side:
\begin{equation}\label{intg=intsec}
\int_Y (D+\alpha_\cX)\wedge\gamma^{d-1}=[Y]\mathbin{\cdot}(D+\alpha_\cX)\mathbin{\cdot}\gamma^{d-1}.
\end{equation}

We can suppose that $Y$ is invariant by the $\C^*$-action. 
Indeed, since $E\subseteq \cX_0$ is irreducible, it is itself invariant.
Therefore, denoting $Y_\tau \doteq \tau \cdot Y$, we get  by compacity of each component of the space of effective cycles on $E$, cf. \cite{HS74characterization, Fuj78closedness}, that the limit $\lim_{\tau\to 0} Y_\tau$ exists as an effective cycle $\sum a_Z Z$, where the components $Z$ are $\C^*$-invariant.
Moreover, since $\C^*$ acts trivially on cohomology, we observe that $[Y_\tau] = [Y]$, and:
\[[Y]\cdot (D+\alpha_{\cX})\cdot \gamma^{d-1} =[Y_\tau]\cdot \rho(\tau)^*(D+\alpha_{\cX})\cdot \rho(\tau)^*\gamma ^{d-1} \to \sum a_Z [Z]\cdot (D+\alpha_\cX)\cdot \gamma^{d-1},\]
for $\rho$ the $\C^*$-action on $\cX$.
Replacing $Y$ for $Z$ we get the $\C^*$-invariance.

Let $b\colon\cX^\prime\to\cX$ be the normalized blow-up of $\cX$ along $Y$.
Since $Y$ is $\C^*$-invariant  $\cX^\prime$ is a test configuration. 
Let $\sum_{j}b_jF_j=F\subseteq \cX^\prime$ be the (effective) exceptional divisor, and consider the positive current, of bi-dimension $(d,d)$, given by:
\[T\doteq\delta_F\wedge \Omega^{n-d},\] 
where $\delta_F$ denotes $\sum_{j} b_j\delta_{F_j}$ the weighted sum current of the current of integration along $F_j$, and $\Omega\in\mathcal K(\cX^\prime)$ denotes a Kähler form.

Consider now $b_*(T)$, by definition this is again a positive current of bi-dimension $(d,d)$, with support $\supp b_*(T)\subseteq b(\supp(T))=b(F)= Y$.
By Demailly's support theorem  every current of bi-dimension $(d,d)$ supported on an irreducible cycle of dimension $d$ must be a multiple of the current of integration over that cycle, which implies that the cohomology classes $[b_*(T)]=a[Y]$, for $a\ge 0$.
Choosing $\eta$ to be a Kähler form on $\cX$ such that $\Omega-b^*\eta$ is positive on the fibers of $b$, we have:
\[b_*T\cdot \eta = T\cdot b^*\eta\ge \sum_j b_j\int_{F_j} (\Omega-b^*\eta)^{n-d}\wedge (b^*\eta)^d>0,\]
thus $[b_*T]\ne 0\implies a>0.$

Hence, Equation~\eqref{intg=intsec} becomes: 
\begin{align*}
\frac{1}{a}[b_*(T)]\mathbin{\cdot}(D+\alpha_\cX)\mathbin{\cdot}\gamma^{d-1} &= \frac{1}{a}[T]\mathbin{\cdot}b^*(D+\alpha_\cX)\mathbin{\cdot}b^*\gamma^{d-1}\\
&=\frac{1}{a}[F]\mathbin{\cdot}[\Omega]^{n-d}\mathbin{\cdot}(D_{\cX^\prime}+\alpha_{\cX^\prime})\mathbin{\cdot}\gamma_{\cX^\prime}^{d-1},
\end{align*}
where the first equality holds by the projection formula, and $D_{\cX^\prime}\doteq b^* D$.

Now, let $\cX^\prime_\lambda$ be a test configuration that dominates both $\cX_\lambda$ and $\cX^\prime$
\begin{center}
\begin{tikzcd}
\cX_\lambda \arrow[d, "\mu_\lambda"] & \cX^\prime_\lambda \arrow[d, "\nu_\lambda"] \arrow[l, "b_\lambda"'] \\
\cX                                  & \cX^\prime. \arrow[l, "b"]                                             
\end{tikzcd}
\end{center}
Denoting $F_\lambda\doteq \nu_\lambda^* F$, and $\omega_{\lambda}\doteq \nu_\lambda^*\omega$, we observe:
\begin{align*}
0\le\frac{1}{a}[F_\lambda]\mathbin{\cdot}[\omega_{\lambda}]^{n-d}\mathbin{\cdot}(D_{\cX^\prime_\lambda}+\alpha_{\cX^\prime_\lambda})\mathbin{\cdot}\gamma_{\cX^\prime_\lambda}^{d-1},
\end{align*}
since $F_\lambda$ is effective, and $D_\lambda +\alpha_{\cX_\lambda}$, $\omega$, and $\gamma$ are nef --which implies that $D_{\cX^\prime_\lambda} +\alpha_{\cX^\prime_\lambda}$, $[\omega_{\lambda}]$, and $\gamma_{\cX^\prime_\lambda}$ are nef as well.

Again by the projection formula, we have:
\begin{align*}
0&\le\frac{1}{a}[F_\lambda]\mathbin{\cdot}[\omega_{\lambda}]^{n-d}\mathbin{\cdot}(D_{\cX^\prime_\lambda}+\alpha_{\cX^\prime_\lambda})\mathbin{\cdot}(\gamma_{\cX^\prime_\lambda})^{d-1}\\
&=\frac{1}{a}[F]\mathbin{\cdot}\alpha^{n-d}\mathbin{\cdot}[{\nu_\lambda}_*D_{\cX^\prime_\lambda}+\alpha_{\cX^\prime}]\mathbin{\cdot}\gamma_{\cX^\prime}^{d-1}.
\end{align*}
Now, since:
\[
{\nu_\lambda}_*D_{\cX^\prime_\lambda}= \sum_{G\stackrel{\text{irred}}{\subseteq} \cX^\prime_0} b_G\, \varphi_{D_\lambda}(v_G) G\longrightarrow \sum_{G\stackrel{\text{irred}}{\subseteq}  \cX^\prime_0}  b_G\, \varphi_D(v_G) G=D_{\cX^\prime}
\]
it follows that:
\begin{align*}
\frac{1}{a}[F]\mathbin{\cdot}\alpha^{n-d}\mathbin{\cdot}[{\nu_\lambda}_*D_{\cX^\prime_\lambda}+\alpha_{\cX^\prime}]\mathbin{\cdot}\gamma_{\cX^\prime}^{d-1}\to \frac{1}{a}[F]\mathbin{\cdot}\alpha^{n-d}\mathbin{\cdot}(D_{\cX^\prime}+\alpha_{\cX^\prime})\mathbin{\cdot}\gamma_{\cX^\prime}^{d-1}\ge 0,
\end{align*}
concluding the proof.
\end{proof}

\subsection{PL Monge--Ampère operator and energy pairing}\label{sec;plenergy}
In this section we will give the definition of the PL Monge--Ampère operator and, more genrally, the PL energy pairing. We also state a few important properties and results.

The notions of pluripotential theory for $X^\beth$, introduced in this paper, are under the synthetic formalism developed in \cite{BJ23synthetic}.
In particular, every result from Section 1 to 3 of Boucksom-Jonsson's synthtetic approach holds in our case.

We will recall some of the results, for more details see \cite[Section 3.2]{BJ22trivval} and \cite[Section 1]{BJ23synthetic}.

\subsubsection{Monge--Ampère measure of a PL function}
\label{sec;mapl}
Let $\beta\in H^{1,1}(X)$ be a cohomology class of positive self intersection, i.e. $V_\beta\doteq [X]\cdot \beta^n>0$, and $\varphi\in\PL(X^\beth)$ a PL function, we can associate to the pair $(\beta, \varphi)$ a signed measure on $X^\beth$, called \emph{Monge--Ampère measure}, given by the construction:
\begin{itemize}
  \item Let $\cX$ be a snc test configuration dominating $X\times \Pro$ such that there exists a vertical divisor $D\in \VCar(\cX)$ satisfying $\varphi_D = \varphi$.
  \item Denote $\cX_0 = \sum b_E\, E$ the decomposition in irreducible components of the central fiber, and let \(c_E\) be the constant given by \[\frac{b_E}{V_\beta} \left((\beta_\cX + D)|_E\right)^n.\]
  \item Define the signed measure as:
    \[\MA_\beta (\varphi) \doteq \sum_{E\overset{\text{irred}}{\subseteq} \cX_0}c_E\,\delta_{v_E}.\] 
  \end{itemize}

With the projection formula one checks that this definition does not depend on the choice of test configuration.
Moreover, if $\beta\in \Pos(X)$ is a positive class and $\varphi\in\cH(\beta)$ the above construction get us a probability measure.
\begin{defi}\label{def;mapl}
  Let $\alpha\in \Pos(X)$ be Kähler class, and $\mathcal P(X^\beth)$ the set of Radon probability measures on $X^\beth$. 
  We call the operator:  
  \[\Hdom(\alpha)\ni \varphi\overset{\MA}{\mapsto}\MA_\alpha(\varphi)\in \mathcal P(X^\beth), \]
  the \emph{Monge--Ampère operator}.
  Whenever the choice $\alpha$ is clear by context, we write $\MA(\varphi)$ for $\MA_\alpha(\varphi)$.
\end{defi}
Using the identification of $X^\beth$ with the limit of dual complexes, 
we can observe that if $\varphi\in \Aff_{\Q}(\Delta_\cX)$ is rationally affine on each face of the dual complex $\Delta_{\cX}$, then $(p_{\cX})^*\varphi \in \PL(\Delta)$ is such that 
$\MA((p_{\cX})^*\varphi)$ is supported on $\Delta_{\cX}\subseteq X^\beth$. 
Where $p_{\cX}$ is the retraction $X^\beth \to \Delta_{\cX}$.

\begin{remark}
  A borelian measure on a compact Hausdorff topological space $K$ is completely determined by its values on a dense subset of $\Cz(K)$, therefore the probability measure $\MA(\varphi)$ is completely determined by its values on the set $\PL(X^\beth)$.
  We will use this approach to generalize the Monge--Ampère measure (see Remark~\ref{rem:mongeamperee1}), and construct the \emph{mixed Monge--Ampère energy} (see next section).
\end{remark}

\subsubsection{Energy pairing for PL functions}

Let $\varphi_0,\dotsc, \varphi_n\in \PL_\R$, and $\beta_0,\dotsc, \beta_n\in H^{1,1}(X)$.
\begin{defi}
Let $\cX$ be a test configuration dominating $X\times \Pro$, such that there exist $D_0,\dotsc, D_n\in\VCar(\cX)$, with the property that for every $i=0,\dotsc, n$ we have \[\varphi_i =\varphi_{D_i}\]
then we define the \emph{energy pairing}
\begin{equation}
(\beta_0,\varphi_0)\mathbin{\cdot} (\beta_1, \varphi_1)\dotsb(\beta_n,\varphi_n)\doteq (\beta_{0,\cX} + D_0)\dotsb(\beta_{n,\cX} +D_n)\in\R,
\end{equation}
where, in the right hand side of the inequality, the intersection product is against the fundamental class of $\cX$.
As before, the projection formula guarantees that the energy pairing is well defined.

We also refer to the energy pairing as the \emph{energy coupling}, or as \emph{mixed Monge--Ampère energy}, see item $(iv)$ of Proposition~\ref{prop:basicenergypairing} below.
\end{defi}

Clearly the pairing is a symmetric multi-linear form, which further satisfies the following properties:
\begin{prop}
  \label{prop:basicenergypairing}
Let $\varphi_0,\dotsc, \varphi_n\in \PL_\R$, and $\cX$ a test configuration such that there exist $D_0,\dotsc, D_n\in \VCar_\R(\cX)$ with \[\varphi_i = \varphi_{D_i},\] and let $\beta_0,\dotsc, \beta_n\in H^{1,1}(X)$, $t\in\Q_{>0}$, and $\alpha\in\Pos(X)$ then we have:
\begin{enumerate}[(i)]
\item $(0,1)\mathbin{\cdot} (\beta_1, \varphi_1)\dotsb (\beta_n, \varphi_n) = \beta_1\dotsb \beta_n$
\item $(\beta_0, 0)\dotsb(\beta_n, 0) =0$
\item $(\beta_0, t\mathbin{\cdot} \varphi_0)\dotsb (\beta_n, t\mathbin{\cdot}\varphi_n) = t(\beta_0, \varphi_0)\dotsb (\beta_n, \varphi_n)$
\item For $\psi\in\PL_\R$, and $V_\alpha\doteq [X]\cdot\alpha^n$, we have \[\frac{1}{V_\alpha}(0,\psi)\mathbin{\cdot}(\\\alpha,\varphi_0)^n = \int_{X^\beth} \psi\MA_\alpha(\varphi_0) \]
\item More generally, for $\psi\in\PL_\R$ we have \[(0,\psi)\mathbin{\cdot}(\beta_1,\varphi_1)\dotsb (\beta_n,\varphi_n) = \sum_{E\subseteq \cX_0}b_E\,\psi(v_E)(\beta_1+D_1)|_E\dotsb(\beta_n+D_n)|_E \] for $E$ irreducible component and $b_E = v_E(\cX_0)$. 
\end{enumerate}
\end{prop}
\begin{proof}
($i$) This follows from the remark that, as cohomology classes, $[\cX_0] = [\cX_1]$. 
Indeed, both of them can be written as $\pi^*([0])$ and $\pi^*([1])$ respectively, where $[0]$ and $[1]$ represent the cohomology classes of $0,1\in \Pro$ in $H^2(\Pro, \C)\cong \C$, but since $[0]=[1]$, the result follows from the flatness of $\pi\colon \cX\to \Pro$. 

($ii$) Observe that: \[[\cX_0]\mathbin{\cdot} \beta_{0,\cX}\dotsb \mathbin{\cdot} \beta_{n,\cX} = [X]\mathbin{\cdot}\beta_0\dotsb\dotsb \beta_n =0.\]

($iii$) It is enough to check that for each $d\in\mathds Z_{>0}$ we have $(\beta_0, d\mathbin{\cdot} \varphi_0)\dotsb (\beta_n, d\mathbin{\cdot}\varphi_n) = d(\beta_0, \varphi_0)\dotsb (\beta_n, \varphi_n)$, but then again $d\mathbin{\cdot}\varphi_D = \varphi_{D_{d}}$, where $D_d$ is the pullback of $D$ under the normalized base change:
\begin{equation}
\begin{tikzcd}
E_d\subseteq \widetilde{\cX_d} \arrow[d] &                           \\
 \cX_d \arrow[r] \arrow[d]               & E\subseteq\cX \arrow[d] \\
\Pro \arrow[r, "t^d"]                      & \Pro,                   
\end{tikzcd}
\end{equation}
in addition  if $\alpha_0, \dotsc, \alpha_{n+1}\in H^{1,1}(\cX)$, then $(\alpha_0)_{\cX_d}\dotsb(\alpha_{n+1})_{\cX_d} = d\, \alpha_0 \dotsb \alpha_{n+1}$
and the result follows.

($iv$) Follows from ($v$).

($v$) Let $\cX^\prime$ be a test configuration $\mu$-dominating $\cX$ such that there exists $G\in\VCar_\R(\cX^\prime)$ with \[\psi = \varphi_G,\]  then we have:
\begin{equation}\label{eq;energypairing}
\begin{aligned}
(0,\psi)\mathbin{\cdot}(\beta_1,\varphi_1)\dotsb (\beta_n,\varphi_n) &= G\mathbin{\cdot} \left(\beta_{1,\cX^\prime} + \mu^*D_1\right)\dotsb \left(\beta_{n,\cX^\prime} +\mu^*D_n\right)\\
&=\mu_* G\cdot \left(\beta_{1,\cX} + D_1\right)\dotsb \left(\beta_{n,\cX} +D_n\right)\\
&= \sum_{E\subseteq \cX_0} \ord_E(G)\, E\mathbin{\cdot} (\beta_1 + D_1)\dotsb (\beta_n +D_n)\\
&=\sum_{E\subseteq\cX_0}b_E\,\varphi_G(v_E) (\beta_1 + D_1)|_E\dotsb (\beta_n +D_n)|_E\\
&= \sum_{E\subseteq \cX_0}b_E\,\psi(v_E)(\beta_1+D_1)|_E\dotsb(\beta_n+D_n)|_E,
\end{aligned}
\end{equation}
where the second equality is given by the projection formula.
\end{proof}

\begin{corollary}
Let $\beta_1,\dots \beta_n\in \Pos(X) $ and for each $i$ a $\beta_i$-psh function $\varphi_i\in\PL\cap\PSH(\beta_i)$. If $\gamma\in H^{1,1}(X)$, and $\psi,\psi^\prime\in \PL$ are such that $\psi\leq \psi^\prime$, then 
\[(\gamma,\psi)\cdot(\beta_1, \varphi_1)\dotsb(\beta_n,\varphi_n) \leq (\gamma,\psi^\prime)\cdot(\beta_1, \varphi_1)\dotsb(\beta_n,\varphi_n)\]
\end{corollary}
\begin{proof}
Follows directly from equation~\ref{eq;energypairing}.
\end{proof}

\begin{lemma}[Zariski's Lemma]\label{lem:zariskilemma}
Let $\psi$ be a PL function, and, for $i = 2, \dotsc, n$, let $\varphi_i\in \PL\cap\PSH(\beta_i)$. 
Then 
\begin{equation}
(0,\psi)^2\mathbin{\cdot}(\beta_2, \varphi_2)\dotsb (\beta_n, \varphi_n)\leq 0.
\end{equation}
\end{lemma}
\begin{proof}
Let $\cX$ be a test configuration that dominates $X\times\Pro$ with $D, D_2, \dotsc, D_n\in \VCar(\cX)$ such that $\psi = \varphi_D$ and $\varphi_i = \varphi_{D_i}$.

The energy pairing induces a bilinear form on the finite dimension vector space $\VCar_\R(\cX)$ as the map:
\[(G_1, G_2)\mapsto (0,\varphi_{G_1})\mathbin{\cdot}(0, \varphi_{G_2})\mathbin{\cdot}(\beta_2, \varphi_2)\dotsb (\beta_n, \varphi_n).\]
Therefore, we must prove that this bilinear form is negative semidefinite.

The strategy will be to apply  Lemma~\ref{lem;bilinear}. 
Let $(E_i)$ be the irreducible components of $\cX_0$, and note that they form a basis of $\VCar$. 
Then:   
\begin{enumerate}
\item For every $i, j$ different to each other, $(E_i, E_j)=E_i\mathbin{\cdot} E_j \mathbin{\cdot}(\beta_2 +D_2)\dotsb (\beta_n+ D_n)$ is non-negative, since $ \beta_K +D_k$ is nef, for $k=2, \dots n$.
\item Consider $\cX_0$ as an element of $ \VCar(\cX)$, then  $(\cX_0, E_i) = \cX_0\mathbin{\cdot} E_i\mathbin{\cdot} (\beta_2 +D_2)\dotsb (\beta_n+ D_n)=0$, since every component is supported on $\cX_0$. 
\end{enumerate}
Therefore, the bilinear form is negative semi-definite, and the result follows.
\end{proof}
This version of the Zariski Lemma that we have just proved is what allow us to use all the synthetic pluripotential theory of \cite{BJ23synthetic}, that will provide us with important a priori estimates for the mixed energy.
For more details see Appendix~\ref{apx:synthetic}.

Now we will recall some notions from --the above mentioned-- synthetic pluripotential theory restricting it to our case.
\subsubsection*{Recall of some synthetic facts of Boucksom--Jonsson}
If we let $\varphi_k\in \PL\cap\PSH(\beta_k)$ for $k=1, \dotsc, n-1$,
and denote the symbol $(\beta_1,\varphi_1)\dotsb (\beta_{n-1},\varphi_{n-1})$ by $\Gamma$, we can associate a semi-norm:
\[
	\lVert \psi\rVert_{\Gamma}\doteq \sqrt{-(0,\psi)^2\mathbin{\cdot}\Gamma},
\]
for $\psi\in\PL_\R$.

\begin{remark}
Since the (positive semi-definite) quadratic form $-(0,\psi)^2\mathbin{\cdot}\Gamma$ comes from a bilinear form we have an associated Cauchy-Schwarz inequality:
\[
\lvert (0,\psi_1)\mathbin{\cdot}(0,\psi_2)\mathbin{\cdot}\Gamma\rvert\leq \lVert \psi_1\rVert_{\Gamma}\lVert\psi_2\rVert_{\Gamma},
\]
that is the base of the synthetic estimates of \cite{BJ23synthetic}.
\end{remark}

\begin{defi}
Let $\varphi, \psi\in \PL\cap\PSH(\alpha)$, and denote:
\begin{align}
\mathrm J_\alpha(\varphi, \psi)&\doteq \frac{1}{V_\alpha}\sum_{j=1}^{n}\frac{j}{n+1}\lVert \varphi - \psi\rVert_{(\alpha, \varphi)^{j-1}\cdot (\alpha,\psi)^{n-j}};\label{eq;Jenergy}\\
\mathrm I_\alpha(\varphi, \psi)&\doteq \frac{1}{V_\alpha}\sum_{j=1}^{n}\lVert \varphi - \psi\rVert_{(\alpha, \varphi)^{j-1}\cdot (\alpha,\psi)^{n-j}}\label{eq;Ienergy}.
\end{align}
\end{defi}
Theorem~1.33 of \cite{BJ23synthetic} gives that these functionals define equivalent quasi-metrics.

We also recall the following defintion:
\begin{defi}
We define the \emph{Monge--Ampère energy} as the functional \(\mathrm E_\alpha\colon \PL(X^\beth)\to \R\)
given by the expression
\[\PL\ni \varphi\mapsto \frac{1}{(n+1)V_\alpha} (\alpha,\varphi)^{n+1}.\]
\end{defi}
If $\varphi, \psi\in \PL(X^\beth)$, the variation of $\mathrm E_\alpha$ is given by
\begin{equation}\label{eq;derivativeenergy}
\begin{aligned}
\dtz\mathrm E_\alpha\left(t\psi + (1-t)\varphi\right) &= \dtz \frac{1}{(n+1)V_\alpha} (\alpha,t\psi+(1-t)\varphi)^{n+1}\\
&=\frac{1}{V_\alpha}(\alpha, \varphi)^n \mathbin{\cdot} (0,\psi - \varphi)\\
&=\int_{X^\beth}(\psi - \varphi) \MA_\alpha(\varphi),
\end{aligned}
\end{equation}
justifying its name.

The energy $\mathrm E_\alpha$ restricts to a concave functional on $\PL\cap\PSH(\alpha)$, and thus for $\varphi, \psi\in \PL\cap\PSH(\alpha)$ we get that
\begin{equation}\label{eq;Econcave}
\mathrm E_\alpha(\psi)\le \mathrm E_\alpha(\varphi) + \int_{X^\beth}(\psi-\varphi) \MA_\alpha(\varphi),
\end{equation}
and the difference
\[ \mathrm E_\alpha(\varphi) - \mathrm E_\alpha(\psi) + \int_{X^\beth}(\psi-\varphi) \MA_\alpha(\varphi)\]
 coincides with $\mathrm J_\alpha(\varphi, \psi)$.

\subsection{Non-Archimedean psh functions}
We wil now define one of the most important objects of study of this paper the: \emph{non-Archimedean psh functions}.
\begin{defi}
A function \[\psi\colon X^\beth\to \left[ -\infty, +\infty\right[ \] is $\alpha$-\emph{psh} if $\psi\not\equiv -\infty$, and there exists a decreasing net $(\varphi_\lambda)_{\lambda\in\Lambda}\in \PL\cap\PSH(\alpha)$ such that 
\begin{equation}
\varphi_\lambda (v)\searrow \psi(v), \text{ for every } v\in X^\beth.
\end{equation}
The set of all $\alpha$-psh functions will be denoted $\PSH(\alpha)$.
\end{defi}
Like for the PL functions we have the following properties:
for $\varphi, \psi\in \PSH(\alpha)$, and $f\colon X\to Y$ a finite holomorphic map:
\begin{enumerate}
\item $f^*\varphi \in \PSH(f^*\alpha)$, 
\item $\varphi + c$, and  $t\mathbin{\cdot} \varphi$  are $\alpha$-psh for $c\in \R$ and $t\in\Q_{>0}$,
\item $\max\{\varphi, \psi\}\in \PSH(\alpha).$
\end{enumerate}

Next, an immediate consequence of the definition:
\begin{lemma}\label{lem:intersectionpsh}
  The intersection \[\bigcap_{\lambda>1} \PSH(\lambda\cdot \alpha) = \PSH(\alpha).\]
\end{lemma} 
\begin{proof}
  It is clear that $\bigcap_{\lambda>1} \PSH(\lambda\cdot \alpha) \supseteq \PSH(\alpha)$, we will prove now the other inclusion.
  Let $\varphi\in \PSH(\lambda\cdot \alpha)$ for every $\lambda>1$, but subtracting a constant we can suppose that $\varphi\le 0$.
  Now, it is enough to observe that the net $\varphi_\lambda\doteq\frac{1}{\lambda}\varphi\in \PSH(\alpha)$ decreases to $\varphi$ as $\lambda\searrow 1$, which then implies that $\varphi\in \PSH(\alpha)$.
\end{proof}

\medskip

The next result is well known in the algebraic case, see for instance \cite[Corollary~4.17]{BJ22trivval} or \cite{BFJ16semipositive}, and the proof in our transcendental setting goes without a change.
We follow the proof of \cite[Section 6.1]{BFJ16semipositive}, which is added here for completeness.
\begin{theorem}
If $\psi\in \PSH(\alpha)$, then \[\psi|_{\Xdiv} > -\infty.\]
\end{theorem}
\begin{proof}
Let $v$ be a divisorial valuation, and $\cX$ be a test configuration such that, decomposing the central fiber in irreducible components  \[\cX_0 =\sum_{j=0}^k b_j\, E_j, \] we have $ v =v_{E_0}$.

Consider now, $\gamma\in\mathcal \Pos(\cX)$, and $(\varphi_\lambda)_\lambda\in \PL\cap\PSH(\alpha)$ such that:
\[\varphi_{\lambda} (v)\searrow \psi(v), \text{ for every } v\in X^\beth.\]
We may assume  $\sup \psi = 0$, and hence that $\sup \varphi_{\lambda} = 0 = \max \{\varphi_{\lambda}(v_{E_j})\}$.

If $\cX_\lambda$ is a test configuration that $\mu_\lambda$-dominates $\cX$, with $D_\lambda\in\VCar( \cX_\lambda)$ such that \(\varphi_\lambda = \varphi_{D_\lambda}\), then, since $(\mu_\lambda)^* E_j$ is an effective divisor, and the classes 
\[\alpha_{\cX_\lambda} + D_\lambda,\, (\mu_\lambda)^*\gamma\text{ are relatively nef w.r.t. }\Pro,\]
we have:
\begin{equation}\label{eq:inequality1}
0\le (\mu_\lambda)^*E_j\mathbin{\cdot}(\alpha_{\cX_i} + D_\lambda)\mathbin{\cdot}\mathcal (\mu_\lambda)^*\gamma^{n-1}=E_j\mathbin{\cdot}(\alpha_{\cX} + (\mu_\lambda)_* D_\lambda)\mathbin{\cdot}\gamma^{n-1},
\end{equation}
where the equality is given by the projection formula.

Since $(\mu_\lambda)_* D_\lambda = \sum_k b_k\,\varphi_{D_\lambda}(v_{E_k})\, E_k$,  rewriting the inequality~\eqref{eq:inequality1}, it follows:
\[
\sum_k b_k\,\varphi_{D_\lambda}(v_{E_k})\, (E_j\mathbin{\cdot} E_k\mathbin{\cdot} \gamma^{n-1})\geq - E_j\mathbin{\cdot} \alpha_\cX\mathbin{\cdot} \gamma^{n-1}.
\]
Now, if $E_j\cap E_k\neq \empty$, then $E_j\mathbin{\cdot} E_k\mathbin{\cdot} \gamma^{n-1}> 0$, and thus for all $j$:
\begin{align*}
b_j (E_j\mathbin{\cdot} E_j\mathbin{\cdot} \gamma^{n-1})&= E_j\mathbin{\cdot} (b_j E_j - \cX_0)\mathbin{\cdot} \gamma^{n-1}\\
 &= -\sum_{k\neq j}b_k(E_j\mathbin{\cdot} E_k\mathbin{\cdot} \gamma^{n-1})\leq -1
\end{align*}
where the first equality comes from flatness of $\pi\colon \cX\to\Pro$, and the last inequality comes from $\cX_0$ being connected with at least two irreducible components.
Exactly like \cite[Section 6.1]{BFJ16semipositive}, we get:
\begin{equation}
\lvert \varphi_{D_\lambda}(v_{E_j})\rvert \leq C(\cX,\alpha,\gamma),
\end{equation}
for some constant $C$ depending only on $\cX,\alpha$ and $\gamma$.
Hence: \[ C\ge \lim_i \lvert \varphi_{D_\lambda}(v_{E_j}) \rvert = \lvert \psi(v_{E_j})\rvert,\]
concluding the proof. 
\end{proof}

Thanks to the above result, a natural topology to endow $\PSH(\alpha)$ will be the topology of pointwise convergence on divisorial valuations.

Bellow, we will also prove that:
\[\varphi\le \psi \text{ on } \Xdiv \implies \varphi \le \psi \text{ on } X^\beth\] 
for $\varphi\in \PSH(\alpha)$ and $\psi\colon X^\beth\to\left[-\infty, +\infty\right[$ a usc function.
To get this result we need the following description of divisorial valuations:
\begin{defi}
Let $\ia\in \flag$ be a flag ideal, the set $\Sigma_\ia\subseteq \Xdivc= \sigma(\Xdiv)$ of divisorial valuations given by the irreducible components of the exceptional divisor of the normalized blow-up $\widetilde{\bl_{\ia}X\times\Pro}$ are called the \emph{Rees valuations associated to $\ia$.}
\end{defi}
It is clear from definition that:
\begin{enumerate}
\item $\Sigma_\ia = \Sigma_{\iac}$
\item $\Sigma_{\ia^m} = \Sigma_\ia$
\item $\bigcup_{\ia\in \ideal_X}\Sigma_\ia = \Xdivc$
\end{enumerate}

\medskip

The next result is a generalization of \cite[Lemma~2.13]{BJ22trivval}, and the proof follows the same general lines.
\begin{lemma}\label{lem;sup}
Let $\ia, \ib \in \flag$, and $m\in\mathds N$, then 
\[\sup_{X^\beth} \left\{\frac{1}{m}\varphi_\ib - \varphi_\ia\right\} = \max_{\Sigma_\ia} \left\{\frac{1}{m}\varphi_\ib - \varphi_\ia\right\} \] 
\end{lemma}
\begin{proof}
After replacing $\ib$ for $\ib^m$, we can suppose that $m=1$. 
Set
\begin{equation}\label{eq;max}
C\doteq \max_{\Sigma_\ia} \left\{\varphi_\ib - \varphi_\ia\right\}
\end{equation}

Let $\cX = \widetilde{\bl_{\ia}X\times\Pro}$, and let $\cX_0 = \sum b_i E_i$ be the decomposition into irreducible components of the central fiber, so that $\Sigma_\ia =\{v_{E_1}, \dotsc, v_{E_k}\}$.
Then, we observe that:
\begin{enumerate}
\item The ideal $\ia$ becomes invertible on $\cX$.
\item We can then ``subtract", and Equation~\ref{eq;max} reads
\begin{equation}\label{eq;positiveord}
\ord_{E_i}\left(\ib\cdot \mathcal O_{\cX}(D)\right)\ge 0, \quad \text{ for every } i = 1, \dotsc, k 
\end{equation}
for some divisor $D\in \VCar(\cX)$, that depends on the constant $C$.
\item The polar variety of $\ib\cdot \mathcal O_{\cX}(D)$ is contained in the central fiber $\cX_0$, hence Equation~\ref{eq;positiveord} implies that the polar variety is of codimension at least $2$.
\item Since $\cX$ is normal,  $\ib\cdot \mathcal O_{\cX}(D)\subseteq \mathcal O_{\cX}$ hence proving the result.
\end{enumerate}
\end{proof}

As a consequence of the previous lemma, like in \cite[Lemma~4.26]{BJ22trivval}, we conclude that:

\begin{prop}\label{lem:sigmapsi}
Let $\psi\in \PL(X^\beth)$, then there exists a finite subset $\Sigma(\psi)= \Sigma\subseteq \Xdiv$ such that for every $\varphi\in \PSH(\alpha)$ we have:
\[\sup_{X^\beth} (\varphi - \psi) = \max_{\Sigma} (\varphi-\psi). \]\qed
\end{prop}

\begin{corollary}\label{cor;contsup}
Let $\psi\in \Cz\left(X^\beth, \R\right)$, then the function
\begin{align*}
\PSH(\alpha)\to \R, \quad \varphi\mapsto \sup_{X^\beth} (\varphi - \psi)
\end{align*}
is continuous.
\end{corollary}
\begin{proof}
Follows from the previous lemma, together with the density of $\PL(X^\beth)$ in $\Cz(X^\beth)$.
\end{proof}

\begin{theorem}\label{cor:uniqueextension}
  Let $\varphi\in\PSH(\alpha)$, and $\psi\colon X^\beth\to\R\cup\{-\infty\}$ a usc function, we then have:
  \[\varphi\leq \psi \text{ on } \Xdiv \iff \varphi\leq\psi \text{ on } X^\beth.\]
  \end{theorem}
  \begin{proof}
  If $\psi\in\PL$ this is an easy consequence of Proposition~\ref{lem:sigmapsi}.
  
  Since every continuous function is a uniform limit of PL functions,  the same result holds if $\psi\in \Cz(X^\beth)$.
  
  Lastly, if $\psi$ is a decreasing limit of the net $(\psi_\lambda)_\lambda \in \Cz(X^\beth)$, we have that $\varphi\leq\psi\leq\psi_\lambda$ on $\Xdiv$, and hence by the previous case $\varphi\leq \psi_\lambda $, and finally this implies that $\varphi\leq \psi$.
  \end{proof}

\begin{remark}
  Darvas, Xia, and Zhang developed on \cite{DXZ23transcendental} a notion of transcendental non-Archimedean psh metrics.
  They use the formalism of Ross--Witt Nystrum of \emph{test curves} on the complex manifold $X$, and call a \emph{non-Archimedean psh metric}, a \emph{maximal} test curve.
  
  Their approach is more general to define non-Archimedean $\beta$-psh metrics for a transcendetal big class $\beta$.
  Bellow, in Section~\ref{rem:DXZ}, we compare the present approach with theirs.
  As we will see they coincide when $\beta$ is Kähler, and the metric is of finite energy.
\end{remark}

\subsection{Extending the energy pairing}
In this section we will extend the energy pairing to general psh functions.
Unlike section~\ref{sec;plenergy}, the synthetic approach of Boucksom-Jonsson does not cover this singular case, even though similar generalizations can be done.

Here we will follow closely Section 7 of  \cite{BJ22trivval}.

Let $\alpha_0, \dotsc, \alpha_n\in \Pos(X)$, and $\varphi_i\in \PSH(\alpha_i)$ for $i=0,\dotsc, n $, we define:
\begin{equation}
(\alpha_0, \varphi_0)\dotsb (\alpha_n, \varphi_n)\doteq \inf\left\{(\alpha_0, \psi_0)\dotsb(\alpha_n, \psi_n): \psi_i\in \Hdom(\alpha_i), \psi_i\ge \varphi_i\right\}.
\end{equation}

\begin{lemma}\label{lem;usc}
The energy pairing,
\begin{align*}
\prod_{i=0}^n \PSH(\alpha_i)&\to \R\cup\{-\infty\}\\
(\varphi_0, \dotsc, \varphi_n)&\mapsto (\alpha_0, \varphi_0)\dotsb(\alpha_n, \varphi_n),
\end{align*} 
is upper semi-continuous.
\end{lemma}

It is clear that the energy pairing is increasing in each variable.
Hence, together with the previous lemma, we conclude that the energy pairing is continuous along decreasing nets.

\begin{proof}[Proof of Lemma~\ref{lem;usc}]
The proof follows from Corollary~\ref{cor;contsup}. 
For more details see \cite[Theorem~7.1]{BJ22trivval}.

\end{proof}

Just like in the algebraic setting, Corollary 7.11 of \cite{BJ22trivval}, we have the following result: 

\begin{prop}\label{prop;lowerbound}
Let  $\alpha_0, \dotsc, \alpha_n \in \Pos(X)$, and for $i=0, \dotsc, n$ $\varphi_i\in \PSH(\alpha_i)$, with $\varphi_i\le 0$, then 
\begin{equation}
(\alpha_0, \varphi_0)\dotsb(\alpha_n, \varphi_n)\gtrsim t^{n^2}\min_i\left\{(\alpha_i, \varphi_i)^{n+1}\right\}
\end{equation}
for $t\in \R$ sufficiently large in order to satisfy $\alpha_i\le t\alpha_j$ for every $i, j\in \{0, \dotsc, n\}$.
\end{prop}
\begin{proof}
Theorem~1.18 of \cite{BJ23synthetic} gives the inequality for $\varphi_i\in \PL\cap\PSH(\alpha_i)$, by taking decreasing sequences we conclude.
\end{proof}

We now extend the Monge--Ampère energy functional to the class of $\alpha$-psh functions.
\begin{defi}
Let $\alpha\in \Pos(X)$, and $V_\alpha=\int_X \alpha^n$, we define the \emph{Monge--Ampère energy functional} to be 
\begin{align*}
\mathrm E_\alpha\colon \PSH(\alpha)&\to \R\cup \{-\infty\}\\
\varphi&\mapsto \frac{V^{-1}_\alpha}{n+1} (\alpha, \varphi)^{n+1}.
\end{align*}
We define the set of \emph{finite energy non-Archimedean potentials} to be the set 
\[\cE^1(\alpha)\doteq \left\{\varphi\in \PSH(\alpha): \mathrm E_\alpha(\varphi)>-\infty\right\}.\]
When is clear by context we may omit $\alpha$.

Moreover, 
\[\Eabs \doteq \bigcup_{\alpha\in \Pos(X)}\cE^1(\alpha).\]
\end{defi}

As a direct consequence of Proposition~\ref{prop;lowerbound}, we have:
\begin{prop}\label{prop;energyE1}
Let $\alpha_0, \dotsc, \alpha_n\in \Pos(X)$, and $\varphi_i\in \cE^1(\alpha_i)$, then 
\begin{equation}
(\alpha_0, \varphi_0)\dotsb (\alpha_n, \varphi_n)\in \R
\end{equation}
is finite.\qed
\end{prop}

\begin{remark}
This allow us to extend the $\mathrm J_\alpha$ and the $\mathrm I_\alpha$ functionals, defined in Section~\ref{sec;plenergy}, to $\cE^1(\alpha)$ by the formulas of Equations~\eqref{eq;Jenergy} and \eqref{eq;Ienergy} respectively.

Moreover, the quasi-triangular inequality, and quasi-symmetry of $\mathrm J_\alpha$ for $\PL\cap\PSH(\alpha)$ functions, pass through, taking decreasing limits, to $\cE^1(\alpha)$.
\end{remark}

We can also see that the pairing of Proposition~\ref{prop;energyE1} is additive on the forms, and hence can be extended by linearity to $H^{1,1}(X)$.

Indeed, fix $\alpha_1, \dotsc, \alpha_n\in \Pos(X)$ and $\varphi_i\in \cE^1(\alpha_i)$, for $i = 1, \dotsc, n$, finally denote $\Gamma\doteq (\alpha_1, \varphi_1)\dotsb (\alpha_n, \varphi_n)$, we then define:
\begin{defi}
Let $\beta\in H^{1,1}(X)$, $\varphi\in \Eabs$, and $\alpha_0, \tilde \alpha_0\in \Pos(X)$, such that \[\beta = \alpha_0 - \tilde\alpha_0.\] 
We define the energy pairing $(\beta, \varphi)\mathbin{\cdot}(\alpha_1, \varphi_1)\dotsb (\alpha_n, \varphi_n)$ by the formula:
\begin{equation}
(\beta, \varphi)\mathbin{\cdot}(\alpha_1, \varphi_1)\dotsb (\alpha_n, \varphi_n)\doteq (\alpha_0 + \alpha, \varphi)\mathbin{\cdot}\Gamma - (\tilde\alpha_0 +\alpha, 0)\mathbin{\cdot}\Gamma,
\end{equation}
where $\varphi\in \cE^1(\alpha)\subseteq \cE^1(\alpha+\alpha_0)\subseteq \Eabs$, for some $\alpha\in \Pos(X)$.

Similarly, we get the energy pairing defined on $\prod_{k=0}^n\left(H^{1,1}(X)\times \Eabs\right)$.
\end{defi}

In particular, we define another functional whose importance will become apparent in Section~\ref{sec;csck}, which will be a twisted version of the Monge--Ampère energy.

\begin{defi}
Let $\alpha\in \Pos(X)$, $V_\alpha = \int_X \alpha^n$, and $\beta\in H^{1,1}(X)$, we define the \emph{Monge--Ampère twisted energy} to be the functional
\begin{align*}
\mathrm E^\beta_\alpha\colon \cE^1(\alpha)&\to \R\\
\varphi&\mapsto V^{-1}_\alpha (0,\beta)\cdot(\alpha, \varphi)^{n}.
\end{align*}
\end{defi}

\begin{remark}\label{rem:mongeamperee1}
  We can also extend the Monge--Ampère operator to the set $\cE^1(\alpha)$.
   We associate to $\varphi \in \cE^1(\alpha)$ the probability measure $\MA_\alpha(\varphi)$ satisfying:
  \[\PL(X^\beth)\ni\psi\mapsto \int_{X^\beth}\psi\MA_\alpha(\varphi)\doteq \frac{1}{V_\alpha}(0, \psi)\cdot (\alpha, \varphi)^n.\]
\end{remark}

\section{From Complex to non-Archimedean Geometry}
\label{sec;complexandnon}
In this section $X$ will be a compact Kähler manifold, and we will fix a Kähler metric $\omega\in \K(X)$, and $\alpha\in H^{1,1}(X)$ its cohomology class. 
We also will only consider smooth test configurations dominating $X\times \Pro$, with snc central fiber.

\addtocontents{toc}{\SkipTocEntry}
\subsection*{Basic Kähler Geometry tools}

Let $\iu$ be a coherent ideal of $X$, and $\beta\in H^{1,1}(X)$, we say that $\iu\otimes\beta$ is \emph{nef}, if $-F+\beta$ is nef, for \(F\subseteq Y\) a log resolution of $\iu$, and $F$ the effective divisor induced by $\iu$.

\begin{prop}\label{prop;extension} 
Let $\iu$ be an ideal such that there exists $U\colon X\to\R\cup\{-\infty\}$ a $\omega$-psh function of analytic singularity type along $\iu$, then $\alpha\otimes \iu$ is nef. 
\end{prop}
\begin{proof}

Let $\mu\colon Y\to X$ be a log resolution of $\iu$, and $F\subseteq Y$ the effective divisor induced by $\iu$.
We have, by Siu's decomposition theorem, that:
\begin{equation}
	0\le \mu^*\omega + \mathrm{dd^c} (U\circ\mu) = \delta_F + T
\end{equation}
for $T$ a positive current of bounded potential, that is, \[T = -\eta_F + \mu^*\omega + \mathrm{dd^c} \psi\] for $\psi\in L^\infty$, and $\eta_F$ a smooth representative of $\chern_1(\mathcal O(F))$.
Hence \[\psi\in \PSH(-\eta_F + \mu^*\omega)\cap L^\infty.\]
By a classical result due to Demailly  $[-\eta_F +\mu^*\omega ]$ is nef.
\end{proof}

\begin{lemma}\label{lem;sumnef}
Let $D, E\subseteq X$ effective irreducible divisors and $\beta\in H^{1,1}(X)$, such that $\beta- D$ and $\beta - E$ admit a smooth representatives which are semi-positive.
Then $\beta\otimes \left\{\mathcal O(-D)+\mathcal O(-E)\right\}$ is nef.
\end{lemma}
\begin{proof}
Let $h_D$ ($h_E$ resp.) be a smooth metric on $\OX(D)$ ($\OX(E)$ resp.) such that the associated curvature $\theta_D$ ($\theta_E$ resp.) is a smooth form with $\eta - \theta_D$ ($\eta-\theta_E$ resp.) semi-positive, for $\eta$ a smooth representative of $\beta$.

Let $s_D$ be the canonic section of $\OX(D)$, and $s_E$ of $\OX(E)$, then $\psi_D\doteq\log \lvert s_D\rvert_{h_d}$ and $\psi_E\doteq \log\lvert s_E\rvert_{h_E}$ are such that: 
\begin{align*}
	\mathrm{dd^c}\psi_D + \theta_D = [D]\ge 0,\quad \text{ and } \quad \mathrm{dd^c} \psi_E +\theta_E=[E]\ge 0,
\end{align*} 
and thus $\theta_D$ ($\theta_E$ resp.)-psh functions. 

Now, since $\eta - \theta_D$ ($\eta-\theta_E$ resp.) is semi-positive,  both $\psi_D$ and $\psi_E$ are $\eta$-psh.
In particular, $\psi\doteq \max\{\psi_D, \psi_E\}$ is $\eta$-psh, and has the singularity type of ${\mathcal O(-D) + \mathcal O(-E)}$, which by Proposition~\ref{prop;extension} implies that $\beta\otimes \left\{\mathcal O(-D) +\mathcal O(-E)\right\}$ is nef.
\end{proof}

\subsection{Geodesic rays and non-Archimedean psh functions}
The goal of this section is to get the analogues of Theorem~6.2  and Theorem~6.6 from \cite{BBJ21YTD} in our transcendental setting.
These results are essential for the non-Archimedean approach for the YTD conjecture developed by Berman--Boucksom--Jonsson, of which \cite{Li22geodesic} and the present paper rely on.

Remember that we have fixed a Kähler form $\omega\in \mathcal K(X)$, and its cohomology class $\alpha = [\omega]$.

\subsubsection{Quick recall on geodesic rays}
In this section we will use the conventions of \cite{BBJ21YTD}.

We define a \emph{psh ray} as a map $U\colon \R_{\ge 0}\to \PSH(\omega)$ such that the associated $S^1$-invariant function,
\begin{equation}\label{eq;s1invariant}
U\colon X\times \mathds D^*\to \left[-\infty, +\infty\right[, \qquad U(x, \tau) \doteq U_{-\log \lvert \tau\rvert},
\end{equation}
is $p_1^*\omega$-psh, where $\mathds D$ denotes the open unit disk on $\C$.

Whenever a psh ray has image in $\cE^1(\omega)$, is continuous for the strong topology as a map \[U\colon \R_{\ge 0}\to \cE^1(\omega),\] and $t\mapsto\mathrm E_\omega(U_t)$ is affine, we say that $U$ is a \emph{psh geodesic ray}.

Moreover, a psh ray $U$ has \emph{linear growth}, if there exist $C, D>0$ such that:
\[ U_t\le C\, t +D.\]
Every psh geodesic has linear growth, cf. \cite[Proposition~4.1]{BBJ21YTD}.

\begin{remark}
  Darvas proves in \cite[Theorem~2]{Dar17weak} that psh geodesic rays are --a distinguished class of-- actual geodesic rays for the Darvas metric $d_1$, and in \cite{Dar15mabuchi} that for $U_0$ and $U_1$ finite energy potentials, there always exists a psh geodesic joining them.
\end{remark}

\medskip

We will study now the relationship between --Archimedean-- rays of functions on $X$, with non-Archimedean functions on $X^\beth$.
\begin{defi}
  A $S^1$-invariant function, $U\colon X\times\D^*\to \R\cup\{-\infty\}$, is $C^\infty$ (resp. $L^\infty$)-compatible with $D\in \VCar(\cX)$, for $\cX$ a $\mu$-dominating test configuration, if:
  \[U\circ\mu + \log\lvert f_D\rvert \text{ locally extends to a smooth (resp. bounded) function (across } \cX_0\text{),}\]
  for $f_D$ a local equation of $D$.
\end{defi}

Furthermore, if $U$ is a compatible (either smoothly, or boundedly) with the vertical divisor $D$ we write:
\begin{equation}\label{eq;unasmooth}
U^{\beth}\doteq\varphi_{D}.
\end{equation}

Using this new terminology, we adapt Proposition~\ref{prop;extension} to this language.

\begin{lemma}\label{lem:unapsh}
  Let $U$ be a $\omega$-psh ray $L^\infty$-compatible with a vertical divisor $D\in \VCar(\cX)$, then $U^\beth = \varphi_D$ is $\alpha$-psh.
\end{lemma}
\begin{proof}
  We first observe that we can suppose $D$ effective, otherwise consider \[\varphi_{D+c\cX_0} = \varphi_D +c\] for $c\gg0$, that is $\alpha$-psh iff $\varphi_D$ is.
  Let $\mu\colon \cX\to X\times\Pro$ be a morphism of test configurations, then by Siu's decomposition formula we have:
  \[0\le \omega_\cX + \ddc (U\circ \mu) = -\delta_D + T,\]
  for $T$ a positive current of bounded potential, that is 
  \[0\le T = \eta_D + \omega_\cX + \ddc \psi\]
  with $\psi\in L^\infty$, and $\eta$ a smooth representative of $c_1(\cO_\cX(D))$.
  
  Consider the irreducible decomposition $\cX_0 = \sum b_k E_k$, since $T$ is of bounded potential, we can restrict $T$ to a bounded positive current supported on $E_k$:
  \[0\le T|_{E_k}\] 
  And hence by a result of Demailly we have:
  \[[T|_{E_k}] = [T]|_{E_k} = (D +[\omega_{\cX}])|_{E_k} \text{ is nef.}\]
  Therefore $U^\beth$ is $\alpha$-psh.
\end{proof}

More generally, for any $U\colon \R_{\ge 0}\to \PSH(\omega)$ psh ray of linear growth there is an induced ``non-Archimedean'' map:
\[U^{\NA}\colon \Xdiv\to \R\]
given by the following procedure:
\begin{enumerate}
\item Let $E\subseteq \cX\xrightarrow{\mu}X\times \Pro$ be a prime vertical divisor.
\item Consider the function $V \doteq U\circ\mu\colon \mu^{-1}(X\times \mathds D^*)\to \left[-\infty, +\infty\right[$, where $U$ is like in Equation~\ref{eq;s1invariant}.
\item Define \[U^{\NA}(v_E) \doteq -\nu(V, E)\]
where $\nu$ denotes the generic Lelong number along $E$.
\end{enumerate}

The goal of the next result is to extend the above construction of $U^{\NA}$ to an $\alpha$-psh function on $X^\beth$, generalizing Lemma~\ref{lem:unapsh} for a more general singularity type.
This result is an analogue of Theorem~6.2 of \cite{BBJ21YTD}, the proof here follows the same strategy as in \cite{BBJ21YTD} but we use directly a regularization result of Demailly, \cite[Proposition~3.7]{Dem92regularization}, without passing by the Castelnuovo-Mumford criterion of global generation (remember that in the projective case $\alpha = c_1(L)$).

\begin{theorem}\label{thm:extension}
Let $U\colon\mathds R_{>0}\to \PSH(\omega)$ be a psh ray of linear growth, then 
\[U^{\NA}\colon \Xdiv \to\mathds R\]
extends to a $\alpha$-psh function
\[U^{\beth}\colon X^\beth\to \mathds R\cup\{-\infty\}.\]
\end{theorem}
\begin{remark}
By Theorem~\ref{cor:uniqueextension}  if such an extension exists it is unique.
\end{remark}

\begin{proof}[Proof of Theorem~\ref{thm:extension}]
We will show that there exists a sequence $\varphi_m\in \cH(\alpha)$ such that:
\begin{enumerate}
\item $(\varphi_m)_m$ is decreasing
\item $\varphi_m (v_E)\searrow U^{\NA}(v_E) $ for every $v_E\in \Xdivc$
\end{enumerate}
Hence $U^\beth(v) \doteq \lim \varphi_m(v)$ will be the desired function.

By \cite[Proposition~3.7]{Dem92regularization}, for every $\lambda>1$, there is a sequence of $S^1$-invariant functions:
\[V_m\colon X\times \D_{1-\delta}\subseteq X\times \mathds D\to\R\cup \{-\infty\},\]
that are $(\lambda\cdot\omega)$-psh and of analytic singularities of type $\J(mU)^{\frac{1}{m}}$, where $\D_{1-\delta}$ is the disk of radius $1-\delta$, and $\J(m U)$ denotes the multiplier ideal sheaf of $m U$.

Therefore, $U_m\doteq \max\{V_m, \log\lvert \tau\rvert\}$ is $(\lambda\cdot\omega)$-psh, and has analytic singularities of type $(\ia_m)^{\frac{1}{m}}$ for $\ia_m\doteq\cJ(mU)\cdot\cO_{X\times\Pro} + (t^m)$, a flag ideal.

Moreover, if we let $\mu_m\colon \cX_m\to\trive$ be the test configuration given by the normalized blow-up of $\trive$ along $\ia_m$, and $E_m\subseteq\cX_m$ be the exceptional divisor, 
then the function $U_m\circ\mu_m$:
\begin{enumerate}
  \item  is $\mu_m^*(\lambda\cdot\omega)$-psh;
  \item has divisorial singularities along $\frac{1}{m}E_m$.
\end{enumerate}
Hence, $U_m$ is a psh ray $L^\infty$-compatible with $ \frac{1}{m}E_m$, and by Lemma~\ref{lem:unapsh} \[\varphi_m \doteq U_m^\beth = \varphi_{\ia_m}\] is $(\lambda\cdot\alpha)$-psh. 

The item (ii) of Proposition~3.7 of \cite{Dem92regularization} gives us that the Lelong numbers of $V_m$ along divisors over the central fiber $X\times\{0\}$ approach the Lelong numbers of $U$ over the same divisors, in particular, the Lelong numbers of $U_m$ have the same property. 
Thus in non-Archimedean terms:  \[\varphi_m|_{\Xdiv}\to U^{\NA}.\] 
Moreover, by the subadditivity of multiplier ideals --like in \cite[Lemma~5.7]{BBJ21YTD}-- the sequence $(\varphi_{2^m})_m$ is decreasing.
Applying Lemma~\ref{lem:intersectionpsh} we conclude the proof.
\end{proof}


Now, we will prove a result in the converse direction of the above theorem.
For that we remember the following definition:
\begin{defi}
  A psh geodesic ray $U$ in $PSH(\omega)$ is \emph{maximal} if for every other geodesic ray $V$ with $U_0\ge V_0 $ and $U^{\beth}\ge V^{\beth}$ we have $U_s\ge V_s$ for every $s\in\R_{\ge0}$.
\end{defi}

Maximal geodesic rays are in correspondence with the non-Archimedean potentials of finite energy, as we will see in the next theorem, an analogue of \cite[Theorem~6.6]{BBJ21YTD} in the Kähler setting.

\begin{theorem}\label{thm:correspond}
  Let $\varphi\in \cE^1(\alpha)$ be a non-Archimedean potential, and $u\in\cE^1(\omega)$ a reference metric, then there exists a unique maximal geodesic ray $U\colon \left[0,+\infty\right[\to \cE^1(\omega)$ starting at $u$, such that: \[U^\beth = \varphi.\]
\end{theorem}

\begin{lemma}\label{lem:maximalhna}
  Let $\varphi \in \cH(\alpha)$, and $\cX$ a smooth dominating test configuration with $D\in \VCar(\cX)$ such that $\varphi = \varphi_D$, and 
  \[D+\alpha_\cX \text{ is Kähler relatively to } \Pro.\]
  Then,
  \begin{enumerate}[(i)]
    \item there exixts a psh ray, starting from $u\in \cH(\omega)$, which is $C^\infty$-compatible with $(\cX, D)$.
    \item The enevlope usc of rays like in ($i$), is a maximal psh geodesic and is $L^\infty$-compatible with $(\cX, D)$.
  \end{enumerate}
\end{lemma}
\begin{proof}
  For ($i$) see \cite[Lemma~4.4]{SD18}.
  For ($ii$) we observe that in the terminology of \cite[Proposition~2.7]{Ber16kpoly}, a positively curved metric $\phi$ on a test configuration $(\cX, \cL)$, $\cL = L_\cX + D$, it is a psh ray, being locally bounded it is equivalent to ours $L^\infty$-compatibility with $(\cX, D)$, and \[(\ddc\varphi)^{n+1} = 0\] is equivalent --under the positivity condition-- of being a psh geodesic.
  Translating to our language their proof follows with no change.
\end{proof}

\begin{remark}\label{rem:affine}
  Let $\varphi\in \cH(\alpha)$, and $U$ a maximal psh geodesic s.t. $U^\beth = \varphi$, like in the previous lemma.
  Then by \cite[Remark 4.11]{SD18} we have that:
  \[\mathrm E_\omega(U_t) = \mathrm E_\omega(U_0) + t\cdot \mathrm E_\alpha(\varphi).\]
\end{remark}

\begin{proof}[Proof of Theorem~\ref{thm:correspond}]
The proof is like the one \cite[Theorem~6.6]{BBJ21YTD}. 
It relies on the following observations:

\begin{itemize}
\item If $u\in \cH(\omega)$, and $\varphi\in \Hdom(\alpha)$, we apply Lemma~\ref{lem:maximalhna}, and get a maximal geodeisc ray connecting $u$ and $\varphi$.

\item Now, if $u\in \cE^1(\omega)$, and $\varphi \in \cE^1(\alpha)$, we can take ``smoothing" decreasing sequences $\varphi_m\in\Hdom$, and $u_m\in \cH(\omega)$.

\item Using the first point there exists maximal geodesic ray, $U_{m, s}$, uniting $u_m$ and $\varphi_m$. 

\item By maximality  $U_{m+1}\le U_m$, and hence there exists the limit $\lim U_m(x)$, which we will denote by $U(x)$.

\item Since the energy is affine on these maximal geodesic rays, by Remark~\ref{rem:affine}, and decreasing on decreasing sequences, we have a uniform lower bound on the energy of $U_{m, s}$. 
This implies that $U_s$ is of finite energy, and a psh geodesic ray.

\item Moreover, from $U\le U_m$ we get $U^\beth\le \varphi_m$, and thus $U^\beth\le \varphi$, but also the previous estimate gives us that $\mathrm E_\alpha(U^\beth) = \mathrm E_\alpha(\varphi)$, and therefore $U^\beth = \varphi$. 

\item Finally, if $V$ is a psh ray of linear growth, such that $V_0\le u\le u_m$, and $V^\beth\le U^\beth = \varphi\le \varphi_m$, by maximality of $U_m$ we have $V\le U_m$, and thus $V\le U$, getting maximality of $U$ and concluding the proof.
\end{itemize}
\end{proof}

\subsection{Comparison with Darvas--Xia--Zhang non-Archimedean metrics}\label{rem:DXZ}
On \cite{DXZ23transcendental} and \cite{xia24operations} the authors develop a theory of non-Archimedean plurisubharmonic functions attached to a compact Kähler manifold.

Their approach is to define a non-Archimedean $\alpha$-psh function as a \emph{test curve} on the manifold $X$.
Test curves are the Ross--Witt Nyström transforms of maximal geodesic rays.

Following the strategy of \cite[Theorem~3.17]{DXZ23transcendental}, together with Theorem~\ref{thm:correspond}, for $\beta\in \Pos(X)$ a Kähler class, one can gets a correspondence from their $\beta$-psh functions of finite energy to ours associating to every $\cI$-maximal test curve, $\psi_\tau$, the ``beth" of its Ross--Witt Nyström transform $(\check \psi_t)$:
\[\cR^1_{\cI}(\theta)\ni({\psi}_\tau) \mapsto (\check{\psi})^{\beth}\in \cE^1(\beta)\]
for a smooth Kähler representative $\theta$ of $\beta$.

\begin{remark}
As mentioned before, their theory remains more general since they can consider the case when $\beta$ is big.
On the other hand, our theory is a direct analogue of the algebraic setting, which for instance enables us to associate a Monge--Ampère measure to a $\beta$-psh function.
\end{remark}

The comparison of $\beta$-psh functions --without the energy assumption-- is more delicate even in the algebraic case, and we refer to \cite[Theorem~3.14]{DXZ23transcendental} for more details.

\subsection{Asymptotics for the mixed energy}

The goal of this section is to prove Theorem~\ref{thm;kempf-ness}, which will be of central importance to relate the variational cscK problem with non-Archimedean geometry.
We begin with a useful lemma.
\begin{lemma}\label{lem:babyslope}
  Let $\varphi \in \cE^1(\alpha)$ be a non-Archimedean potential, and $U_s$ the associated maximal geodesic ray.
  Then \[\frac{V^{-1}_\omega}{s}\int_X U_s\, \omega^n\overset{s\to \infty}\longrightarrow \varphi(v_{\trivial}) = \int_{X^\beth}\varphi \MA_\alpha(0).\]
\end{lemma}
\begin{proof}
  Observe that, since $0\le \sup U_s - V^{-1}_\omega \int_X U_s\, \omega^n\le T_\omega$, see Lemma~\ref{lem;energyestimate}, we are left to prove that $\frac{\sup U_s}{s} \to \varphi(v_{\trivial}) = \sup \varphi$. 
  Let $\varphi_m\in \cH(\alpha)$ be a decreasing sequence converging to $\varphi$, and $U_s^m$ the associated maximal geodesic rays.
  Let's assume for simlplicity that $U_0 = 0 = U^m_0$, by \cite[Proposition 1.10]{BBJ21YTD} we have that: 
  \[\sup U_s = \ell \cdot s, \quad \sup U_s^m = \ell_m \cdot s,\] 
  for some real number $\ell\in \R$.
  
  By Theorem B of \cite{SD18}, it follows that $\ell_m = \sup \varphi_m = \varphi_m(v_{\trivial})$, hence \[\ell =\frac{\sup U_s}{s} \swarrow \frac{\sup U_s^m}{s} = \varphi_m(v_{\trivial})\searrow \varphi(v_{\trivial}),\]
  concluding the proof.
\end{proof}

We will now recall Theorem~3.6 from \cite{BJ23synthetic} that will be useful later.
Here we will state (and use) only the complex analytic version of the theorem.

\begin{lemma}
\label{lem;energyestimate}
Let $\eta_0, \dotsc, \eta_n$ be smooth closed (1,1)-forms, and for $i = 0, \dotsc n$ consider $U_i, V_i\in \cE^1(\omega)$ normalized for $\int U_i\omega^n =0 = \int V_i\omega^n$, then
\[
\lvert (\eta_0, U_0)\dotsb (\eta_n, U_n) - (\eta_0, V_0)\dotsb(\eta_n, V_n) \rvert \lesssim A\left(\max_{i}\mathrm J_\omega(U_i, V_i)^q\mathbin{\cdot} \max_i \left\{\mathrm J_\omega(U_i) + T_\omega\right\}^{1-q}\right),
\]
for 
\begin{align*}
q\doteq 2^{-n}, \quad A\doteq V_\alpha\prod_i(1+2\lVert \eta_i\rVert_{\omega}), \quad \text{and}\quad T_\omega \doteq\sup_{f\in\PSH(\omega)\cap C^\infty}\left\{\sup f - V_\alpha^{-1}\int f\,\omega^n\right\}.
\end{align*}
\end{lemma}
\begin{remark}

  The quantity $T_\omega$ is well known to be finite, see for instance \cite[Theorem~1.26]{BJ23synthetic}.

\end{remark}
\begin{proof}[Proof of Lemma~\ref{lem;energyestimate}]
Whenever $U^i, V^i$ are smooth functions the result follows from \cite[Theorem~3.6]{BJ23synthetic}.

In the general case, it suffices to take decreasing sequences of smooth potentials converging to $U^i$ and $V^i$ respectively, and to observe that the bound proved for smooth functions is uniform.
Thus, since the energy pairing is continuous along decreasing sequences, taking limits on both sides of the inequality we conclude.

\end{proof}

The following statement is a generalization to singular metrics of \cite[Theorem~B]{SD18}, and of \cite[Theorem 4.15]{DR17kstability} -that after a small modification can be adapted to general pairings-, and to more general functionals of \cite[Theorem~4.1]{Li22geodesic} --where they only consider the twisted Monge--Ampère energy estimate. 
It is the key ingredient to relate the non-Archimedean pluripotential theory with the complex analytic one, which will be essential to prove Theorem~\ref{thm;ytd}.

For the next theorem, let $\eta_0, \dotsc, \eta_k$ be real smooth closed $(1,1)$-forms.

\begin{theorem}[Slope Formula]\label{thm;kempf-ness}
Let $\varphi_0, \dotsc, \varphi_k\in \PL$, and $\varphi_{k+1}, \dotsc, \varphi_n\in \cE^1(\alpha)$.
Denoting by $U_i$ a smooth ray $C^\infty$-compatible with $\varphi_i$ if $i\le k$, and a  maximal geodesic ray compatible with $\varphi_i$ if $i>k$, we have:
\[\frac{1}{s}(\eta_0, U_{0,s})\cdots (\eta_n, U_{n,s})\overset{s\to\infty}{\longrightarrow} ([\eta_0], \varphi_0)\cdots([\eta_n], \varphi_n).\] 
\end{theorem}

\begin{proof}
When $k =n$ the result corresponds to \cite[Theorem B]{SD18}.
  We restrict ourselves to the case when $k = n-1$ the general case, when $k\le n-1$, will be similar.

Moreover, we observe that we can suppose that $\eta_n $ is a Kähler form,  otherwise
\[(\eta_n , U_{n,s})\cdot \Gamma_s = (k\omega , U_{n,s})\cdot \Gamma_s - (k\omega - \eta_n , 0)\cdot \Gamma_s\]
where $\Gamma_s \doteq (\omega_0, U_{0,s})\cdots (\omega_{n-1}, U_{n-1,s})$ and $k\gg 0$. 
Therefore, denoting $[\Gamma]\doteq ([\eta_0], \varphi_0)\cdots ([\eta_{n-1}], \varphi_{n-1})$, and applying the result for $\eta_n = \omega$,  we have:
\begin{align*}
\frac{1}{s}(\eta_n, U_{n,s})\cdot\Gamma_s &\to (\alpha, \varphi_n)\cdot[\Gamma] - ([\omega - \eta], 0)\cdot [\Gamma] \\
&= ([\eta_n], \varphi_n)\cdot[\Gamma],
\end{align*}
which, by symmetry, implies the result.

Now, let's prove the result.
Let $\varphi_n^m\searrow \varphi_n$ be a decreasing sequence of functions in $\Hdom(\alpha)$, and $U_{n}^m$ the associated maximal geodesic ray.

By \cite[Theorem~B]{SD18}:   
\begin{equation}
\frac{1}{s}(\eta_0, U_{0,s})\dotsb (\eta_{n-1}, U_{n-1,s})\cdot(\omega, U_{n,s}^m)\overset{s\to \infty}{\longrightarrow}([\eta_0], \varphi_0)\dotsb ([\eta_{n-1}], \varphi_{n-1}) (\alpha, \varphi_n^m),
\end{equation}
and as $m\to\infty$ the right hand side converges to $([\eta_0], \varphi_0)\dotsb (\alpha, \varphi_n)$.
Thus, to complete the proof we need to check that
\[\lim_{m\to\infty}\lim_{s\to\infty} \frac{1}{s}(\eta_0, U_{0, s})\dotsb (\omega, U_{n,s}^m) = \lim_{s\to\infty} \frac{1}{s}(\eta_0, U_{0,s})\dotsb  (\omega, U_{n,s}).\]

What we will do next is then to study the difference:
\begin{align*}
  (\star)_s&\doteq   \lvert \left(\eta_0,U_{0,s} \right)\dotsb (\omega, U_{n,s} ) - \left(\eta_0,U_{0, s}\right)\dotsb \left(\omega, U^m_{n,s}\right)\rvert\\
 &=   \lvert \left(\eta_0,U_{0,s} \right)\dotsb (0, U_{n,s} - U^m_{n,s} ) \rvert.
\end{align*}
Denoting by $a_{i,s}$ and $a^m_{n,s}$ the averages \(\int_X U_{i,s}\,\omega^n\) and \(\int_X U^m_{i,s}\,\omega^n\) respectively, and by $V_{i,s}$ and $V_{n,s}^m$ the normalized potentials  \[U_{i, s}-a_{i,s}, \quad \quad U^m_{n,s}-a_{n,s}^m,\]
we observe that:
\[(\star)_s = \lvert \left(\eta_0,V_{0,s} \right)\dotsb (\eta_{n-1}, V_{n-1, s})\cdot(0, U_{n,s} - U^m_{n,s} ) \rvert\]
and by the triangle inequality it follows that:
\begin{align*}
  (\star)_s&\leq  \lvert \left(\eta_0,V_{0,s} \right)\dotsb (0, V_{n,s} - V^m_{n,s} ) \rvert + \lvert \left(\eta_0,V_{0,s} \right)\dotsb (0, a^m_{n,s} - a_{n,s} ) \rvert\\
  &= \lvert \left(\eta_0,V_{0,s} \right)\dotsb (0, V_{n,s} - V^m_{n,s} ) \rvert + a^m_{n,s} - a_{n,s}.
\end{align*}
By the Lemma~\ref{lem:babyslope} we have that taking the slope at infinity and letting $m$ tend to infinity the second term vanishes.
Next we will focus our attention on the first term.

Note that we can suppose that, for $i\le n-1$, $\varphi_i$ is in $\Hdom(\alpha)$ and $U_{i,s}$ is a smooth psh ray $C^\infty$-compatible with $\varphi_i$. 
If it is not the case we write $\varphi_i = \psi_i^\prime - \psi_i^\dprime$ the difference of $\Hdom(\alpha)$ functions, then we consider $U_i^\prime$ and $U_i^\dprime$ psh-rays that are smoothly compatible with $\psi_i^\prime$ and $\psi_i^\dprime$ respectively, and the difference
\[U_{i,s}\doteq U^\prime_s-U^\dprime_s\]
will be smoothly compatible with $\varphi_i$, and the result follows from linearity of the pairing. 

Consequently,  it follows, by Lemma~\ref{lem;energyestimate}, that, for $q = 2^{-n}$:
\begin{align*}
  \lvert \left(\eta_0,V_{0,s} \right)\dotsb \left(\omega, V_{n,s} \right) - \left(\eta_0,V_{0, s}\right)\dotsb \left(\omega, V^m_{n,s}\right)\rvert &\lesssim \mathrm J_\omega (V_{n,s}, V_{n,s}^m)^q\cdot \max_i \left\{\mathrm J_\omega(V_{i,s}) + T_\omega\right\}^{1-q}\\
  &= \mathrm J_\omega (U_{n,s}, U_{n,s}^m)^{q}\cdot\max_i \left\{\mathrm J_\omega(U_{i,s}) + T_\omega\right\}^{1-q}\\
  &\lesssim d_1(U_{n,s}, U_{n,s}^m)^q \cdot (s+T_\omega)^{1-q},
\end{align*}
where the equality follows from the constant invariance of the J functional, and the last inequality by linear growth of $U_{i}$, since it implies that $\mathrm J(U_{i,s})\lesssim s$.

Moreover, we observe that:

By maximality $U_{n,s}\le U^m_{n,s}$, and
\begin{align*}
  d_1(U_{n,s}, U^m_{n,s}) &= \mathrm E_\omega(U^m_{n,s}) - \mathrm E_\omega(U_{n,s})\\
  &=\left(\mathrm E_\omega(U^m_{n,1}) - \mathrm E_\omega(U^m_{n,0})\right)s - \left(\mathrm E_\omega(U_{n,1}) - \mathrm E_\omega(U_{n,0})\right)s + C_m\\
  &=\left(\mathrm E_\omega(U^m_{n,1}) - \mathrm E_\omega(U_{n,1})\right)s - \left(\mathrm E_\omega(U^m_{n,0}) - \mathrm E_\omega(U_{n,0})\right)s + C_m,
\end{align*}
for $C_m\doteq \mathrm E_\omega(U^m_{n,0}) - \mathrm E_\omega(U_{n,0})$.

Therefore, taking the slope at infinity, and letting $m$ tend to infinity we have the desired result.
\end{proof}

We have already seen the non-Archimedean version of some classical functionals arising from pluripotential theory.
We recall their archimedean --original-- version.  
If $\omega\in \K(X)$ is a Kähler form, and $\eta$ is any closed $(1,1)$-form we have for $u\in \cE^1(\omega)$:
\begin{align*}
\mathrm E_\omega(u) &\doteq \frac{1}{n+1} V_\alpha^{-1}(\omega,u)^{n+1}\\
\mathrm E_\omega^\eta(u) &\doteq V_\alpha^{-1}(\eta, 0)\mathbin{\cdot} (\omega, u)^n\\
\mathrm J_\omega(u) &\doteq V_\alpha^{-1}(\omega, u)\mathbin{\cdot} (\omega, 0)^n - \mathrm E_\omega(u).
\end{align*}

By Theorem~\ref{thm;kempf-ness}, we can relate the above functionals with their non-Archimedean counterpart. 
If $\varphi\in \cE^1(\alpha)$, and $U$ the associated maximal geodesic ray, we have:
\[
\lim_{s\to\infty}\frac{\mathrm E_\omega(U_s)}{s} = \mathrm E_{\alpha}(\varphi), \quad \lim_{s\to\infty}\frac{\mathrm E^\eta_\omega(U_s)}{s} = \mathrm E_{\alpha}^{\beta}(\varphi), \text{ and } \quad \lim_{s\to\infty}\frac{\mathrm J_\omega(U_s)}{s} = \mathrm J_{\alpha}(\varphi),
\]
for $\beta=[\eta]$.

\section{CscK metrics and the Yau--Tian--Donaldson conjecture}\label{sec;csck}
In this section we will generalize a result by Chi Li, on the existence of cscK metrics.
Let $(X,\omega)$ again be a compact Kähler manifold, and $\alpha = [\omega]$ the cohomology class of $\omega$, $\eta\doteq -\Ric(\omega)$ minus the Ricci form of $\omega$, $\zeta$ its cohomology class, and $\underline s$ the cohomological constant $n\cdot\frac{[\Ric(\omega)]\cdot[\omega]^{n-1}}{[\omega]^{n}}$.

\subsection{The variational approach to the cscK problem}

The cscK equation is the Euler--Langrange equation for the \emph{Mabuchi functional}:
\begin{equation}\label{eq:chentian}
  \mathrm M_\omega = \underline{s}\,\mathrm E_\omega + \mathrm E_\omega^{\eta} + \mathrm H_\mu,
\end{equation}
where $\mu$ is the probability measure associated to $\omega^n$, and $\mathrm H_\mu(u)$ is the entropy of the Monge--Ampère measure of $u$ with respect to $\mu$, see \eqref{eq:entropydef}.

By the work of Chen--Cheng, \cite{CC21csck1, CC21csck2}, there exist a unique cscK metric in $\alpha$ if, and only if, the Mabuchi functional is coercive, that is:
\[\mathrm M_\omega \ge \delta \,\mathrm J_\omega - C\]
for some $\delta, C>0$.

What we do next is to define the non-Archimedean counterpart of the Mabuchi energy, $\mathrm M_\alpha$, and prove that the coercivity of $\mathrm M_\omega$ follows from the --non-Archimedean-- \emph{coercivity over $\cE^1$} of $\mathrm M_\alpha$.

Before studying the non-Archimedean version of the entropy functional, we recall a Legendre transform formula for the Archimedean entropy, if $u\in \cE^1(\omega)$ we have
\begin{equation}\label{eq:entropydef}
\mathrm H_{\mu}(u)= \sup_{f\in \Cz(X)}\left\{\int_X f\,\MA_\omega(u) - \log \int_X \exp (f) \,\mathrm{d}\mu \right\}.
\end{equation}

\subsection{Non-Archimedean entropy and the non-Archimedean Mabuchi functional}
\begin{defi}
Let $\mathrm H_\alpha\colon \cE^1(\alpha)\to \R$ be defined as follows
\begin{equation}
\mathrm H_\alpha(\varphi)\doteq  \int_{X^\beth} A_X \MA_\alpha(\varphi)
\end{equation}
where $A_X\colon X^\beth\to [0,+\infty] $ is the log discrepancy function on $X^\beth$.
We call $\, \mathrm H_\alpha$ the \emph{non-Archimedean entropy functional}.

Moreover, in analogy to the Chen--Tian formula of Equation~\eqref{eq:chentian}, we define the \emph{non-Archimedean Mabuchi functional}, $\mathrm M_\alpha\colon \cE^1(\alpha)\to \R$, as:
\[\mathrm M_\alpha \doteq \underline{s}\,\mathrm E_\alpha + \mathrm E_\alpha^{\zeta} + \mathrm H_\alpha.\]
\end{defi}
Let $\cX$ be a snc test configuration and $\varphi\in\cE^1(\alpha)$,  we denote by $H_{\alpha}^\cX(\varphi)$ the  integral:
 \[\int_{X^\beth} (A_X\circ p_\cX) \MA_\alpha(\varphi),\] with $p_\cX$ just like in section~\ref{sec;dual}.

\begin{prop}\label{prop;entropyineq}
Let $\psi\in \cE^1(\alpha)$, consider $V_s\in \cE^1(\omega)$ the maximal geodesic ray associated, then we have 
\begin{equation}
  \mathrm H_{\alpha}(\psi)\le\lim_{s\to+\infty} \frac{\mathrm H_{\mu}(V_s)}{s}.
\end{equation}

\end{prop}
\begin{proof}
Let $\psi\in \cE^1(\alpha)$, and consider $\cX$ a snc test configuration.
As seen before  $A_X\circ p_{\cX}$ is a PL function, let's denote it $\varphi$. 
We can write $A_\alpha^\cX$ in terms of $\varphi$:
\begin{align*}
H_\alpha^{\cX}(\psi) =\int_{X^\beth} (A\circ p_\cX) \MA_\alpha(\psi) &= \int_{X^\beth} \varphi \MA_\alpha(\psi)\\
&=(0,\varphi)\mathbin{\cdot} (\alpha, \psi)^n.
\end{align*}
Then, by 
Theorem~\ref{thm;kempf-ness}:
\begin{equation} 
  (0,\varphi)\mathbin{\cdot} (\alpha, \psi)^n = \lim_{s\to+\infty} \frac{1}{s}(0,U_s)\mathbin{\cdot}(\alpha, V_s)^n,
\end{equation}
for $U$ a smoothly compatible ray with $\varphi$.

On the other hand, for $f = U_s$:
\begin{align*}
\frac{1}{s}H_{\mu}(V_s)&\ge\frac{1}{s} \left\{\int_X f\omega_{V_s}^n - \log \int_X \exp (f) \,\mathrm{d}\mu \right\}\\
&= \frac{1}{s}(0,U_s)\mathbin{\cdot} (\omega, V_s)^n -\frac{1}{s}\log\int_X\exp(U_s)\,\mathrm{d}\mu \longrightarrow (0,\varphi)\mathbin{\cdot} (\alpha, \psi)^n - 0= H_\alpha^\cX(\psi)
\end{align*} 
 where in the limit we make use of Lemma~3.11 of \cite{BHJ19}, to get:
  \[\log\int_X \exp(U_s)\,\mathrm{d}\mu = O\left(\log (s)\right).\]

Therefore, 
 \[\mathrm H_\alpha(\psi)=\sup_{\cX}H_\alpha^{\cX}(\psi)\leq \lim_{s\to+\infty} \frac{1}{s} \mathrm H_{\mu}(V_s),\] 
concluding the proof.
\end{proof}

\begin{corollary}\label{cor;inequality}
Let $\varphi\in \cE^1(\alpha)$, and $U_s\in \cE^1(\omega)$ the maximal geodesic ray associated.
Then,
\begin{equation}
\mathrm M_{\alpha}(\varphi)\leq \lim_{s\to +\infty} \frac{\mathrm M_\omega(U_s)}{s}.
\end{equation}
\end{corollary}
\begin{proof}
Follows from Theorem~\ref{thm;kempf-ness} together with Proposition~\ref{prop;entropyineq}.
\end{proof}

\subsection{Main theorem}

\begin{prop}[Theorem~1.2 from \cite{Li22geodesic}]\label{prop;maximal}
Let $U_s\in \cE^1$ be a geodesic ray such that the slope 
\[ \lim_{s\to +\infty} \frac{\mathrm M_\omega(U_s)}{s}<+\infty,\] then $U$ is maximal.
\end{prop}
\begin{proof}
The proof goes without change as in the projective setting.

It is based on a local integrability result for the exponential of a difference of psh functions in the same singularity class, and a clever use of Jensen's inequality.
For more details see \cite[Theorem 1.2]{Li22geodesic}.
\end{proof}

\begin{defi}
Let $X$ be a compact Kähler manifold, and $\alpha\in \Pos(X)$ a Kähler class, then  $(X,\alpha)$ is \emph{uniformly K-stable over $\cE^1$}  if there exists $\delta>0$ such that:
\begin{equation}
\mathrm M_\alpha(\varphi)\ge \delta \mathrm J_\alpha(\varphi),\quad \text{ for every }\varphi\in \cE^1(\alpha).
\end{equation}
\end{defi}

Now, we will prove Theorem~\ref{thma;ytd}, the main theorem of this paper.

\begin{theorem}[Theorem~\ref{thma;ytd}]\label{thm;ytd}
  Let $(X,\alpha)$ be a compact Kähler manifold that is uniformly $K$-stable over $\cE^1$. 
  Then, $\alpha$ contains a unique cscK metric.  
\end{theorem}
\begin{proof}
By \cite{CC21csck2}  the existence, and uniqueness of a cscK metric is equivalent to the coercivity of the Mabuchi functional $\mathrm M_\omega$, i.e. the existence of $C,\delta>0$ such that \[\mathrm M_\omega \geq \delta \mathrm J_\omega - C.\]
We will proceed by contradiction.

Suppose that $\mathrm M_\omega$ is not coercive, then, by \cite{BBJ21YTD, Li22geodesic, CC21csck2}, we can find a geodesic ray emanating from $0$, $U_s\in \cE^1(\omega)$, normalized so that $\sup U_s = 0$, such that:
\[\lim_{s\to +\infty} \frac{1}{s}\mathrm M_\omega(U_s)\le 0.\]
 
 By Proposition~\ref{prop;maximal},  $U$ is maximal, therefore it is associated to a non-Archimedean potential  $\varphi\in \cE^1(\alpha)$.
 
 Corollary~\ref{cor;inequality} gives:  \[0\ge\lim_{s\to +\infty} \frac{1}{s}\mathrm M_\omega(U_s)\ge \mathrm M_\alpha(\varphi),\]
 but since $(X,\alpha)$ is uniformly $K$-stable over $\cE^1$,  there exists a $\delta>0$ such  that:
 \[\mathrm M_\alpha(\varphi)\ge \delta \mathrm J_\alpha(\varphi) >0,\]
 yielding a contradiction. 
\end{proof}

\medskip

\appendix
\section{Semi-rings and tropical algebras}
\label{apx;semi-rings}

\addtocontents{toc}{\SkipTocEntry}
\subsection{Semi-rings}
\begin{defi}
A triple $(S, +, \cdot)$ is a commutative \emph{semi-ring} if the following conditions hold:
\begin{itemize}
\item $(S, +)$ is a commutative monoid, with identity element denoted by $0_S$\footnote{A monoid is a semi-group with an identity element.};
\item $(S, \cdot)$ is a commutative semi-group;
\item For every $a, b, c\in S$  \[a\cdot (b+c) = (a\cdot b) + (a\cdot c), \quad \text{ and }\quad 0_S\cdot a =0. \]
\end{itemize}

A \emph{morphism of semi-rings} is a function \[\phi\colon (S, +, \cdot)\to (R, +, \cdot)\] mapping $0_S$ to $0_R$ that satisfies  
\[\phi(a+b) = \phi (a) + \phi(b), \quad \text{and} \quad \phi(a\cdot b) = \phi(a)\cdot\phi(b)\]
Whenever $S$ and $R$ have multiplicative identities, $1_S$ and $1_R$ respectively, we ask
\[\quad \phi(1_S) = 1_R.\]
We denote the set of morphisms from $(S, +, \cdot)$ to $(R, +, \cdot)$ by \(\hom(S, R)\)
\end{defi}
Now some examples
\begin{example}\label{ex;semiring}
\begin{enumerate}
\item Let $S\doteq \R\cup\{+\infty\}$, considered with $\min$ as the sum, and the usual sum, $+$, as the semi-ring multiplication is a semi-ring.
Here $S$ has a multiplicative identity given by: \[0_S = +\infty, \quad \text{and} \quad 1_S= 0 \] 
Equivalently, $S = (\R\cup\{-\infty\}, \max, +)$ is  isomorphic to $(\R\cup\{+\infty\}, \min, +)$.
\item The subset $(\oR, \min, +)$ is also a semi-ring. 
\item Let $X$ be a topological space,  $(\Cz(X, \R)\cup\{-\infty\}, \max, +)$ is a semi-ring.
\item Let $A$ be a commutative ring, and denote by $\ideal(A)$ the set of ideals of finite type of $A$, then together with the usual sum and multiplication of ideals,  $\ideal(A)$ is a semi-ring with neutral elements given by:
\[0_S = \{0\}, \quad \text{and} \quad 1_S= A\]
\end{enumerate}
\end{example}

A semi-ring $S$ comes equipped with a natural order relation, we say 
\[a\le b, \quad \text{ if }\quad a=b+a. \] 
Whenever $S$ is unital we denote by $S_+$ the set:
\[S_+\doteq\left\{a\in S \mid a\ge 1_S\right\}\]
\begin{lemma}
Let $S$ be a semi-ring with multiplicative unit, then $S_+$ inherits a semi-ring structure restricting the  operations.
\end{lemma}
\begin{proof}
Let $a, b\in S_+$, then $ 1_S=a+1_S$ and $1_S = b+1_S$, hence $a+b+1_S = a+ (b+1_S) = a+1_S = 1_S$.
Moreover
\[a\cdot b +1_S=a\cdot b + 1_S+b = (a+1_S)\cdot b +1_S=1_S\cdot b + 1_S = b+1_S = 1_S \]
\end{proof}

\begin{example}
Using the notation of Example~\ref{ex;semiring} we have:
\begin{enumerate}
\item $\ideal(A)_+ = \ideal(A)$, since for every $\iu\in \ideal(A)$ we have \[\iu + A = A\] 
\item If $S= \Cz(X, \R)\cup\{+\infty\}$, then \[S_+ = \Cz(X, \R_{\ge 0})\cup\{+\infty\}\]
\end{enumerate}
\end{example}

\begin{defi}
A semi-ring $(S, +, \cdot)$ is \emph{idempotent} if for every $a\in S$  \[a+a =a\]
\end{defi}

\begin{remark}
Every semi-ring of Example~\ref{ex;semiring} is idempotent.
\end{remark}

The order relation for idempotent semi-rings reads slightly more general,  $a\le b$ if, and only if, we can decompose 
\[a = b+c\] for some $c\in S$.

\addtocontents{toc}{\SkipTocEntry}
\subsection{Tropical spectrum and restrictions}
Let $S$ be a semi-ring, we recall that the \emph{tropical spectrum of $S$} is the set \[\tspec S\doteq\hom(S, \R\cup\{+\infty\})\] with the pointwise convergence topology.

\begin{lemma}
Let $S$ be a unital semi-ring,  the restriction induces a map
\[\tspec S\to \hom(S_+, \oR)\]
\end{lemma}
\begin{proof}
Indeed, let $\chi\in \tspec$, and $f\in S_+$, we then have that $1_S + f = 1_S$, and hence \[0= \chi(1_S)=\chi (f+1_S)  = \min\{\chi(f), \chi(1_S)\}\le \chi(f)\]
\end{proof}
\begin{corollary}
Let $S$ be a semi-ring such that $S=S_+$, then 
\[\tspec S = \hom(S, \oR)\]
\qed
\end{corollary}

\begin{defi}
We define a \emph{$\R$-tropical algebra}, $\mathcal A$, as a $\R$-vector space together with an operation $\{\cdot, \cdot\}$ such that 
$S =(\mathcal A\cup\{\infty\}, \{\cdot, \cdot \}, +)$ is a semi-ring, with \[0_S = \infty, \quad \text{ and } \quad 1_S = 0\]
satisfying 
\[0\le f\implies f\le \lambda f\]
for $\lambda\ge 1$ a real number.
\end{defi}

\begin{remark}\label{rem;indempotenttropicalalgebra}
Tropical algebras are unital, and admit multiplicative inverses.

Moreover, if $\mathcal A\cup\{\infty\}$ is idempotent then every element of $\mathcal A$ can be written as a difference of elements of $\mathcal A_+$.
Indeed, if $f\in \mathcal A$ we can write \[f= -\{-f, 0\} - (-\{f, 0\})\]
and we have 
\begin{equation}
\{0, -f\} = \left\{0, \{0, -f\}\right\}\implies 0 = \left\{-\{0, -f\}, 0\right\}
\end{equation}
which implies $0\le -\{0, -f\}$ and therefore $-\{0, -f\}\in \mathcal A_+$, we proceed similarly for $-\{f, 0\}$.
\end{remark}

\begin{example}
The two main examples are:
\begin{itemize}
\item $\R$ is a tropical algebra.
\item $\Cz(K, \R)$ is a tropical algebra.
\end{itemize}
\end{example}

\begin{lemma}\label{lem;tspecalgebra}
Let $\mathcal A$ be an idempotent tropical algebra, then 
\[\tspec (\mathcal A\cup \{\infty\}) = \left\{\varphi\in \mathcal A^* \mid \varphi(\{f, g\}) = \max\{\varphi(f), \varphi(g)\}\right\}, \]
where $\mathcal A^*$ denotes the algebraic dual.
\end{lemma}
\begin{proof}
For a max commuting linear functional $\varphi\in \mathcal A^*$, we can define $\varphi(\infty)$ as $+\infty$ and  $\varphi\in \tspec(\mathcal A\cup\{\infty\})$.

On the other hand, if $\chi\in \tspec(\mathcal A\cup\{\infty\})$, we observe that taking $f\in \mathcal A$ we have
\[0_\R = \chi(0_{\mathcal A}) =\chi(f + (-f)) = \chi(f) + \chi(-f)\]
getting that $\chi$ is finite  and $\Q$-linear on $\mathcal A$.

We are left to prove that $\chi$ is $\R$-linear.
Indeed if $\lambda\in \R_{>0}$, $p_n\in \Q_{>0}$ an increasing sequence, $q_n\in \Q_{>0}$ a decreasing sequence, both converging to  $\lambda$, and $f\in \mathcal A_+$, we then have:
\[0\le p_n f\le p_{n+1}f\le \lambda f\le q_{n+1}f\le q_n f,
\]
thus getting 
\[p_n\chi(f)\le\chi(\lambda f)\le q_n\chi(f),
\]
taking the limit we get:
\[\chi(\lambda f) = \lambda \chi(f).\]

Applying Remark~\ref{rem;indempotenttropicalalgebra} and the $\Q$-linearity we get the desired result.
\end{proof}

\addtocontents{toc}{\SkipTocEntry}
\subsection{PL spaces}
Let $K$ be a compact Hausdorff topological space.
\begin{defi}
  A \emph{PL structure on $K$} is a $\Q$-linear subspace of the set of continuous functions, $\PL(K)\subseteq\Cz(K,\R)$, such that:
  \begin{itemize}
    \item It separates points;
    \item It contains all the $\Q$-constants;
    \item It is stable by $\max$.
  \end{itemize}
  We refer to the pair $(K,\PL(K))$ as a \emph{PL space}.

  A map $f\colon K_1\to K_2$ is a \emph{morphism of PL spaces} if it is continuous and \[f^*\colon \Cz(K_2, \R)\to \Cz(K_1, \R)\] maps $\PL(K_2)$ to $\PL(K_1)$.
  Moreover, it is an \emph{ismorphism of PL structures} if the induced map $f^*\colon \PL(K_2)\to \PL(K_1)$ is bijective.
\end{defi}

\begin{remark}
  If $\PL(K)$ is a PL structure on $K$, then $(\PL, \max, +)$ is an idempotent semiring,  $\PL_\R(K)\doteq \PL(K)\otimes\R$ is a subtropical algebra of $\Cz(K,\R)$, and by Proposition~\ref{prop:tropicalhomeo}
  \[K\simeq \left(\tspec \PL_\R(K)\right)\setminus\{0\}/\R_{>0}.\]
  In particular, an isomorphism of PL structures $f\colon K_1\to K_2$ is also a homeomorphism.
\end{remark}

\section{Monomial valuations and Lelong--Kiselman numbers}
\label{apx;monomial}

\addtocontents{toc}{\SkipTocEntry}
\subsection{The Lelong--Kiselman number} 
Let $\Delta\subseteq \C^n$ be the unit polydisk centered at $0$, $T\doteq (S^1)^n$ the compact torus, and $w\in \R_{\ge 0}^n$. 
Consider $\varphi\colon \Delta^*\to \R_{\le 0}$ a psh function.
\begin{defi}
For $w\in (\R_{\ge 0})^n$ the \emph{Lelong-Kiselman number of $\varphi$ at $0$ with weight $w$} is given by
\[\nu_w(\varphi, 0)\doteq \sup\left\{ \delta>0\bigm\vert\exists\, U \text{ open neighborhood of } 0, \varphi(z)\leq \delta \max_{w_i\neq 0} \frac{\log\lvert z_i\rvert}{w_i},\,\, \forall z\in U  \right\}\]
More generally, given a complex manifold $X$, $p\in X$ and $\psi\colon X\to \R_{\le 0}\cup\{-\infty\}$ a quasi-psh function on $X$ we can define $\nu_w(\psi, p)$ similarly.

\end{defi}

We will show now that if $p\in \cap_{w_i\neq 0}\{z_i=0\}\subseteq \Delta$, then $\nu_w(\varphi, 0) = \nu_w(\varphi, p)$.
To do that, we'll see that it is locally independent of $p$. More precisely,  we'll show that if there exists a open neighborhood $U$ of $0$, such that $\varphi(z)\leq \nu_w(\varphi, 0)\min_{w_i\ne 0} \frac{\log\lvert z_i\rvert}{w_i}$ for every $z\in U$, then  the inequality \[\varphi(z)\leq \nu_w(\varphi, 0)\max_{w_i\ne 0} \frac{\log\lvert z_i\rvert}{w_i}\]
holds for every $z\in \Delta$.

For that, define
\begin{align*}
\tilde \varphi \colon \Delta^*&\to\R_{\le 0}\\
z&\mapsto \sup_{\xi\in T}\varphi(\xi\mathbin{\cdot} z)
\end{align*}
By the maximum principle  $\tilde \varphi(z) = \sup_{\lvert \alpha_i\rvert\le 1}\varphi(\alpha_1 z_1, \dotsc, \alpha_n z_n)$.

Now, since $\tilde \varphi$ is $T$-invariant,  there exists a convex function $\chi\colon (\R_{\ge 0})^n\to \R_{\le 0}$, such that 
\begin{equation}
\tilde\varphi(z) = \chi\left(-\log\lvert z_1\rvert, \dotsc, -\log\lvert z_n\rvert\right)
\end{equation} 
where $\chi$ is decreasing in each variable, in particular it is also decreasing on rays, i.e. for every $w\in (\R_{\ge 0})^n$ the map $t\mapsto \chi(t\mathbin{\cdot} w)$ is decreasing.

Now, since every bounded above decreasing convex function has a finite slope at infinity we can define:
\begin{equation}
\chi^\prime_\infty(w)\doteq \lim_{t\to+\infty}\frac{\chi(tw)}{t}.
\end{equation}
It is clear that $\chi^\prime_\infty(w) = -\nu_w(\varphi, 0)$.

We have, for $t>0$, $\frac{\chi(tw)-\chi(0)}{t}\leq \chi^\prime_\infty(w)$,  hence
\begin{equation}
\chi(tw)\le t\mathbin{\cdot} \chi^\prime_\infty(w) + \chi(0)\le t\mathbin{\cdot} \chi^\prime_\infty(w)
\end{equation}
For $t=-\max_{w_i\ne 0} \frac{\log\lvert z_i\rvert}{w_i}= \min_{w_i\ne 0} \frac{-\log \lvert z_i\rvert}{w_i}$ gives us the following inequality:
\begin{equation}
\begin{aligned}
\varphi(z)&\le\tilde\varphi(z)=\chi(-\log\lvert z_1\rvert, \dotsc, -\log\lvert z_n\rvert)\le\\
\chi(tw)&\le \chi^\prime_\infty (w)\min_{w_i\ne 0} \frac{-\log\lvert z_i\rvert}{w_i}\\
&\   =\nu_w(\varphi, 0)\max_{w_i\ne 0} \frac{\log\lvert z_i\rvert}{w_i}
\end{aligned}
\end{equation}
for every $z\in \Delta$.

\addtocontents{toc}{\SkipTocEntry}
\subsection{An analytic interpretation of monomial valuations}

Now, let $B = \sum_{i\in I}B_i$ be a reduced snc divisor, $Z$ a connected component of the intersection, and $p\in Z\subseteq X$.

Let $\iu\in \ideal_X$, and $f=\log\lvert \iu\rvert$ a quasi-psh function $f\colon X\to\R\cup\{-\infty\}$, with singularities along $\iu$.

Then  the monomial valuation defined by Equation~\eqref{eq;monomial} satisfies:
\begin{equation}
v_{w, Z, p}(\iu) = \nu_w(f, p).
\end{equation}
In particular, by the discussion the previous section, the right hand side does not depend on locally on $p\in Z$.
Hence  the left hand side also does not depend locally on $p\in Z$, therefore retrieving Proposition~\ref{prop;independencep}.

\section{Basic linear algebra of bilinear forms}

\begin{lemma}\label{lem;bilinear}
Let $V$ be a real finite dimensional vector space, and $B\colon V\times V\to\R$ a symmetric bilinear form, such that
\begin{enumerate}
\item There exists a base $e_i\in V$ such that $B(e_i, e_j)\ge 0$ for $i$ different than $j$.
\item There exist an element of the kernel, $v=\sum_i v^i e_i\in V$, that is for every $i$: \[B(v,e_i) =0,\] such that $v^j> 0$ for every $j$.
\end{enumerate}
Then $B$ is negative semi-definite.
\end{lemma}
\begin{proof}
After changing bases we can suppose that $v = \sum_i e_i$, and then the second item reduces to \[ B(e_i, e_i) = - \sum_{j\neq i} B(e_j, e_i)\]
and thus taking $x = \sum_j x^j e_j \in V$,  
\begin{align*}
B(x,x) &= \sum_i(x^i)^2 B(e_i, e_i) + \sum_{i\ne j}x^i x^jB(e_i, e_j)\\
&= \sum_{i\ne j}\big(x^i x^j - (x^i)^2\big) B(e_i, e_j)\\
&= \sum_{i\ne j}\big(x^i x^j - (x^j)^2\big) B(e_i, e_j)
\end{align*}
where the last equality is given by symmetry.
Thus, by changing the roles of $i$ and $j$,
\begin{align*}
2 B(x,x) = -\sum_{i\ne j}\big(x^i- x^j\big)^2 B(e_i, e_j)\le 0
\end{align*}
we get the desired result.
\end{proof}

\section{A synthetic comment}\label{apx:synthetic}
In this paper we assume the synthetic pluripotential theory developed in \cite{BJ23synthetic}, for $X^\beth$, where $X$ is a compact Kähler manifold.

In fact,  
\begin{itemize}
\item The set $X^\beth$ is underlying the compact Hausdorff topological space.
\item The ``smooth" test functions $\mathcal D$ are the set $\PL_\R(X^\beth)\simeq \varinjlim_{\cX} \VCar_\R(\cX)$, which is dense in $\Cz(X^\beth, \R)$ by Proposition~\ref{prop:pldense}.
\item The  vector space $\mathcal Z$ in our case corresponds to $\varinjlim_{\cX}H^{1,1}(\cX/\Pro)$.
\item The $\ddc\colon \mathcal D\to \mathcal Z$ operator assigns:
\[\PL_\R\ni \varphi_D\mapsto [\mathrm c_1\left(\mathcal O_{\cX}(D)\right)]\]
\item For $\beta\in \mathcal Z$  \[\beta\ge 0,\] if $\beta_{\cX}\in \Nef(\cX/\Pro)$ for some determination $\beta_{\cX}$.
\item The \emph{dimension} of $X^\beth$ is defined to be $\dim X$.
\item The assignment $\mathcal Z^n \to \Cz(X^\beth)^\vee$, $(\beta_1, \dotsc,\beta_n )\mapsto \beta_1\wedge\cdots\wedge \beta_n $, is given by
\[\Cz(X^\beth)\ni f\mapsto \inf\left\{(0, \varphi)\cdot(0,\beta_1)\cdots (0, \beta_n)\mid \varphi\ge f\right\}, \]
that satisfies all the required properties by all the results on Section~\ref{sec;plenergy}, in particular the version of Zariski's Lemma of Lemma~\ref{lem:zariskilemma} gives the seminegativity of 
\[\cD\times\cD\ni(\varphi, \psi)\mapsto \int_{X^\beth} \varphi \, \ddc \psi\wedge \beta_1\wedge \cdots\wedge \beta_n,\]
for $\beta_i\ge 0$.
\end{itemize}


\bibliography{pietro}{}
\bibliographystyle{alpha}

\end{document}